\documentclass{amsart}

\newlength\fontsizelength

\usepackage{ifthen}
\newif\ifdraftmode
\draftmodefalse

\newif\ifarxivversion
\arxivversiontrue

\usepackage{newfile} 
\IfFileExists{\jobname.baaux}{\input{\jobname.baaux}}{}

\newcommand{\numberappendices}{\textbf{???}}

\IfFileExists{\jobname.baaux}{%
\ifthenelse{\equal{\arabicnumberappendices}{2}}
           {\renewcommand{\numberappendices}{two}}{}%
\ifthenelse{\equal{\arabicnumberappendices}{3}}
           {\renewcommand{\numberappendices}{three}}{}%
\ifthenelse{\equal{\arabicnumberappendices}{4}}
           {\renewcommand{\numberappendices}{four}}{}%
\ifthenelse{\equal{\arabicnumberappendices}{5}}
           {\renewcommand{\numberappendices}{five}}{}%
}
  
\newoutputstream{baaux}
\openoutputfile{\jobname.baaux}{baaux}

\usepackage{times,enumerate}

\ifdraftmode\usepackage{showkeys, booktabs}\fi
  
\usepackage{latexsym,amssymb,amscd,amsthm,amsmath,verbatim}
\usepackage[all]{xy}
\usepackage{setspace}
\usepackage{tikz}
\usepackage{hyperref}
\usepackage{relsize,nccmath} 
\usepackage{accents} 
\usepackage{mathtools} 
\usepackage{xcolor}
\usepackage{xr} 
\usepackage[all]{xy}
\usepackage{tensor} 

\usepackage[normalem]{ulem}

\newcommand{\maC}{\mathcal C}

\newcommand{\maE}{\mathcal E}
\newcommand{\maF}{\mathcal F}

\newcommand{\maH}{\mathcal H}

\newcommand{\maL}{\mathcal L}

\newcommand{\maP}{\mathcal P}
\newcommand{\maQ}{\mathcal Q}

\newcommand{\maS}{\mathcal S}

\newcommand{\maV}{\mathcal V}
\newcommand{\maW}{\mathcal W}

\newcommand{\maY}{\mathcal Y}

\newcommand\bl[1]{ \{ #1 \} }
\newcommand\XGV{X_{\mathrm{GV}}}
\newcommand\betaGV{\beta_{\mathrm{GV}}}
\newcommand\maWb{\maW_{\mathrm{bdry}}}
\newcommand\maWint{\maW_{\mathrm{int}}}
\newcommand\maWeucl{\maW_{\eucl}}

\newcommand\geucl{g_{\mathrm{eucl}}}
\newcommand\eucl{\mathrm{eucl}}
\newcommand\argu{\mathchoice{\,\raisebox{.13\fontsizelength}{\scalebox{.5}{$\bullet$}}\,}{\raisebox{.13\fontsizelength}{\scalebox{.5}{\,$\bullet$\,}}}{\,\raisebox{.07\fontsizelength}{\scalebox{.4}{$\bullet$}}\,}{\,\raisebox{.06\fontsizelength}{\scalebox{.35}{$\bullet$}}\,}}
\newcommand\twoargu{\argu{,}\argu} 

\newcommand\refschrSpace[1]{\ref*{schrSpace-#1}}
\newcommand\citeschrSpace[2]{\cite[#1~\refschrSpace{#2}]{AMN1}}

\definecolor{lightgreen}{rgb}{.9,1,.9}

\definecolor{darkgreen}{rgb}{.2,.6,.2}
\definecolor{darkblue}{rgb}{.2,.2,.8}
\definecolor{darksharpblue}{rgb}{.4,.2,.8}

\newcommand\red[1]{\textcolor{red}{#1}}

\newcommand\magenta[1]{\textcolor{magenta}{#1}}
\newcommand\orange[1]{\textcolor{orange}{#1}}
\newcommand\hidden[1]{} 

\ifdraftmode

\newcommand\cbernd[1]{\footnote{\magenta{#1}}}
\newcommand\bernd[1]{\magenta{#1}}
\newcommand\victor[1]{\red{#1}}
\newcommand\jeremy[1]{\orange{#1}}
\newcommand\showdiscussion[1]{#1}
 \renewcommand\sout{\bgroup\markoverwith{\textcolor{red}{\rule[0.7ex]{3pt}{1.4pt}}}\ULon}
\else

\newcommand\cbernd[1]{}
\newcommand\victor[1]{}
\newcommand\bernd[1]{}
\newcommand\jeremy[1]{}
\newcommand\showdiscussion[1]{}
\renewcommand\sout[1]{}
\fi

\DeclareMathOperator{\Lie}{Lie}
\DeclareMathOperator{\Aff}{Aff}
\DeclareMathOperator{\aff}{aff}

\newcommand{\hookdownarrow}{\mathrel{\rotatebox[origin=c]{-90}{$\hookrightarrow$}}}

\newcommand\ovra\overrightarrow
 
\newcommand\GL{\mathop{\mathrm{GL}}} 
 
\newcommand\spann{\mathop{\mathrm{span}}}

\newcommand\eg{e.\kern2pt g.,\ \ignorespaces}
\newcommand\ie{i.\kern2pt e.\ \ignorespaces}

\newcommand\medcup{\mathsmaller{\bigcup}}
\newcommand\unionma[1]{\mathchoice{\medcup\,\mathcal{#1}}{\medmath{\textstyle \bigcup}\,\mathcal{#1}}{\medmath{\scriptstyle\bigcup}\mathcal{#1}}{\medmath{\scriptscriptstyle\bigcup}\mathcal{#1}}}
\newcommand\unionx[1]{\mathchoice{\medcup\,{#1}}{\medmath{\textstyle \bigcup}\,{#1}}{\medmath{\scriptstyle\bigcup}{#1}}{\medmath{\scriptscriptstyle\bigcup}{#1}}}
\newcommand\unionF{\unionma{F}}
\newcommand\unionP{\unionma{P}}
\newcommand\unionS{\unionma{S}}

\newtheorem{theorem}{Theorem}[section]
\newtheorem{lemma}[theorem]{Lemma}
\newtheorem{proposition}[theorem]{Proposition}
\newtheorem{corollary}[theorem]{Corollary}

\theoremstyle{definition} 
                          
  \newtheorem{definition}[theorem]{Definition}
                          
\newtheorem{remark}[theorem]{Remark}

\newtheorem{example}[theorem]{Example}

\newcommand{\psubmanifold}{p-sub\-ma\-ni\-fold}
\newcommand{\psubmanifolds}{p-sub\-ma\-ni\-folds}
\let\psbmanifold\psubmanifold
\let\psbmanifolds\psubmanifolds

\newcommand{\wsbmanifold}{weak submanifold}
\newcommand{\wsbmanifolds}{weak submanifolds}

\newcommand{\dist}{\operatorname{dist}}
\newcommand{\id}{\operatorname{\mathrm{id}}}

\newcommand{\pa}{\partial}

\renewcommand{\SS}{\mathbb{S}}
\newcommand{\CC}{\mathbb C}

\newcommand{\NN}{\mathbb N}

\newcommand{\RR}{\mathbb R}

\newcommand{\ZZ}{\mathbb Z}

\newcommand{\CI}{{\mathcal C}^{\infty}}
\newcommand{\CIc}{{\mathcal C}^{\infty}_{\text{c}}}
\newcommand\<{\langle}
\renewcommand\>{\rangle} 
\newcommand\ede{\ \coloneqq\ }
\newcommand\seq{\ =\ }

\newcommand\interior[1]{#1_0}

\let\ol\overline
\let\wihat\widehat

\newcommand{\oX}{\overline{X}}
\newcommand{\oY}{\overline{Y}}
\newcommand{\oXY}{\overline{X/Y}}

\newcommand\wha[1]{\widehat{#1}}
\newcommand{\doublesharp}[1]{\prescript{\lower1.95pt\hbox{\reflectbox{$\scriptstyle\sharp$}\kern-.8pt}}{}{#1^\sharp}}

\let\setminus\smallsetminus

\newcommand{\Diff}{\operatorname{Diff}}

\newcommand\dvol{\operatorname{dvol}\nolimits}

\newlength{\keyheightlength}
\newcommand{\stelle}[1]{\mathchoice%
  {\lower.27\keyheightlength\hbox{\big|}\lower.57\keyheightlength\hbox{$\scriptstyle #1$}}%
  {\lower.24\keyheightlength\hbox{\big|}\lower.48\keyheightlength\hbox{$\scriptstyle #1$}}%
  {\lower.21\keyheightlength\hbox{|}\lower.42\keyheightlength\hbox{$\scriptscriptstyle #1$}}%
  {\lower.18\keyheightlength\hbox{|}\lower.39\keyheightlength\hbox{$\scriptscriptstyle #1$}}%
}
\newcommand{\Stelle}[1]{\mathchoice%
  {\lower.39\keyheightlength\hbox{\Big|}\lower.74\keyheightlength\hbox{$\scriptstyle #1$}}%
  {\lower.33\keyheightlength\hbox{\big|}\lower.66\keyheightlength\hbox{$\scriptstyle #1$}}%
  {\lower.30\keyheightlength\hbox{\big|}\lower.60\keyheightlength\hbox{$\scriptscriptstyle #1$}}%
  {\lower.25\keyheightlength\hbox{\big|}\lower.51\keyheightlength\hbox{$\scriptscriptstyle #1$}}%
}
\newcommand{\sstelle}[1]{\mathchoice%
  {\lower.39\keyheightlength\hbox{\bigg|}\lower.98\keyheightlength\hbox{$\scriptstyle #1$}}%
  {\lower.37\keyheightlength\hbox{\Big|}\lower.74\keyheightlength\hbox{$\scriptstyle #1$}}%
  {\lower.33\keyheightlength\hbox{\big|}\lower.75\keyheightlength\hbox{$\scriptstyle #1$}}%
  {\lower.29\keyheightlength\hbox{\big|}\lower.63\keyheightlength\hbox{$\scriptstyle #1$}}%
}
\newcommand{\SStelle}[1]{\mathchoice%
  {\lower.40\keyheightlength\hbox{\Bigg|}\lower1.35\keyheightlength\hbox{$\scriptstyle #1$}}%
  {\lower.37\keyheightlength\hbox{\bigg|}\lower.98\keyheightlength\hbox{$\scriptstyle #1$}}%
  {\lower.35\keyheightlength\hbox{\Big|}\lower.90\keyheightlength\hbox{$\scriptscriptstyle #1$}}%
  {\lower.30\keyheightlength\hbox{\big|}\lower.72\keyheightlength\hbox{$\scriptscriptstyle #1$}}%
}

\let\oldtocsection=\tocsection
\let\oldtocsubsection=\tocsubsection
\let\oldtocsubsubsection=\tocsubsubsection
\renewcommand{\tocsection}[2]{\hspace{0em}\oldtocsection{#1}{#2}}
\renewcommand{\tocsubsection}[2]{\hspace{2em}\oldtocsubsection{#1}{#2}}
\renewcommand{\tocsubsubsection}[2]{\hspace{3em}\oldtocsubsubsection{#1}{#2}}

\newcommand\doi[1]{\href{https://doi.org/#1}{DOI:~#1}}

\author[B. Ammann]{Bernd Ammann} \address{B. A., Fakult\"at f\"ur
  Mathematik, Universit\"at Regensburg, 93040 Regensburg, Germany}
\email{bernd.ammann@mathematik.uni-regensburg.de}

\author[J. Mougel]{J\'{e}r\'{e}my Mougel} \address{J. M.,
  Mathematisches Institut Georg-August-Universit\"at G\"ottingen,
  37083 G\"ottingen, Germany} \email{j.mougel@hotmail.com}

\author[V. Nistor]{Victor Nistor}\address{V. N., Universit\'{e} de
  Lorraine, CNRS, IECL, F-57000 Metz, France
and Inst. Math. Romanian Acad.  PO BOX 1-764, 014700 Bucharest
Romania} \email{victor.nistor@univ-lorraine.fr}

\thanks{B.A. has been partially supported by SPP 2026 (Geometry at
  infinity) and the SFB 1085 (Higher Invariants), both funded by the
  DFG (German Science Foundation). V.N. has been partially
  supported by ANR-14-CE25-0012-01 (SINGSTAR) funded by ANR (French
  Science Foundation).
\\
keywords: Schrödinger equation, regularity, eigenfunctions, N-body problem, Georgescu-Vasy compactification}

\date\today

\begin{document}
\title[Regularity for Schr\"odinger operators]{A regularity result for the 
bound states of $N$-body Schr\"odinger operators: Blow-ups and Lie manifolds}

\begin{abstract}
We prove regularity estimates in weighted Sobolev spaces for the 
$L^2$-eigen\-func\-tions of Schrö\-dinger type operators whose potentials have inverse 
square singularities and uniform radial limits at infinity. In particular, 
the usual $N$-body Hamiltonians with Coulomb-type singular potentials
are covered by our result: in that case, the weight 
is $\delta_{\maF}(x) \coloneqq  \min \{ d(x, \unionF), 1\}$, where
$d(x, \unionF)$ is the usual Euclidean distance to the union $\unionF$
of the set of collision planes $\maF$. The proof is based on blow-ups of manifolds 
with corners and Lie manifolds. More precisely, we start with the radial 
compactification~$\oX$ of the underlying space~$X$ and we first blow up the spheres 
$\SS_Y \subset \SS_X$ at infinity of the collision planes $Y \in \maF$
to obtain the Georgescu--Vasy compactification. Then, we blow up the collision 
planes~$\maF$. We carefully investigate how the Lie manifold structure 
and the associated data (metric, Sobolev spaces, differential operators)
change with each blow-up. Our method applies also to higher order differential
operators, to certain classes of pseudodifferential operators, and 
to matrices of scalar operators.
\end{abstract}

\maketitle
\tableofcontents

\section{Introduction}

We obtain regularity estimates for the eigenfunctions of a class of elliptic
differential operators with singular coefficients. Our results cover 
Schr\"o\-din\-ger-type operators whose potentials have inverse square type
singularities and uniform radial limits at infinity. Weaker singularities, 
such as Coulomb-type singularities, are included in our results; in particular,
the ``usual'' $N$-body Hamiltonians \cite{GeorgescuBookNew, 
DerezinskiAnnals, DerGer2} are covered by our results. Our results are obtained by 
combining and extending the results of our previous 
papers \cite{ACN} and \cite{AMN1}, which will be two basic references in what 
follows.  

\subsection{Basic notation and constructions}\label{subsec.basic-notions} 
Here is first a minimum of notation needed to state our results
for the case of Coulomb-type singularities,
Theorem~\ref{theorem.intro.reg} below. 
This notation was explained in more detail in \cite{AMN1}. Let us fix from 
now on a \emph{finite-dimensional,} Euclidean (real) vector space $X$. We shall 
let $\SS_X$ denote its \emph{sphere at infinity} (more precisely: we define 
$\SS_X$ as the set of rays in $X$ emanating form $0$) and, then, we let 
$\oX=X \sqcup \SS_X$, called the \emph{spherical compactification} of $X$, as described, for instance, in  \cite[Section~5.1]{AMN1}
(see also \cite{Kottke-Lin, KottkeMelrose2, VasyReg, VasySurv}). 
Note that here $\sqcup$ denotes disjoint union as sets, but $\oX$ also carries a 
differential structure that makes  $\oX$ diffeomorphic to a closed disk, see again \cite[Section~5.1]{AMN1}. In particular,
we have a diffeomorphism
\begin{equation}\label{eq.def.sphr.comp}
   \oX \smallsetminus \{0\} \simeq \SS_{X} \times (0, \infty]\,.
\end{equation}
Moreover, the Euclidean metric on $X$ defines a natural diffeomorphism from $\SS_X$ 
to the unit sphere in $X$. 

Let us assume throughout the paper that $\maF$ is a finite, 
non-empty set of linear subspaces of $X$. We need to define various
distance functions in terms of $\maF$. First, by $\dist(x, y) = \dist_{\eucl}(x, y)$,
we shall denote the Euclidean distance on $X$ and, for any subset $Z \subset X$, we let 
\begin{equation}\label{eq.def.dist}
   d_Z(x) \ede \dist(x, Z)\ede  \inf_{z\in Z}|x-z|
\end{equation}
denote the Euclidean distance from $x$ to $Z$.
Let then $\unionF \coloneqq  \bigcup_{Y \in \maF}Y$ 
denote the union of the elements of $\maF$ and
\begin{equation}\label{eq.def.dist.maF}
   \delta_{\maF}(x) \ede \min_{Y \in \maF} \{ d_Y(x), 1\} \seq 
   \min\{\dist(x, \unionF), 1\}\,.
\end{equation}
This function will be the weight used
in the definition of the weighted Sobolev spaces that appear (sometimes
only implicitly) in the statements of our results.

We shall assume throughout the paper, 
as in, for instance, \cite{BMBGeorgescu, DerezinskiAnnals}, that $\maF$ 
is stable under intersections. More precisely, let $2^A$ denote the set of all subsets 
of some set $A$. We shall say that $\maS \subset 2^A$ is a \emph{semilattice}
\footnote{In the literature, what we call a semilattice is often called a 
 \emph{meet-semilattice}.}
if $A,B\in \maS$ implies $A\cap B \in \maS$. Thus, from now on, we shall assume that 
$\maF$ is a \emph{finite semilattice} of linear subspaces of $X$ such that $X \notin \maF$
but $\{0\} \in \maF$.
(Mathematically the assumption that $X \notin \maF$ is not a very important assumption, 
but it allows us to simplify the statement of
Theorem~\ref{theorem.intro.reg}, for instance.)

If $Y$ is a linear subspace of $X$, its closure $\oY$ in $\oX$ coincides with the 
spherical compactification of $Y$, so there is no danger of confusion. As in \cite{AMN1}, 
to the semilattice $\maF$ we will associate the semilattices of \emph{spherical 
compactifications} and, respectively, \emph{spheres at infinity} of subspaces in $\maF$:
\begin{equation}\label{eq.def.semilattices}
    \overline{\maF} \ede \{\, \oY \, | \ 
    Y \in \maF \, \} 
 \ \ \mbox{ and } \ \
    \SS_{\maF} \ede 
    \{\, \mathbb{S}_Y \, | \ Y \in \maF \, \}\,. 
\end{equation}
Finally, we can now introduce the main spaces considered in our main results
as iterated blow-ups:
\begin{equation}\label{eq.blownup.spaces}
    \XGV \ede [\oX: \SS_\maF] \ \ \mbox{ and } \ \
    X_{\maF} \ede  [\oX : \SS_\maF \cup \overline{\maF}]\,.
\end{equation}
The notions of blow-up and iterated blow-up will be recalled in Section~\ref{sec2bis}. 
We will see in~\ref{sec.two.bups} that $X_{\maF} \simeq [\XGV: \wha{\maF} ]$, where 
$\wha\maF$ denotes the lift of $\overline\maF$ to $\XGV$, see 
Subsection~\ref{subsec.semilat.b-up} for details.

\subsection{Statement of the main result and comments} 
The following result for eigenfunctions is formulated, for simplicity, just for 
the usual Coulomb type singularities (see Theorems~\ref{theorem.regularity}
and~\ref{theorem.application} for more general versions). There is no loss of generality 
to assume $X = \RR^n$, and we shall do that when convenient, for instance, in the 
next theorem. Let $\pa_i := \frac{\pa}{\pa x^i}$ and
$\pa^\alpha := \pa_1^{\alpha_1} \pa_2^{\alpha_2} \ldots \pa_n^{\alpha_n}$,
as usual.
The Laplacian will always have the ``analytic'' sign 
convention, i.e.\ $\Delta=\sum_{i} \pa_i^2$.

\begin{theorem}\label{theorem.intro.reg} 
Let $\maF\not\kern1pt\ni X$ be a finite semilattice of linear subspaces of 
$X \coloneqq  \RR^n$, $\{0\} \in \maF$. 
Let $X_{\maF} \ede  [\oX : \SS_\maF \cup \overline{\maF}]$ and let $d_Y(x) 
\coloneqq \dist(x, Y)$ and $\delta_{\maF}(x) 
\coloneqq \min\{\dist(x, \unionF), 1\} $ be the distance functions 
introduced in the previous subsection. For every $Y \in \maF \cup \{X\}$,
let $a_Y \in \CI(X_\maF)$ and
\begin{equation}\label{eq.V.theo.intro.reg}
    V(x) \ede\sum_{Y \in \maF} a_Y(x) d_Y(x)^{-1} + a_X(x) \,.
\end{equation}
Let us assume that $u \in L^2(X)$ satisfies 
$(\Delta + V)u = \lambda u$ in distribution sense on $X \smallsetminus \unionF$, for 
some $\lambda \in \CC$. Then, for all multi-indices $\alpha \in \NN^n$, we have
\begin{equation*}
     \delta_{\maF}^{|\alpha|} \pa^\alpha u \in L^2(X)\,.
\end{equation*}
\end{theorem}

Define $\rho(x) \coloneqq  \dist_{\overline{g}}(x, \unionF)$, where $\overline g$ is
a (true) metric on $\XGV$ (that is, a metric on the compact
manifold with corners $\XGV$ that is smooth up to the boundary).
Then we can use $\rho$ instead of $\delta_\maF$, see Appendix~\ref{app.a}.
That is, for $u$ as in the theorem above and for all multi-indices $\alpha \in \NN^n$, we have
\begin{equation*}
     \rho^{|\alpha|} \pa^\alpha u \in L^2(X)\,.
\end{equation*}

Theorem \ref{theorem.intro.reg} is, as we have already mentioned, a particular 
case of Theorem~\ref{theorem.application} (which, in turn, follows from Theorem~\ref{theorem.regularity}). In particular, that theorem covers also the 
case of inverse square potentials, which have been studied recently in 
\cite{DerezinskiFaupin, DerezinskiRichard, Felli1, HLNU1, HLNU2, LiNi09} and 
in many other papers. We also consider higher order uniformly strongly elliptic 
operators. This result strengthens a similar regularity result in \cite{ACN}; compared to this previous result we will obtain better decay rates at infinity. 
Our results generalize immediately to systems. Regularity results for the 
eigenfunctions of Schr\"odinger operators have been obtained in many papers. 
See, for instance, \cite{ACN, Jecko2, Felli2, Felli3, Flad2, Flad3,Flad1, 
Fournais3, Fournais2, Fournais1, grieser:b, Siedentop, Jecko1, VasyAsympt, 
VasyResolvent, Yserentant1, Yserentant3} and the references therein. Many 
methods in this article may be generalized to non-linear equations, e.g.\ 
to semi-linear equations of the form
\begin{equation*}
  \Delta u + V |u|^s u=\lambda u
\end{equation*}
with $V$ as in \eqref{eq.V.theo.intro.reg}, $\lambda\in \CC$ and  
$0<s\leq 2/(n-2)$, see Remark~\ref{rem.nonlin} for details.

\subsection{Contents of the paper}
The paper is essentially self-contained and consists of two parts.
The first part contains mostly background material, but presented in
a novel way. It consists of Sections~\ref{sec2} and~\ref{sec2bis}.
In Section~\ref{sec2}, we review
manifolds with corners, their submanifolds, clean intersections of submanifolds, 
and some other basic concepts needed in the paper. We also recall in
this section some results from \cite{AMN1} and from some earlier papers, 
including \cite{ACN, Kottke-Lin, MelroseBook}.
In Section \ref{sec2bis} we recall and study the blow-ups and their iteration 
with respect to one or more suitable subsets.
In particular, we extend the definition of an \emph{admissible order} to
a $k$-tuple of closed subsets of $M$ and explain that the iterated blow-up 
with respect to a clean semilattice with an admissible order 
is defined and independent of the chosen admissible order \cite{ACN, AMN1}. 
We also explain that the action of a Lie group on $(M, \maS)$ extends to an action 
on the iterated blow-up $[M: \maS]$. The second part of the paper
begins with Section \ref{sec3}, where we study metric aspects 
of the iterated blow-ups. We start by introducing the concept of a ``smoothed distance 
function'' to a \psubmanifold{} $P \subset M$, which is a function that behaves
like the distance to $P$ close to $P$, but is smooth outside~$P$. We study then
the smoothed distance function to a blow-up suitable $k$-tuple, in general, and to
a clean semilattice, in particular. Then we recall how the metrics, the Sobolev 
spaces, and the differential operators change if one conformally changes the metric 
using a smoothed distance function. We apply these results to blow-ups (including
the iterated ones). We distinguish here the case of a blow-up along a submanifold 
contained in the boundary and the case of a manifold not contained in the boundary. 
The main technical result of this section is the behavior of the smoothed
distance functions when performing iterated blow-ups, Proposition \ref{prop.smoothness}.
We also recall here the regularity result for the natural elliptic operators on Lie manifolds 
\cite{sobolev}, which will then provide the regularity result of this paper. (The main
work of this paper is to position ourselves to be able to use the regularity
theorem for Lie manifolds.) In Section~\ref{sec4}, we introduce the relevant
semilattices and Lie manifolds needed to deal with the $N$-body problem and 
we particularize the constructions and results of Section~\ref{sec3} to this setting.
Section~\ref{sec5} explains how to obtain our regularity result, Theorem 
\ref{theorem.intro.reg} and some of its generalizations
from the results of the other sections. Finally, the paper contains
\ifdraftmode\red{\numberappendices}{}\else\numberappendices\fi{} 
appendices. In Appendix~\ref{app.a}, we prove the
equivalence of the functions $\rho$ and $\delta_{\maF}$ used in 
Theorem~\ref{theorem.intro.reg} and right after. (In fact, they are both equivalent 
to $\rho_\maF$, the smoothed distance function to $\maF$, which is the function that 
is actually used in the proofs, but is more difficult to define.)
%
\ifarxivversion
Appendices B, C and D contain  technical details used in the article. Appendix E summarizes some notation.
\fi

\textbf{Acknowledgment:} 
V.N. thanks Chris Kottke for some useful references.
The authors thank Daniel Grieser for useful discussions about iterated 
blow-ups and lifts of submanifolds.

\section{Preliminaries}\label{sec2}

We now recall a few concepts and results needed to understand
the main results of this paper. For the concepts
and results not recalled here, we refer to \cite{ACN} or \cite{AMN1},
on which this paper is heavily based.

\subsection{A few basic notations and definitions}\label{subsec.basics}
We need to complete the notation introduced in the Introduction as follows.
We let
\begin{equation}\label{eq.notation.corner}
    \RR_k^n \ede [0, \infty)^k \times \RR^{n-k}
    \mbox{ and } \ \SS_k^{n-1} \ede \SS^{n-1} \cap \RR_k^n\,,
\end{equation}
where $\mathbb{S}^{n-1}$ is the unit sphere of $\mathbb{R}^n$ for the Euclidean norm. 
The space $\SS_k^{n-1}$ will be called a  \emph{(generalized) orthant of the sphere}, 
for instance, $\SS_1^{n-1}$ is a half-sphere. It will be used
in order to help us understand the smooth structure on the spherical 
compactification $\oX$. More precisely, let us assume that $X = \RR^n$ and let us 
consider the bijective map
$\Theta_n : \oX \to \mathbb{S}^n_1$
\begin{equation}\label{eq.def.theta}
  \begin{gathered}  
\begin{cases}
\ \Theta_n(x) \ede \frac{1}{\sqrt{1+\|x\|^2}} (1,x) \in \mathbb{S}^n_1 
& \mbox{ if } x \in X \,, \\
\ \Theta_n(\RR_+v) \ede \frac{1}{\|v\|}(0,v) \in \mathbb{S}^n_1  
& \mbox{ if } \RR_+v \in \mathbb{S}_X\,.
\end{cases}
\end{gathered}
\end{equation}

and its inverse $\Theta^{-1}_n :\mathbb{S}^n_1 \to \oX$
\begin{equation}\label{eq.def.theta.inv}
\Theta^{-1}_n (y_0,y_1, \ldots, y_n) \longmapsto
     \begin{cases}\frac{1}{y_0}(y_1,\ldots ,y_n) \in \mathbb{R}^n & \text{if } y_0\neq 0\\
       \RR_{+} (y_1,\ldots,y_n) \in \SS_X & \text{if } y_0=0 
     \end{cases} \,.
\end{equation}

We define the smooth structure on $\oX$ by the requirement that the map 
$\Theta_n$ be a diffeormphism as in \cite{Kottke-Lin, MelroseBook, VasyReg}.
For general $X$, we define the smooth structure on $\oX$ using a linear isomorphism 
$X \simeq \RR^n$, for suitable $n$.

\begin{remark}\label{rem.emb}
Let us assume that $Y = \RR^p \subset \RR^n = X$. Then we have the following commutative
diagram of smooth embeddings
\begin{equation}\label{eq.emb}
\begin{matrix}
Y & \hookrightarrow &  \overline{Y} &  \stackrel{\Theta_p}{\longrightarrow}
& \mathbb{S}^p_1 \\
\hookdownarrow& & \hookdownarrow& &\hookdownarrow \\
X &  \hookrightarrow & \oX & \stackrel{\Theta_n}{\longrightarrow} & \mathbb{S}^n_1\,
\end{matrix},
\end{equation}
which is a basic functoriality property of the spherical compactification. In particular, this 
shows that the closure of $Y$ in $\oX$ can be identified as a differential manifold
with $\oY$, the spherical compactification of $Y$, so there is no danger of confusion.
\end{remark}

Let $|A|$ denote the number of elements of a set $A$.
We continue with several basic definitions.

\begin{definition}\label{def.boundary.depth}
Let $I \subset \{1, \ldots, n\}$ and
\begin{equation}\label{eq.p-manifold}
  L_{I,k}^n \ede \{\, x=(x_1,\ldots , x_n) \in \RR^n_k 
  \ede [0, \infty)^k \times \RR^{n-k}
  \mid x_i=0   \ \mbox{ if }\ i \in I\, \} \, .
\end{equation}
Then $b_{I,k}^n \coloneqq |I\cap \{1,\ldots,k\} |$ will be called the
\emph{boundary depth} of $L_{I,k}^n$ in $\RR_k^n$ and
$b_x \coloneqq  | \{i\in \{1,\ldots,k\} \mid x_i=0\}| $
will be called the \emph{boundary depth} 
of $x = (x_1,\ldots,x_n)$ in $\RR_k^n$.
Clearly, $c\coloneqq |I|$ is the \emph{codimension} of $L_I$ in $\RR_k^n$
and $d \coloneqq  n-c$ is its \emph{dimension}.
\end{definition}

Obviously $b_{I,k}^n\coloneqq \min \{ b_{x} \mid x \in L_{I,k}^n\}$.
The role of the sets $L_{I,k}^n$ is to serve as models for \psubmanifolds\
\cite[Definition~1.7.4]{MelroseBook}. A \psubmanifold\ is a suitable submanifold 
of a manifold with corners, a concept recalled next.  
The letter ``p'' comes from the fact that for such a \psubmanifold\ and any of
its points, a local chart can be found around that point such that
 the submanifold is a factor of a \emph{product}.

\subsection{Manifolds with corners and their submanifolds}
Let us now summarize a few basic definitions related to manifolds with corners 
\cite{aln1, JoyceCorners, Kottke-Lin, Kottke-gblow, KottkeMelrose2, nistorDesing}.
Recall that a \emph{manifold with corners} $M$ of dimension $n$ is a topological 
space locally modeled on $\RR^n_k \coloneqq  [0, \infty)^{k} \times \RR^{n-k}$ with 
smooth transition functions.  The spaces $\RR^n_k$ and $\SS^n_k$ introduced in
\eqref{eq.notation.corner} provide simple examples of manifolds with corners. 
Given a point~$x$ in a manifold with corners~$M$, the \emph{(boundary) depth} 
of $x$ in $M$ is defined locally using Definition~\ref{def.boundary.depth}.
More precisely, the boundary depth of~$x$ in~$M$ is the number of non-negative coordinate 
functions vanishing at~$p$ in any local coordinate chart at $p$. It is the least 
$k$ such that, for all $x \in M$, there exists a chart $\phi : U \to \RR_k^n$ 
defined on an open neighborhood of $x$ in~$M$. For further reference, let us
formalize the following concept in a definition.

\begin{definition}\label{def.bdry.depth}
The \emph{boundary depth} of $P \subset M$ 
is the minimum over all boundary depths of points $x\in P$.
\end{definition}

Let $(M)_k$ be the set of points of~$M$ 
of boundary depth $k$. Its connected components are called the \emph{open} boundary faces of 
codimension (or boundary depth) $k$ of~$M$. A \emph{boundary face} of boundary depth $k$
is the closure of an open boundary face of  boundary depth $k$. A boundary face $H$ of codimension
one of~$M$ will be called a \emph{boundary hyperface} of~$M$. The union of
the boundary hyperfaces~$H$ of~$M$ is denoted $\pa M$; it is 
the \emph{boundary} of $M$. It is the set of points of $M$ of
boundary depth $\ge 1$. In particular, a subset $A \subset M$ has boundary
depth 0 if, and only if, it is not contained in $\pa M$, the boundary of $M$.

The smooth functions on a manifold with corners $M$ are defined as on
their counterparts without corners. This is conveniently done by embedding $M$
into a smooth manifold without corners $\widetilde M$, called an \emph{enlargement}
of $M$. This enlargement can also be used to define the tangent spaces of $M$.

\begin{definition}\label{def.b.def.f}
Let $H$ be a hyperface of $M$ or a union of hyperfaces of $M$. 
A \emph{boundary defining function} of $H$ (in $M$) is a function $0 \le x_H \in
\CI(M)$ such that $H = x_H^{-1}(0)$ and $dx_H \neq 0$ on $H$.
\end{definition}

\begin{remark}
The example of the teardrop domain, see Figure~\ref{fig.teardrop}, shows that not 
all hypersurfaces have a boundary defining function. However, each (connected)
boundary face $F$ of codimension $k$ can \emph{locally} (\ie, in a sufficiently
small neighborhood~$U$ of any given point~$x$) be represented as 
\begin{equation*}
    F \cap U \seq \{x_1 = x_2 = \ldots = x_k = 0\}\,,
\end{equation*}  
where $x_j$ are boundary defining functions  in~$U$ of
the hyperfaces of $U$ containing $F \cap U$. Then $k$ is the boundary depth of $F$.
\end{remark}

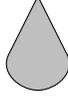
\begin{figure}\label{fig.teardrop}
\begin{center}
  \begin{tikzpicture}[scale=.25]
    \draw[fill=gray!50] (1.5,3) -- (0,0) .. controls (-.61,-1)  and (.5,-2) .. (1.5,-2) --  (1.5,-2)  .. controls (2.5,-2) and (3.61,-1) .. (3,0) -- cycle;
    \draw[fill=black] (1.5,3) circle (.08);
  \end{tikzpicture}
\end{center}
\caption{A teardrop domain as a subset in $\RR^2$.}
\end{figure}

In order to define the blow-up, 
we need the concept of an \emph{inward pointing normal bundle} of a \psubmanifold\ 
$P$ of $M$. We first recall the concept of an inward pointing tangent space to a 
manifold with corners $M$, a concept used also in \cite{sobolev, aln1}.

\begin{definition}\label{def.T+M}
Let $M$ be a manifold with corners.
We let the \emph{inward pointing tangent space} $T_x^+M$ to be defined
as the set of derivatives
\begin{equation*} 
  T_x^+M \ede \{ \gamma'(0) \mid \gamma:[0,\epsilon)\to M \text{ smooth},
\quad \gamma(0)=x \}\,.
\end{equation*}
Its elements are called \emph{inward-pointing tangent vectors} to $T_xM$.
\end{definition}

We stress that, according to this definition, 
a tangent vector to $M$ that is tangent to the boundary will
automatically be an inward pointing tangent vector.

\subsection{Submanifolds of manifolds with corners}

The following concept introduced in \cite{MelroseBook} will play a crucial role
in what follows. (Recall that $|A|$ denotes the number of elements of a 
set $A$.)

\begin{definition}
\label{def.submanifold}
A subset $P$ of a manifold with corners $M$ is a
\emph{\psubmanifold} if, for every $x \in P$, there exists a chart
$\phi : U \to \RR_k^n$, with $U$ an open neighborhood of $x$ in $M$,
and $I \subset \{1, 2, \ldots, n\}$ such that
\begin{equation*}
    \phi(P \cap U) \seq L_{I,k}^n \cap \phi(U),
\end{equation*}
with $L_{I,k}^n$ as defined in Equation \eqref{eq.p-manifold}.~The number
$n-|I|$ (respectively, $|I|$, respectively, $|I\cap
\{1,\ldots,k\}|$) will be called the \emph{dimension}
(respectively, the \emph{codimension} of $P$ at $x$, respectively, the
\emph{boundary depth} of $P$ at $x$). We allow \psubmanifolds\ $Y$ of
non-constant dimension. We define $\dim (Y)$ as the maximum of the
dimensions of the connected components of $Y$ and $\dim (\emptyset) = -\infty$. 
\end{definition}

In particular, $\SS^{n}_k$ is a closed \psubmanifold\ of $ S^{n+n'}_{k+k'}$ 
of codimension $n'$ and boundary depth $k'$.

\begin{definition}\label{def.NPM}
Let $P \subset M$ be a \psubmanifold{} of the manifold with
corners~$M$.
\begin{enumerate}[(i)]
\item The quotient bundle $N^MP \coloneqq  TM\vert_P/TP$ is called the \emph{normal
  bundle of $P$ in $M$}. 
\item The image $N^M_+P$ of $T^+M\vert_{P}$ in
$N^M P$ is called the \emph{inward pointing normal fiber bundle of $P$ in
  $M$.} One can show that, in a neighborhood of $x\in P$, it is a fiber bundle with 
  fiber $\RR^{n'}_{k'}$ where $n'$ is the codimension of $P$ in $M$ at $x$ and where 
  $k'$ is the boundary depth of $P$ in $M$.
\item The set $\SS(N^M_+P)$ of rays (emanating from $0$) in $N^M_+P$ is called the 
\emph{inward pointing spherical normal bundle of $P$ in $M$.} It is a fiber bundle 
with fibers  $\SS^{n'-1}_{k'}$, where $n'$ and $k'$ are as above.
\end{enumerate}    
\end{definition}

\begin{remark}A (positive definite) scalar product on $TM$
induces a scalar product on $N^MP$. The choice of a metric on $M$ will 
thus define a natural diffeomorphism from $\SS(N^M_+P)$ to the set of unit vectors in $N^M_+P$.
\end{remark}

Besides \psubmanifolds, we shall also need \emph{\wsbmanifolds,} which are really
just plain submanifolds (without any other qualification or condition). We now recall
(and slightly reformulate) the definition of a \wsbmanifold\ from \cite{AMN1}; 
the reformulation provides a definition equivalent to  
\citeschrSpace{Definition}{def.weak-submanifold}, but avoids introducing the concept of 
``submanifolds in Melrose's sense,''  \citeschrSpace{Definition}{def.submanifold-gen} and 
its simplified version \citeschrSpace{Definition}{def.weak-s.0}

\begin{definition}
\label{def.weak-submanifold}
A subset $S$ of a manifold with corners $M$ is a \emph{\wsbmanifold{} of $M$} if, for 
every $p \in S$,
there are
\begin{itemize}
\item natural numbers $k=k_p, m=m_p\in \{0,\ldots,n\}$ and $\ell=\ell_p\in \{0,\ldots,m\}$,
\item a chart $\phi=\phi_p : U \xrightarrow{\sim} V \subset \RR^n_k$ on $M$ with $V$ open in $\RR^n_k$, 
\item a diffeomorphism $\psi=\psi_p: \widetilde V \xrightarrow{\sim}  W$ with $\widetilde V$ and 
$W$ open in $\RR^n$
\end{itemize}
such that 
\begin{enumerate}
\item\label{def.weak-submanifold.i} $p \in U$,
\item $V=\widetilde V\cap \RR^n_k$, 
\item\label{def.weak-submanifold.ii} $\psi\bigl(\phi(S \cap U)\bigr)= W\cap \RR^m_\ell$.
\end{enumerate}
\end{definition}

If $S$ is a weak submanifold of a manifold with corners $M$, then 
$\{\psi_p\circ \phi_p\mid p\in S\}$ provides an atlas for a manifold with corners structure 
on $S$ such that the inclusion map $S\hookrightarrow M$ is an injective immersion and a homeomorphism 
onto its image. In conclusion, a weak submanifold of $M$ is a
manifold with corners on its own. Conversely, if $S$ is an abstract manifold with 
corners, together with an injective immersion $\iota:S\to M$ that defines a homeomorphism 
from $S$ to $\iota(S)$, then \citeschrSpace{Proposition}{prop.smfd.crit} states that $S$ 
is weak submanifold of $M$. Moreover, any \psbmanifold{} of $M$ is a weak
submanifold of $M$.

The number $m_p$ is the dimension of $S$ at $p$ and thus locally constant.
(Again we do not require the function $\dim_p(S)$ to be constant on $S$.)
The number $\ell_p$ is the (boundary) depth of $p$ in $S$, and might by smaller, 
larger or equal to $k_p$, the  boundary depth of $p$ in $M$.

Note that Definition~\ref{def.weak-submanifold} is weaker than the notion of a submanifold in Melrose's 
sense \citeschrSpace{Definition}{def.submanifold-gen}, which explain the word ``weak'' in the above 
definition.

\subsection{Clean intersections}
In order to study the iterate the blow-up construction in the next section, we will need 
\emph{clean intersections,} which we recall next. For simplicity, we discuss only 
the case of \psubmanifolds.

\begin{definition}\label{def.clean.in}
Let $P$ and $Q$  be two \psubmanifold s of a manifold with corners $M$
such that $P \cap Q$ is also a \psubmanifold\ of $M$. We say that 
$P$ and $Q$ \emph{intersect 
cleanly} or that \emph{they have a clean intersection} 
if, for every $x\in P\cap Q$, we have $T_x(P\cap Q)=T_x P\cap T_x Q	$.
\end{definition}

For example, if $P$ is a \psubmanifold\ of $M$ and $F$ is a boundary face of $M$, then 
$F$ and $P$  have a clean intersection. This is not, however, the case, in general,
if $P$ is just a \emph{weak submanifold} of $M$. This explains, in particular,
the need to consider \psubmanifolds. We next recall how to extend the 
definition of clean intersection to semilattices of \psubmanifolds.

\begin{definition}\label{def.clean.semilat}
Let $\maS$ be a finite semilattice of \psubmanifolds\ of $M$. 
We call $\maS$ a \emph{clean semilattice of\, \psubmanifolds\ (of $M$)} if
any $P,Q \in \maS$ intersect cleanly.
\end{definition}

\begin{remark}\label{rem.clean.intersection}
If $\maS$ is a clean semilattice of \psubmanifolds\, then, 
for all $P_1,\ldots,P_k\in \maS$ and all $x \in \bigcap_{j=1}^k P_{j}$, 
we obtain that 
\begin{equation} \label{eq.clean.intersection}
   T_x \Big(\bigcap_{j=1}^k P_{j}\Big) \seq \bigcap_{j=1}^k
   T_x P_{j} \,.
\end{equation}
In particular, our definition of clean semilattices coincides with 
Definition~2.7 in \cite{ACN}. This property does not hold anymore, 
if $\maS$ is not a semilattice, as seen in the next example.
\end{remark}

\begin{example}
Assume that $\maS$ is a set of \psubmanifolds\ with $k$ elements, such 
that each pair $\{P_i,P_j\}\subset \maS$ intersections cleanly, then we 
cannot conclude, in general, that we have the property 
expressed in Equation \eqref{eq.clean.intersection} for $k\geq 3$.
Indeed, let us consider the family consisting of the following 
three surfaces $P_1$, $P_2$ and $P_3$ in $\RR^3$ (i.e. $2$-dimensional submanifolds of $\RR^3$):
\begin{equation*}
\begin{gathered}
    P_1 \ede \left\{(x,y,z)^T\in \RR^3 \mid y = 0\right\}, 
    \qquad P_2\ede\left\{(x,y,z)^T\in \RR^3
    \mid z=0\right\}\,,\\
  P_3 \ede \left\{(x,y,z)^T\in \RR^3\mid y+z=x^2\right\}\,.
\end{gathered}
\end{equation*}

Each pair intersects cleanly in a $1$-dimensional submanifold. However $P_1\cap P_2\cap P_3=\{0\}$, but $T_0P_1\cap T_0P_2 \cap T_0P_3= \RR (1,0,0)^T$. Thus the triple $(P_1,P_2,P_3)$ does not intersect cleanly.
\end{example}

The following Lemma will provide the needed examples of pairs of cleanly intersecting 
\psubmanifolds.

\begin{lemma}\label{lem.clean}
Let $Y,Z$ be two linear subspaces of $X$. Then all subsets of 
$\{\overline Y,\overline Z, \SS_Y,\SS_Z\}$ intersect cleanly. 
\end{lemma}
\ifarxivversion
The proof is straightforward but involves many cases, so we relegate
it to Appendix~\ref{app.clean.intersections}, 
see Corollary~\ref{cor.ess.lem.clean}.
\fi

\section{Blow-ups} \label{sec2bis}
In this section we gather the needed results on blow-ups, 
following the approach of~\cite{AMN1}. Most of the results in this 
section were discussed in a similar form in the literature, see 
\cite{ACN, AMN1,grieser:b,grieser:overviewAMS, Kottke-Lin, MelroseBook} for more details. 
In particular, the basic construction of the blow-up along a \psubmanifold{} coincides 
with the one in these references. However, the lifting (or pull-back) of submanifolds 
differs in certain cases from the one used in, for instance,  \cite{grieser.talebi.vertman:p20, 
Kottke-Lin, MelroseBook}. As a consequence, our notion of iterated blow-up exhibits some subtle 
differences to the iterated blow-up discussed in the aforementioned papers. Our slightly different 
approach to the iterated blow-up has better locality and functorial properties, in particular it is 
compatible with restriction to open subsets, and thus avoids discussing additional special cases. 
However, our modified approach requires additional attention when citing existing literature: for 
example \cite[Lemma 7.2 (b)]{grieser.talebi.vertman:p20} does not hold using our definitions.
As a consequence, the current section will be written in a self-contained way,
both in order to provide reliable foundations and for better readability. The subtle difference 
in the approaches, however will finally have no effect in our applications: in these applications, 
our iterated blow-up spaces coincide with the iterated blow-up spaces in the work of Melrose and 
in the related work cited above.

\subsection{The blow-up along a closed \psubmanifold}
We now recall the definition of the blow-up $[M: P]$ of a manifold with
corners $M$ along a closed \psubmanifold\ $P$ of $M$. See \cite{AMN1,Kottke-Lin} 
for further details and references. We begin by specifying the underlying set of 
the blow-up $[M: P]$. Its topology and smooth structure will be defined shortly after that. 
If $A$ and $B$ are disjoint, we sometimes denote $A \sqcup B \ede A \cup B$ their union.

\begin{definition}\label{def.blow-up}[{see \cite[Definition~3.1]{AMN1}}]
Let $M$ be a manifold with corners and $P$ be a \emph{closed}
\psubmanifold\ of $M$. We let $\SS(N^M_+P)$ denote the inward pointing
spherical normal bundle of $P$ in $M$ (Definition \ref{def.NPM}). Then, as
a set, the \emph{blow-up of $M$ along $P$}
is the disjoint union
\begin{equation*}
  [M:P] \ede (M\setminus P) \,\sqcup \,\SS(N^M_+P)\,.
\end{equation*}
The blow-down map $\beta = \beta_{M,P} : [M:P] \to M$ is defined as the
identity map on $M\setminus P$ and as the fiber bundle projection
$\SS(N^M_+P) \to P$ on the complement.
\end{definition}

If $P \subset \pa M$ in the above definition, we shall say that $[M: P]$
is a \emph{boundary blow-up}. If no component of $P$ is contained in the
boundary of $M$, then we shall say that $[M: P]$ is an \emph{interior
blow-up}. See \cite{AMN1, Kottke-Lin, MelroseBook}, for instance, for
the definition of the topology and smooth structure on the blow-up $[M: P]$.
Let us say, nevertheless, that the smooth structure on $[M: P]$ is defined using
Euclidean model spaces and is such
that it induces the given smooth structures on $M\setminus P$ and on $\SS(N^M_+P)$.
We also remark that the topology on $[M: P]$ is such that $M$ has the quotient topology with 
respect to $\beta_{M, P}$. The following extreme cases of blow-ups deserve a special 
treatment.

\begin{remark}\label{rem.extreme}
If $P$ is a union of connected components of $M$, 
then $[M: P] = M \smallsetminus P$. In particular, we have $[M:\emptyset]= M$ 
and $[M:M]=\emptyset$.
\end{remark}


We shall need the following Proposition from  \cite{AMN1}. In that proposition a \emph{lift} 
denotes a smooth map $j^{\beta}$ that yields a commutative diagram
\begin{equation*}
\xymatrix{ [Q:P\cap Q] \ar[r]^{j^\beta} \ar[d]_{\beta_{Q,P\cap Q} }&
  [M:P] \ar[d]^{\beta_{M,P}} \\
  Q \ar[r]^\phi & M. }
\end{equation*}

\begin{proposition}[ \citeschrSpace{Proposition}{prop.beta.m1}]\label{prop.beta.m1}
Let $P$ and $Q$ be closed \psbmanifolds\ of $M$ intersecting
cleanly. Then the inclusion $j : Q \to M$ lifts uniquely to a smooth map
\begin{equation*}
  \underbrace{(Q \smallsetminus (P\cap Q))
  \,\sqcup\, \SS(N^Q_+ (P \cap Q) )}_{\displaystyle [Q : P \cap Q] \ede} \xrightarrow{\;\; j^{\beta}\;\; } \underbrace{(M\smallsetminus P)\,\sqcup\,
  \SS(N^M_+P)}_{\displaystyle [ M : P ] \ede}.
\end{equation*}
This map is an injective immersion, a homeomorphism onto its image, 
and the image of $j^\beta$ is a \psbmanifold. Moreover
\begin{equation*}
   \overline{\beta_{M,P}^{-1}(Q \smallsetminus P)} \seq j^\beta([Q: P \cap Q]).
\end{equation*}
\end{proposition}

In view of the above proposition, we shall write $[Q : P \cap Q] \subset [M : P ]$,
by abuse of notation.

We have the following useful factorization lemma due (essentially) 
to Kottke \cite{Kottke-Lin}. It deals with a particular, but important
case of the iterated blow-up. We shall need to
recall the formulation of the first factorization Lemma from \cite{AMN1}, i.e.\ 
\citeschrSpace{Lemma}{lemma.factorization1}, which deals with the setting 
$Q\subset P\subset M$ (see also \cite{ACN}).

\begin{lemma} \label{lemma.cancellation1}
Let us assume that $Q$ is a \psubmanifold\ of~$P$ and that $P$ is a \psubmanifold{} 
of~$M$. Then $Q$ is a \psubmanifold\ of~$M$. Moreover, there exists a smooth, canonical map
\begin{equation*}
  \zeta_{M, Q, P} : [M: Q, P ] \ede \bigl[[M: Q]: [P:Q]\bigr] \ \to \ [M : P]
\end{equation*}  
that restricts to the identity on $M \smallsetminus P$. 
\end{lemma}

\subsection{Iterated blow-ups}
In the applications of the blow-up, we typically have to blow-up
several subsets. This subsection deals with some of the intricacies of this 
procedure. The first thing to discuss is the \emph{pull-back} of a subset in $M$ to
the blow-up $[M: P]$ following \cite{AMN1, Kottke-Lin, KottkeMelrose2}.

\begin{definition}\label{def.beta.pullback}
Let $P$ be a \psubmanifold{} of $M$ and $Q$ be a closed subset of $M$. The \emph{pull-back} 
or \emph{lifting} $\beta^\ast_{M,P}(Q)$ of $Q$ to $[M: P]$ is defined by
\begin{equation}\label{eq.def.pull-back}
   \beta^\ast_{M,P}(Q) \ede \overline{\beta^{-1}(Q \smallsetminus P)} 
\end{equation}
\end{definition}

In the case $Q \subset P$, our definition of $\beta^\ast_{M,P}(Q)$ is different from the 
one given by Melrose in \cite[Chapter 5, Section 7]{MelroseBook},
see \citeschrSpace{Remark}{betaMPQremark} for details.

The following factorization lemma is similar in spirit, only easier 
(see \cite{ACN, AMN1, Kottke-Lin}).

\begin{lemma}
\label{lemma.cancellation2}
Let us assume that $P$ and $Q$ are closed, \emph{disjoint} \psbmanifolds{} of $M$. 
We have $\beta_{M, Q}^*(P) = P$ and similarly for $P$ and $Q$ switched. With these 
identifications, we then have  $\bigl[[M: Q]: P\bigr] \simeq \bigl[[M: P]: Q\bigr]$ 
canonically, and hence there exists a smooth, canonical map
\begin{equation*}
  \zeta_{M, Q, P} : \bigl[[M: Q]: P\bigr] \ \to \ [M : P]
\end{equation*}  
that restricts to the identity on $M \smallsetminus (P \cup Q)$.
\end{lemma}

In order to introduce the iterated blow-up with respect to an ordered family 
($n$-tuple) of closed subsets, we first introduce the families with
respect to which we can define the blow-up. We found it convenient to allow repetitions 
in these families. The following is Definition~\refschrSpace{def.iterated.bu} in \cite{AMN1} 
and Definition~2.9 in \cite{ACN}, but see also \cite{Kottke-Lin} and \cite{MelroseBook}.

\begin{definition}\label{def.bu.suitable}
Let $M$ be a manifold with corners and let $\maP \coloneqq  (P_i)_{i=1}^k$ be a  
$k$-tuple of closed subsets of $M$, $k \ge 1$. If $P_1$ is a \emph{closed \psubmanifold} 
of $M$ and $k > 1$, we define the \emph{pull-back} of $\maP$ to be the $(k-1)$-tuple
\begin{equation*} 
    \maP' \ede \beta_{M, P_1}^*(\maP \smallsetminus \{P_1\}  ) 
    \ede \bigl(\beta_{M, P_1}^*(P_i)\bigr)_{i=2}^k\,.
\end{equation*}
Then, by induction on $k$, we say that $\maP$ is \emph{blow-up-suitable (in $M$)} if:
\begin{enumerate}[(i)]
\item $P_1$ is a \emph{closed \psubmanifold} of $M$ and 
\item \label{item.eq.def.pbt} if $k>1$, the pull-back
$\maP'$ is blow-up-suitable in $[M:P_1]$.
\end{enumerate}
Of course, $k = |\maP|$. If $k = 0$, that is, if $\maP = \emptyset$, then
we also say that $\maP$ is blow-up-suitable in $M$.
\end{definition}

In what follows, we will often use the pull-back $\maP'$ of various $k$-tuples
$\maP$. We now define the iterated blow-up with respect to blow-up-suitable families.

\begin{definition}\label{def.iterated.bu}
We use the notation introduced in Definition \ref{def.bu.suitable}, in particular, 
$\maP \coloneqq  (P_i)_{i=1}^k$ is a $k$-tuple of closed subsets 
of $M$ and $\maP' \coloneqq \beta_{M, P_1}^*(\maP \smallsetminus \{P_1\}  ) 
\coloneqq \bigl(\beta_{M, P_1}^*(P_i)\bigr)_{i=2}^k$ (the pull-back of $\maP$). 
If $\maP$ is blow-up-suitable, then we define by induction the \emph{iterated blow-up}
$[M: \maP]$ by
\begin{equation*}
  [M: \maP] \seq [M : (P_i)_{i=1}^k] \ede
  \begin{cases}
    \ [M: P_1] & \mbox{ if }\ k = 1\,,\\
   \, \big[[M:P_1]: \maP' \big] &
   \mbox{ if }\ k > 1\,.
   \end{cases}
\end{equation*}
The blow-down map $\beta_{M, \maP} : [M: \maP] \to M $ is defined by induction 
on $k$ as the composition of the blow-down maps $[M: \maP] := [[M: P_1] : \maP'] \to 
[M: P_1] \to M$. 
\end{definition}

It is not difficult to see, by induction, that Definition 
\ref{def.iterated.bu} makes sense. This is the point of introducing
Definition \ref{def.bu.suitable}. We shall also use $[M: P_1, P_2, \ldots, P_k]$ 
as an alternative notation for $[M: \maP]$. Let $\maE := \emptyset$,
(which is regarded as a $k$-tuple with $k = 0$) and define $[M: \maE] = M$
and $\maP' = \maE$ for if $\maP$ has only one element. Then the relation 
$[M: \maP] =  \big[[M:P_1]: \maP' \big]$ remains true also for $k = 1$.
This may be useful for proofs by induction, for instance in the proof
of Lemma \ref{lemma.disjoint.sd}.

\begin{remark}\label{rem.Pprime}
Of course, a $k$-tuple $\maP$ \emph{without} 
repetitions is the same thing as a \emph{linearly ordered set}. 
We want the $\maP$ in the definition of the iterated blow-up
$[M: \maP]$ to be a $k$-tuple rather than a linearly ordered finite 
set of \psubmanifolds. That is, we want to \emph{allow repetitions} in $\maP$. The
reason for this choice is that, even if $\maP$ does not have repetitions,
its pull-back $\maP' \coloneqq  \bigl(\beta^*(P_i)\bigr)_{i=2}^k$ might have repetitions.
An example is provided by $\maP = (A, B, A \cup B)$, where $A$ and $B$
are disjoint closed \psubmanifolds\ of $M$, in which case $\maP' = (B, B)$. 
We shall often consider semilattices with an additional total order (other than
the order given by inclusion).
\end{remark}

In the next remark we will explain how to eliminate repetitions. 
For this purpose, the following proposition is a useful technical result.
Recall that~$\unionP \coloneqq  \bigcup_{P \in \maP} P$.

\begin{proposition} \label{prop.density} 
We use the notations of Definition \ref{def.iterated.bu} and we
assume that $\maP$ is blow-up-suitable (and hence that $[M: \maP]$ is defined). Then 
\begin{enumerate}[(i)]
\item $M \smallsetminus \unionP \subset [M: \maP]$ and $\beta = \id$ 
on $M \smallsetminus \unionP$ and
\item $M \smallsetminus \unionP$ is open and dense in $[M: \maP]$.
\end{enumerate}
\end{proposition}

\begin{proof}
The first part follows from the definitions of the blow-up and of the iterated
blow-up. To prove the second part, we proceed by induction on $k$, the number
of elements of~$\maP$. The case $k = 1$ follows from the definition (and was
discussed also in \cite{AMN1}). For the induction step, we notice that
\begin{equation*}
   M \smallsetminus \unionP \seq
   (M \smallsetminus P_1) \smallsetminus \unionP' \subset 
   [M :P_1] \smallsetminus \unionP'
\end{equation*}
is dense in $[M : P_1]  \smallsetminus \unionP'$ by the case $k = 1$, 
since $\unionP'$ is closed. Since $[M : P_1]  \smallsetminus \unionP'$ is dense 
in $[M: \maP] \coloneqq  
\big[[M:P_1]: \maP' \big]$ by the case $k-1$ (the induction hypothesis), 
the result follows.
\end{proof}

Given two blow-ups of $M$, we shall say that they are \emph{canonically
diffeomorphic} if there exists a diffeomorphism that is the identity outside
the sets that are blown up. In view of the above proposition, this entails
a uniqueness property for the canonical diffeomorphisms.
We now explain how we can eliminate the repetitions in $\maP$.

\begin{remark} \label{rem.repetitions}
We use the notation introduced in Definition~\ref{def.iterated.bu}
and the results of Proposition~\ref{prop.density}. Let $\maP_{\mathrm{red}}$ be 
obtained from $\maP$ by removing repetitions by keeping only the \emph{first} 
appearance of a \psubmanifold. Thus, if $\maP = (\emptyset, A, B, A)$, then 
$\maP_{\mathrm{red}} = (\emptyset, A, B)$. Then, $[M: \maP]$ is defined if, and only if,
$[M: \maP_{\mathrm{red}}]$ is defined. Moreover, these two iterated blow-ups are
canonically diffeomorphic (when defined), in the sense that:
\begin{enumerate}[(i)]
  \item $M \smallsetminus \unionP = M \smallsetminus \unionP_{\mathrm{red}}$;
  
  \item \label{item.rem.repetitions.ii} $M \smallsetminus \unionP$ is dense in $[M: \maP]$ 
  and $M \smallsetminus \unionP_{\mathrm{red}}$ is dense in $[M: \maP_{\mathrm{red}}]$,
  by Proposition \ref{prop.density}; and 
   
  \item the identity map $M \smallsetminus \unionP \to M \smallsetminus \unionP_{\mathrm{red}}$ 
  extends to a diffeomorphism $[M: \maP] \to [M: \maP_{\mathrm{red}}]$, which is unique
  by the density properties of \ref{item.rem.repetitions.ii}. 
\end{enumerate}
\end{remark}

In addition to removing repetitions, we could as well remove entries $P_j = \emptyset$
from $\maP$ with the same effect. However, we found it convenient for exposition 
purposes to do exactly the opposite, that is, to usually assume $\emptyset$ to 
be the first entry of $\maP$, especially if $\maP$ is a semilattice. We have the 
following properties.

\begin{remark} \label{rem.repetitions2}
We use the notation introduced in Definition~\ref{def.iterated.bu}. 
In what follows, the pull-back operation $\maP \mapsto \maP'$
introduced in 
\begin{enumerate}[(i)]
     \item We have $[M: \maP] \coloneqq [[M: P_1], \maP'] = [[M: P_1] : \maP'_{\mathrm{red}}]$, 
     which may be useful if one wants to deal only   with \emph{reduced} tuples (which is 
     the same as linearly ordered finite sets).
     \item Let us introduce the \emph{iterated pull-backs} of $\maP$ by
     $\maP^{(1)} \coloneqq \maP'$ and $\maP^{(k+1)} \coloneqq (\maP^{(k)})'.$
     Then 
\begin{equation*}
     [M: \maP] \seq [[M: P_1, P_2, \ldots, P_k]: \maP^{(k+1)}] \,.
\end{equation*}  
     \item $(\maP'_{\mathrm{red}})'_{\mathrm{red}} = \maP''_{\mathrm{red}}$ 
     (where $'$ always comes before ${}_{\mathrm{red}}$, meaning
     that $\maP'_{\mathrm{red}} \coloneqq (\maP')_{\mathrm{red}}$).  
     \item\label{jrhesgs} A special situation arises if $P_1$, the first element of $\maP$, is 
     $P_1 = \emptyset$, which is, in fact,
     the norm when dealing with semilattices. Then, in our definition of the blow-up,
     the first step, the blow-up with respect to $P_1$ is trivial (it does not
     change our sets, except that it removes $P_1$ from the list). In particular, 
     $\maP = (\emptyset, \maP')$. It is the next
     blow-up that may be interesting. For us, it will then be useful to introduce
     the following notation
     \begin{equation*}
         \widetilde \maP \ede (\emptyset, \maP'') \ede (\emptyset,
         \beta_{M: P_2}^*(P_3), \beta_{M: P_2}^*(P_4), \ldots , \beta_{M: P_2}^*(P_k)) \,.
     \end{equation*}
     (Note that $\beta_{M: P_2}^*(P_2) = \emptyset$.) We then have the relation
     \begin{equation*}
        [M: \maP] \seq [M: \maP'] \seq [[M: P_2]: \maP''] 
        \seq [[M: P_2]: \widetilde \maP] 
        \seq [[M: P_2]: \widetilde \maP_{\mathrm{red}}]\,.
     \end{equation*}
     We also notice that, if $\maP_{\mathrm{red}}$ is a clean semilattice,    
     then $\widetilde \maP_{\mathrm{red}} = (\widetilde {\maP_{\mathrm{red}}})_{\mathrm{red}}$ 
     is also a clean semilattice;  here the semilattice property is obvious and the 
     cleanness was proved in \cite[Theorem~2.8]{ACN}. (In this context we define again
     $\widetilde \maP_{\mathrm{red}} \coloneqq (\widetilde \maP)_{\mathrm{red}}$,
     that is, the ``tilde'' comes before ``reduced''.)
     \item \label{item.iter}
     To wrap up the list of needed properties, 
     let us notice that 
\begin{equation*}
     (\emptyset, \maP)' = \maP \ \mbox{ and } \ [M: (\emptyset, \maP)] = [M: \maP] \,.
\end{equation*}     
     In particular, if we define $\maP_0 := \maP^{(0)} := \maP$,
     $\maP_{k+1} := \widetilde \maP_k$ and $\maP^{(k+1)} = (\maP^{(k)})'$, then $\maP_{k}
     = (\emptyset, \maP^{(k+1)})$ for $k \ge0$. Hence the iteration of the tilde operation 
     can be expressed in terms of the iteration with respect to the 
     prime operation and the addition of the empty set.  
\end{enumerate}
\end{remark}

\color{black}

The following results generalize to the iterated blow-up 
\citeschrSpace{Lemma}{lemma.product} on the compatibility of the blow-up with
products.

\begin{proposition}\label{prop.lemma.product} 
Let $M$ and $M_1$ be two manifolds with corners and 
$\maP = (P_1, P_2, \ldots, P_k)$ be a blow-up-suitable (in $M$) 
$k$-tuple of closed subsets of $M$. Then $\maP \times M_1 \coloneqq (P_1 \times M_1, 
P_2 \times M_1, \ldots, P_k \times M_1)$ is a blow-up-suitable (in $M \times M_1)$
$k$-tuple of closed subsets of $M \times M_1$ and there exists a canonical
diffeomorphism $[M \times M_1 : \maP \times M_1] \simeq [M: \maP] \times M_1$
such that the following diagram commutes:
\begin{equation}\label{eq.CD}
\begin{CD}
  [M \times M_1 : \maP \times M_1] @>{ \simeq }>>[M: \maP] \times M_1 \\
  @V{ \beta_{M \times M_1, \maP \times M_1}}VV
  @VV{\beta_{M,\maP} \times \id}V \\
   M \times M_1 @>{ \id }>> M \times M_1 \,.
\end{CD}
\end{equation}
\end{proposition}

\begin{proof}
If $k=1$ (that is, $\maP$ consists of a single set), then the result was
proved in \cite{ACN} (it can be found also in \cite{AMN1}). In general, 
it follows by induction, using again
the result from \cite{ACN} (the case $k = 1$) and using also
that $\beta^*(P) \times M_1 = \beta^*(P \times M_1)$, where
$\beta$ is an appropriate blow-down map.
\end{proof}

\subsection{Admissible orders}
A natural question when we do iterated blow-ups of a manifold with 
corners $M$ along an ordered family $\maP = (P_0, P_1, \ldots, P_k)$, 
is to decide how the order of the blow-up influences the final space
\cite{Kottke-Lin}. In particular, a related question is 
whether the iterated blow-up is defined for a given order.
The aim of this subsection is to recall the results of \cite{ACN, AMN1, Kottke-Lin}
that give a positive answer to these questions if ``admissible orders'' are
used. Before stating the main result from \cite{AMN1}, we first introduce
admissible orders and graph blow-ups, which will be needed for the
statement of the theorem.

If, in the definition of a blow-up-suitable $k$-tuple $\maP$,
we further require $P_1$ to be minimal for inclusion, we obtain the notion
of an ``admissible'' $k$-tuple. Let us state this explicitly.

\begin{definition}\label{def.admissible} 
Let $M$ be a manifold with corners and let $\maP \coloneqq  (P_i)_{i=1}^k$ be a  
$k$-tuple of closed subsets of $M$, $k \ge 1$. By induction on $k$, 
we say that $\maP$ is \emph{admissible (in $M$)} if:
\begin{enumerate}[(i)]
\item $P_1$ is a \emph{closed \psubmanifold} of $M$,
\item there is no $i > 1$ such that $P_i \subsetneqq P_1$, and
\item \label{item.eq.def.pbt2} if $k>1$, 
$\maP'$ is admissible in $[M:P_1]$, where $\maP'=(\beta_{M,P_1}^*(P_j))_{j=2}^k$ 
is the pull-back of $\maP$ as before (in particular, $\beta^*_{M,P_1}$ is as defined 
in \eqref{eq.def.pull-back}).
\end{enumerate}
Of course, $k \ge |\maP|$. If $k = 0$, that is, if $\maP = \emptyset$, then
we also say that $\maP$ is admissible in $M$.
\end{definition}

This definition of and admissible $k$-tuple 
is more general than the one in \cite{AMN1}.

\begin{remark}\label{rem.semilattice.t}
We notice the following
\begin{enumerate}
\item A $k$-tuple with an admissible order is, in particular,
also blow-up-suitable.

\item $\maP$ is admissible if, and only if, $(\emptyset, \maP)$ is admissible.

\item Assume the $(k+1)$-tuple $\maS = (P_0, P_1, \ldots, P_k)$ is a clean semilattice, 
$P_0 = \emptyset$. Then $\widetilde \maS \coloneqq \bigl (\emptyset, 
\beta_{M, P_1}^*(P_2), \beta_{M, P_1}^*(P_3), ... , \beta_{M, P_1}^*(P_k) \bigr)$ 
is also a clean semilattice by the results of \cite{ACN}, 
see Remark~\ref{rem.repetitions2} \eqref{jrhesgs}.

Moreover, $\maS$ (with the indicated order) is
admissible if, and only if, $\widetilde{\maS}$ (with the induced order) is admissible. 
This explains why we sometimes consider $\widetilde \maS$ and do not work exclusively with
the iterated pull-backs $\maS', \maS'', \ldots, \maS^{(j)}$.
\end{enumerate}
\end{remark}

We now consider a different type of blow-up with respect to a $k$-tuple
of \psubmanifolds\ that is immediately seen not to depend on the
choice of the order. 
To define it, we introduce the \emph{multi-diagnonal map}. For a semi-lattice $\maS$
endowed with an admissible order,
we let $U\coloneqq M\smallsetminus \unionma{S}$, which is a dense subset of $[M:\maS]$. 
The multi-diagonal map is the map
  \begin{equation}\label{eq.multi.diag.map}
    \delta:U \to \prod_{P\in \maS} [M:P], \quad     x \mapsto  (x,x,\ldots,x).
  \end{equation}

\begin{definition}\label{def.unres.blowup}
Let $\maP = (P_i)$ be a $k$-tuple of closed 
\psbmanifolds{} of the manifold with corners~$M$ and let $\delta$ be the multi-diagonal map
defined in \eqref{eq.multi.diag.map}.
Then the \emph{graph blow-up} $\bl{M: \maP}$ of $M$ along $\maP$ is defined by
\begin{equation*}
    \bl{M: \maP} \ede \overline{ 
    \delta(M\smallsetminus \unionP) }
    \seq       
    \overline{ \{ (x, x, \ldots, x) \, \vert \ x \in
    M\smallsetminus \unionP \} } \ \subset\ \prod_{i \in I} [M: P_i] \,.
\end{equation*}
\end{definition}
For the graph blow-up, we can also remove the repetitions.

\begin{remark}
The graph blow-up  $\bl{M: \maP}$ is a weak submanifold of $Q\coloneqq \prod_{i \in I} [M: P_i]$. 
The notion of a ``weak submanifold'' was introduced in \cite[Subsec.~2.3.1]{AMN1}, however the 
only aspect we need to know about it here is that, if one restricts the sheaf of smooth functions 
on $Q$ to $\bl{M: \maP}$, then this restriction defines the structure of a manifold with corners 
on $\bl{M: \maP}$ which is compatible with the topology induced from $Q$.
\end{remark}

Theorem~\refschrSpace{thm.main1} of \cite{AMN1} shows that, if $\maS$ is a semilattice endowed 
with an admissible order, then the iterated blow-up and the graph blow-ups of $M$ with 
respect to $\maS$ are canonically diffeomorphic. The following statement combines the
statement of that theorem with part of Remark~\refschrSpace{rem.size-order} 
and with Theorem~\refschrSpace{thm.cor.main1} 
of that paper.

\begin{theorem}\label{theorem.mainAMN1}
Let $\maS \ni \emptyset$ be a clean semilattice of closed \psbmanifolds{} of~$M$ with 
an admissible order on its elements (Definition \ref{def.admissible}). Then the
iterated blow-up $[M: \maS]$ is defined and $\bl{M:\maS}$ is a manifold with corners 
that is canonically diffeomorphic to $[M:\maS]$. If $G$ is a Lie group acting 
smoothly on $M$ such that $G$ maps $\maS$ to itself, then $G$ acts by diffeomorphisms
on $[M: \maS]$.
\end{theorem}

In the theorem ``canonically diffeomorphic'' means that the multi-diagonal map defined in \eqref{eq.multi.diag.map}
 has a unique smooth extension which provides this diffeomorphism.

The iterated blow-up does not depend on the order of the sets 
(up to a canonical diffeomorphism). Hence, an immediate consequence of the 
last theorem is that the iterated blow-up $[M: \maS]$ does not depend
on the choice of the \emph{admissible} order on $\maS$. 

The reader may wonder, at this time, whether admissible orders exists on a given 
$k$-tuple. Note that, at this time, it is not clear even that a blow-up-suitable
order exists on a given $k$-tuple of closed \psbmanifolds{} of $M$. 
The following proposition gives a positive answer
to this question if our $k$-tuple is a \emph{clean semilattice}. This result and the
previous theorem motivates the use of semilattices in our work.

\begin{remark}\label{rem.ex.ao}
Let $\maS$ be a finite, clean semilattice of closed \psbmanifolds{} of $M$, 
$\emptyset \in \maS$. Let us assume construct an order $(P_0, P_1,\ldots, P_k)$ 
on the elements of $\maS = \{P_0, P_1,\ldots, P_k\}$
by choosing $P_j \in \maS$ by induction on $0 \le j \le k= |\maS|-1$ as follows.
\begin{enumerate}[(i)] 
   \item Let us choose an arbitrary initial order on the elements of $\maS$,
   to be able to talk about pull-backs. With each choice of $P_j$, we
   modify this order by moving $P_j$ on the $(j+1)$-position. (So, after
   choosing $P_j$, the first $(j+1)$-elements of the modified order will be 
   the chosen elements $(P_0, P_1, \ldots, P_j)$.)
   
   \item We begin by choosing $P_0 \coloneqq \emptyset$, which implies that 
   the pull-back $\maS'$ is given by $\maS' = \maS \smallsetminus \{\emptyset\}$ 
   (after which we change the order on $\maS$ according to the previous point). 
   
   \item We let $P_1$ to be an arbitrary minimal element of $\maS'$ for
   inclusion. 
   
   \item Let $$\widetilde \maS \coloneqq (\emptyset, \maS'') \coloneqq
   \big (\emptyset, \beta_{M, P_1}^*(\maS \smallsetminus \{\emptyset, P_1 \}) \big )
   = \beta_{M, P_1}^*(\maS \smallsetminus \{\emptyset\}) \,.$$ 
   Then the elements of $\widetilde \maS$ form again a clean semilattice 
   of closed \psbmanifolds{} of~$M$ \cite[Theorem~2.8]{ACN} and
   we choose $P_2 \in \maS \smallsetminus \{P_0, P_1\}$ such that 
   the lift $\beta_{M, P_1}^*(P_2)$ of $P_2$ is
   minimal element of $\maS''$ for inclusion. In particular, $\beta_{M, P_1}^*(P_2)$ 
   will be a closed \psbmanifold{} of $[M: P_1]$ by the aforementioned
   result in \cite{ACN}.
   
   \item We then iterate this construction with $\widetilde{\maS}$
   in place of $\maS$. More precisely, recall from Remark \ref{rem.repetitions2}(\ref{item.iter})
   that $\maS_j := \widetilde{\maS}_{j-1} = (\emptyset, \maS^{(j+1)})$, 
   where $\maS^{(j+1)} = (\maS^{(j)})'$ and $\maS_0 = \maS^{(0)} 
   = \maS$. Assume $P_0, P_1, \ldots, P_j$ were chosen. Then we choose $P_{j+1} \in \maS    
   \smallsetminus \{P_0, P_1, \ldots, P_j\}$, $j+1 \le k$, to correspond to a minimal element of 
   $\maS^{(j+1)}$ for inclusion. This satisfies the desired condition that
   the lift of $P_{j+1}$ in $\maS^{(j+1)}$ be a \psbmanifold.
\end{enumerate}
\end{remark}

\begin{proposition}\label{prop.ex.ao}
We use the notation introduced in Remark \ref{rem.ex.ao}. First, the procedure
of that remark is well-defined in the sense that it yields an order
$(P_0, P_1, \ldots, P_k)$ on the elements of $\maS$. Most
importantly, the resulting order is admissible. Let 
$\emptyset \neq Y \in \maS$. With suitable choices in the procedure
of that remark, we obtain an admissible
order on $\maS$ such that all elements that precede $Y$ in this
order are contained in $Y$.
\end{proposition}

\begin{proof}
The procedure of Remark \ref{rem.ex.ao} is well-defined since, at each step, the 
resulting tuples $\widetilde \maS := (\emptyset, \maS'')$ and 
$(\emptyset, \maS^{(j+1)}) := \widetilde {(\emptyset, \maS^{(j)})}$, $k \ge 1$, 
are clean semilattices, by \cite{ACN}, so the procedure stops only after we
have chosen $P_k$ (i.e. after having ordered all elements). 
(See Remark \ref{rem.repetitions2}\eqref{item.iter} 
for the notation.) The resulting order 
on $\maS$ is admissible, by the definition of an admissible order and by induction on 
the number of elements of $\maS$. Finally, given $Y \in \maS$, we are going to successively 
choose the sets $P_k$ to be either $Y$ or a set contained in $Y$, if possible (that is, unless 
$Y$ has already been chosen). This procedure will yield the desired
admissible order on $\maS$.
\end{proof}

\subsection{A blow-up point of view on the spherical compactification}\label{subsec.blow-up-spherical}
This subsection is not needed for the main results of the article. We will show the following 
proposition which might add a helpful perspective for future research. In particular, it shows 
that if we blow-up at infinity the one-point compactification $X_\infty$ of a finite-dimensional 
vector space $X$ obtained by stereographic projection, then we obtain a compactification 
canonically diffeomorphic to the spherical compactification $\ol X$.

For introducing the one-point compactification $X_\infty$, we 
use the scalar product $\<\argu,\argu\>$ on $X$, and we consider the (unit) sphere 
$\SS_{\RR\times X}$ in $\RR\times X$. We define the \emph{south pole} 
$S\coloneqq (-1,0)\in \SS_{\RR\times X}$. The \emph{stereographic projection} is the map
\begin{eqnarray*}
  \sigma\colon\SS_{\RR\times X} \setminus\{S\}&\to& X,\\
  \RR\times X\ni  \begin{pmatrix}\cos\theta\\\sin(\theta)\cdot y \end{pmatrix}&\mapsto& \tan \Bigl(\frac\theta2\Bigr)\,\cdot\, y,\qquad \theta\in (0,\pi],\quad y\in \SS_{X} \,.
\end{eqnarray*}
After formally defining $\infty\coloneqq\sigma(S)$, there is a unique topology and a unique 
smooth structure on $X_\infty\coloneqq X\cup\{\infty\}$ such that $\sigma$ is a diffeomorphism. 
This manifold $X_\infty$ is called the \emph{one-point-compactification} of X.
\begin{figure}\label{fig.compactific}
   \def\circlethick{.4mm}
   \def\strahlthick{.25mm}
   \def\strahlcoleins{red} 
   \def\stereocol{darkblue} 
   \def\sphericol{darkgreen} 
\begin{center}
   \begin{tikzpicture}[scale=2]
      \node[right] (X) at (2,0) {$X$}; 
      \node[above] (R) at (0,1.2) {$\RR$}; 
      \draw[line width=\circlethick,\stereocol] (0,0) circle (1);
      \node[left,\stereocol] (R) at (-.53,.9) {$\SS_{\RR\times X}$}; 
      \draw[black,fill=black] (0,-1) circle (.02);
      \node[below,black] (R) at (-.1,-1) {$S$}; 
      \draw[line width=\circlethick,\sphericol] (1,-1) arc (0:180:1);
      \draw[line width=\circlethick] (-2,0) -- (2,0);
      \draw[fill=gray!50] (0,-1.2) -- (0,1.2);
      \draw[\strahlcoleins,line width=\strahlthick] (0,-1) -- (1.1,1.2);
      \draw[\strahlcoleins,line width=\strahlthick] (0,-.6) arc (90:62.5:.4);
      \node[\strahlcoleins] (winkeleinsa) at (.06,-.7) {$\scriptscriptstyle\alpha$}; 
      \draw[\strahlcoleins,line width=\strahlthick] (.39,.28) arc (210:242.5:.4);
      \node[\strahlcoleins] (winkeleinsb) at (.543,.28) {$\scriptscriptstyle\alpha$}; 
      \draw[\strahlcoleins!40,line width=\strahlthick] (0,0) -- (1.2,.9);
      \draw[\strahlcoleins!40,line width=\strahlthick] (0,.4) arc (90:36:.4);
      \node[\strahlcoleins!40] (winkeleinsc) at (.1,.26) {$\scriptscriptstyle2\alpha$}; 
      \draw[\strahlcoleins,fill=\strahlcoleins] (.5,0) circle (.025);
      \draw[\stereocol,fill=\stereocol] (.8,.6) circle (.025);
      \draw[\sphericol,fill=\sphericol] (.45,-.1) circle (.025);
   \end{tikzpicture}
\end{center}
\caption{One-point and disk compactification. The stereographic projection $\sigma$ maps the \textcolor{\stereocol}{blue dot} to the \textcolor{\strahlcoleins}{red dot}, and the map $\Theta_X$ maps the \textcolor{\strahlcoleins}{red dot} to the \textcolor{\sphericol}{green dot}.}
\end{figure}
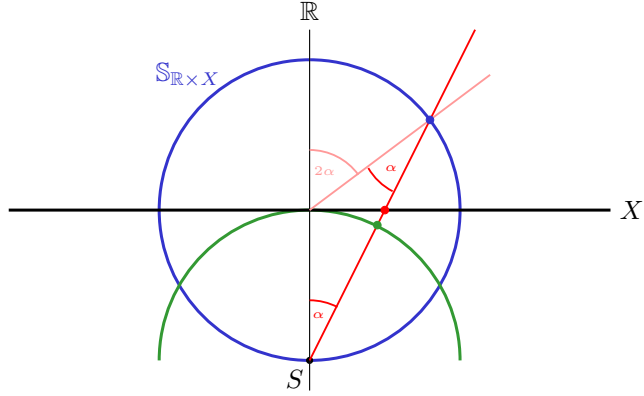

We also consider a new map $\Theta_X$, which is
a slightly modified version of the map $\Theta_n$ from \eqref{eq.def.theta}. 
For that purpose, let
\begin{equation*}
\begin{gathered}
\SS_+':= \bigl\{x +S\mid x\in \SS_{[0,\infty)\times X}\bigr\} = \bigl\{x -(1,0)\mid x=(x^0,x'),\; \|x\|=1,\; x^0\geq 0\bigr\}\,,\\
\Theta_X:\ol X\to \SS_+',\qquad \begin{cases}
\ \Theta_X(x) \ede \frac{1}{\sqrt{1+\<x,x\>}} (0,x) \in \SS_+' 
& \mbox{ if } x \in X \,, \\
\ \Theta_X(\RR_+v) \ede \frac{1}{\|v\|}(-1,v) \in \SS_+'  
& \mbox{ if } \RR_+v \in \mathbb{S}_X\,.
\end{cases}
\end{gathered}
\end{equation*}

We write $\SS_{X} \coloneqq \ol X\setminus X$ for the sphere at infinity of $X$,
as usual.

\begin{lemma}
Let $\Psi\colon\ol{X}\to X_\infty$ be the map with $\Psi|_X=\id_X$ and $\Psi(\SS_{X})
= \{\infty\}$. This is a smooth map from the spherical compactification $\ol X$ to the one-point-compactification $X_\infty$, extending  the identity $\id_X$. There is a diffeomorphism $\Phi:\ol X\to \bigl[X_\infty:\{\infty\}\bigr]$ extending $\id_X$, and thus $\Psi\circ \Phi^{-1}$ is the blow-down 
map $\beta_{X_\infty,\{\infty\}}$.
\end{lemma}

\begin{proof} 
We define $\psi\coloneqq \sigma^{-1}\circ \Theta_X^{-1}\colon  \SS_+'\to \SS_{\RR\times X}$. 
Then $\Psi=\sigma\circ \psi \circ \Theta_X$. By construction, $\sigma$ and $\Theta_X$ are 
diffeomorphisms. To prove the lemma, one thus has to show that $\psi$ is that there is a 
diffeomorphism $\phi\colon \SS_+'\to \bigl[\SS_{\RR\times X}:\{S\}\bigr]$ such that  
$\psi\circ \phi^{-1}=\beta_{\SS_{\RR\times X},\{S\}}$. The construction of such a $\phi$ 
is an easy exercise, if one uses the following formula for $\psi$, which is apparent in 
view of Figure~\ref{fig.compactific}, setting $\theta\coloneqq 2\alpha$ above.   
\begin{equation*}
 \psi\Bigl(\bigl((\cos \alpha)-1,(\sin \alpha) y\bigr) \Bigr) 
 \seq \bigl((\cos 2\alpha), (\sin 2 \alpha) y\bigr), \qquad \alpha\in [0,\pi/2],\; Y\in \SS_X \,.
\end{equation*}
The function $\cos \alpha$ is a boundary defining function for $\SS_+'$ and a smoothed distance 
function for $S$ in $\SS_{\RR\times X}$. This completes the proof.
\end{proof}

\section{Distance functions and Sobolev spaces for blown-up spaces}\label{sec3}

We now investigate how several geometric quantities (metrics, distance functions, 
Sobolev spaces, natural differential operators, ...) change when performing a 
blow-up. The manifolds~$M$ we consider 
have a complete metric in the interior $\interior{M}$ described precisely in terms of 
Lie manifolds, Definition~\ref{def.comp.metric}.
One important case will be that when $\interior M$ is a Euclidean vector space $X$ with 
its spherical compactification $M=\oX$, which was described in the 
Introduction. Some other times, $M$ will be a blow-up of $\oX$.

We distinguish here the case of a blow-up along a submanifold \emph{contained 
in the boundary} (the easy case) and the case of a manifold \emph{not contained in 
the boundary} (the difficult case, but treated already in \cite{ACN} and in other 
papers). In the first case, the \emph{boundary case,} we will additionally require 
that our vector fields are tangent to the submanifold contained in the boundary. In 
that case, the metric in the interior remains the same, only the compactification 
is altered, which makes many investigations much easier. In the second case, the 
\emph{interior case,} the metric will be changed conformally 
by multiplication with $r^{-2}$, where~$r$ is a ``smoothed distance function,'' 
(a concept that will be defined in this section). 
The main technical result of this section is the behavior of ``smoothed
distance functions'' when performing iterated blow-ups, Proposition 
\ref{prop.smoothness}. We also recall in this section the needed regularity 
result on Lie manifolds from \cite{sobolev}.

\subsection{Smoothed distance function to a \psbmanifold}\label{ssec.s.dist}
Let $M$ denote a manifold with corners, as before. We assume that 
$\maP \subset 2^M$ -- equipped with a suitable ordering -- is a  blow-up-suitable 
$k$-tuple of subsets of $M$ (so that $[M: \maP]$ is defined). As always, we shall 
write
\begin{equation*}
   \unionP \ede \bigcup\limits_{P \in \maP} P\,.
\end{equation*}
We shall need the following simple concept of ``equivalent functions''
on $[M: \maP]$.

\begin{definition}\label{def.equiv.funct} Let $M$ be a manifold with
corners, let $\maP$ be a blow-up-suitable $k$-tuple
of \psubmanifolds\ of $M$ (so, in particular, $[M: \maP]$ is defined), and 
let $k\in \NN\cup\{0,\infty\}$. Assume that we have two continuous functions 
$f_i : M \to [0,\infty)$, $i = 0, 1$, such that 
the functions $f_i$ are $\maC^k$ on $M \smallsetminus \bigcup \maP$ and nowhere 
vanishing. We shall say that $f_0$ and $f_1$ are \emph{$\maC^k$-equivalent 
(on $[M: \maP]$)} or that \emph{$f_0$ is $\maC^k$-equivalent (on $[M: \maP]$) 
to $f_1$} if
\begin{equation*}
  \frac{f_1\stelle{M \smallsetminus \bigcup \maP}}
  {f_0\stelle{M \smallsetminus \bigcup \maP}}
\end{equation*}
extends to a nowhere vanishing  $\maC^k$ function on $[M: \maP]$.
In the case $k=\infty$, we shall say that $f_0$ is
\emph{smoothly equivalent} or just \emph{equivalent} to $f_1$ and write $f_0 \sim f_1$. 
In the case $k=0$, we say $f_0$ and $f_1$ are \emph{continuously equivalent.}
\end{definition}

The family $\maP$ as well as the space $[M: \maP]$, which we assumed to be
defined, will sometimes be understood, so we may occasionally not mention them.
Recall that two functions $f_1, f_2: U \to [0, \infty)$
are called \emph{Lipschitz equivalent} if there exists $C > 0$ such that
\begin{equation}\label{eq.Lip.equiv}
   C^{-1} f_1 \le f_2 \le C f_1 \,.
\end{equation}
Two functions $f_1, f_2: M \to [0, \infty)$ are called 
\emph{locally Lipschitz equivalent} if any point in~$M$ has an open neighborhood 
$U$ such that $f_1|_U$ and $f_2|_U$ are Lipschitz equivalent.

\begin{remark} Let us record a few easy consequences of the definition:
  \begin{enumerate}[(i)]
    
\item Clearly, the $\maC^k$-equivalence of functions on $[M: \maP]$ is an equivalence
  relation.
\item Since $M \smallsetminus \unionP$ is dense in $[M: \maP]$
(Proposition~\ref{prop.density}) the extension of the quotient
$\dfrac{f_1 \vert_{M \smallsetminus \bigcup \maP}}{f_0 \vert_{M \smallsetminus 
    \bigcup \maP}}$ to $[M: \maP]$ by continuity is unique (when it exists).
\item Obviously $\maC^k$-equivalence implies $\maC^\ell$-equivalence for $\ell\leq k$.  
\item If two functions $f_i : M \to [0,\infty)$ as above are continuously 
equivalent, then they are locally Lipschitz equivalent.
As a consequence, for $M$ compact, $\maC^0$-equivalence implies Lipschitz equivalence. 
\item If $H$ is a boundary hyperface of $M$ with at least one defining
function, then any two defining functions of~$H$ will be equivalent on $[M: H] \simeq M$.
\end{enumerate}
\end{remark}

Let us now discuss the type of metrics that we will use.
Let $M$ be a manifold with corners. In the following, we shall make use of two 
kinds of metrics on $M$. The first kind of metrics will consist of 
what we are calling ``true 
Riemannian metrics'' on $M$. A \emph{true Riemannian metric} on $M$ is, by definition, 
nothing but a smooth, fiberwise positive definite symmetric section of 
$T^*M\otimes T^*M\to M$ (the usual kind of metrics on $M$). The second kind of metrics will
consist of the so-called \emph{compatible metrics,} which will be defined below, Definition
\ref{def.comp.metric}, and whose definition requires some additional data (a structural
Lie algebra of vector fields on~$M$). Compatible metrics are Riemannian metrics 
on $\interior{M} \coloneqq  M\smallsetminus \partial M$ and do not extend to a true metric on $M$, 
unless $\partial M=\emptyset$.

Let $P$ be a closed \psubmanifold\ of $M$. Recall that $\beta_{M, P} : [M: P] \to M$ 
denotes the blow-down map (see Definition \ref{def.blow-up}).

\begin{definition} \label{def.r_X}
Let $M$ be a manifold with corners and $P \subset M$ 
be a closed \psubmanifold. A function $r_P : M \to [0, \infty)$ will be called 
\emph{a smoothed distance function to $P$ (in $M$)} if its lift 
$\beta_{M, P}^*(r_P) \coloneqq  r_P \circ \beta_{M, P}$ is a boundary defining
function (Definition \ref{def.b.def.f}) for $\SS N_+^M P = \beta_{M, P}^{-1}(P)$, 
the hyperface (or union of hyperfaces) of $[M: P]$ obtained by blowing up $M$ along~$P$. 
\end{definition}

We also remark that a smoothed distance function to $P$ in $M$ is continuous since 
it lifts to a continuous function on $[M: P]$ and $M$ has the quotient topology.
The following remark explains the name ``smoothed distance function to $P$'' in 
$M$ for $r_P$.

\begin{remark} \label{rem.def.r_X}
The function $f : M \to [0, 1]$ is a smoothed distance function to $P$ in 
$M$ if, and only if, it satisfies the following conditions:
\begin{enumerate}[(i)]
 \item\label{cond.i} it is continuous on $M$ and smooth on $M\setminus P$, 
 \item $f^{-1}(\{0\})=P$, and most importantly,
 \item\label{cond.iii} there is a neighborhood $V$ of $P$ such that 
 $f$ is (smoothly) equivalent on $[V: P]$ to 
 the distance to $P$ with respect to some suitable true metric on $M$.
\end{enumerate}
We omit details here, as this fact 
will not be used in our article.
\end{remark}

\begin{remark} \label{rem.r_X.one}
In general, a smoothed distance to~$P$ in $M$ function $r_P$ will \emph{not}
be smooth on $M$; in fact, smoothness on $M$ will only hold if $P$ is empty, 
or a union of connected components of $M$. However, it will be continuous on 
$M$ and smooth on $[M: P]$. This is one of the main reasons for considering 
the blow-up $[M:P]$. 
\end{remark}

\begin{remark}\label{rem.r_X.two} Given a closed \psubmanifold\ $P \subset M$,
  we see that a smoothed distance to $P$ in $M$ (Definition \ref{def.r_X})
is \emph{not uniquely determined}.
However, any two such smoothed distance functions are equivalent
on $[M: P]$, and hence they are equivalent on any other iterated blow-up $[M: \maP]$ 
as well, as long as $P \in \maP$. This is, in fact, our motivation for introducing 
the notion of equivalence 
of such functions. For this reason, we will sometimes talk about \emph{the} smoothed 
distance to $P$, although we will really mean the equivalence class of the
smoothed distances to $P$ in $M$.
\end{remark}

\begin{remark}
In the case that $M=\overline{X}$ is the spherical compactificaction of some Euclidean 
vector space $X$ as explained in Subsection~\ref{subsec.basic-notions} and if $Y$ is a 
linear subspace of $X$, we may consider the Euclidean distance function 
$d_Y := \min\bigl\{\|x-y\|\mid y\in Y\bigr\}$ introduced in Equation 
\eqref{eq.def.dist}. If $r_{\overline{Y}}: \overline{X}\to [0,\infty)$ is a 
\emph{smoothed} distance function to 
$\overline{Y}$, then the function $r_{\overline{Y}}$ is clearly not Lipschitz 
equivalent to~$d_Y$, as the first one is bounded and the second one is not.  
One can also show that $r_{\overline{Y}}$ is also not bi-Lipschitz 
to $\arctan \circ d_Y$, as these two functions behave differently close to $\SS_X$.
The precise behavior of $r_{\overline{Y}}$ is that it is continuously equivalent 
and Lipschitz equivalent to $x\mapsto \arctan\bigl((\|x\|^2+1)^{-1/2}d_Y(x)\bigr)$.
Furthermore, $r_{\overline{Y}}$ is continuously equivalent and thus locally Lipschitz 
equivalent to $x \mapsto (\|x\|^2+1)^{-1/2}d_Y(x)$.
\end{remark}

We now consider as in \cite{AMN1} and \cite{Kottke-Lin} pairs $(P, Q)$ of closed \psubmanifolds\
of $M$. As in those papers, we need to consider the cases $Q \subset P$ and $Q \cap P = 
\emptyset$. We begin with the first case.

\begin{lemma} \label{lemma.Y:P}
Let $P \subset M$ be a closed \psubmanifold{} of the manifold with corners $M$, let  $Q \subset P$ 
be a closed \psubmanifold\ of $P$, and let $r_P$ (respectively, $r_Q$) be a smoothed distance 
function to~$P$ (respectively, to~$Q$) in $M$.
Then
\begin{equation*}
   (r_Q^{-1} r_P)\stelle{M \smallsetminus P}
\end{equation*}
extends to a smoothed distance function to $\beta_{M, Q}^*(P) \coloneqq  
\overline{\beta_{M, Q}^{-1}(P \smallsetminus Q)}$ in $[M: Q]$.
\end{lemma}

Of course, in the above lemma, we have  $\beta_{M, Q}^*(P) \coloneqq  
\overline{\beta_{M, Q}^{-1}(P \smallsetminus Q)} \simeq [P:Q]$ see
Proposition \ref{prop.beta.m1}. Here $\simeq$ denotes the existence of a 
diffeomorphism extending the identity on $P \setminus Q$. Also, for the simplicity of 
the presentation, we may and will assume  in the following that $\beta_{M, Q}^*(P) = [P:Q]$.

\begin{remark}
Note that the  assumptions in the lemma allow the case when, for all 
$q\in Q$, we have $\dim_q (Q) = \dim_q (P)$. In this case  $\beta_{M, Q}^*(P)=P\setminus Q$ 
is a union of some connected components of~$P$, and thus for any $q\in Q$ we have
$(\beta_{M,Q})^{-1}(q)\cap [P:Q]=\emptyset$. If $U$ is an open neighborhood of $q$, 
then any positive function defined on $\beta_{M,Q}^{-1} (U)$ is thus a smoothed distance 
function to $[P:Q]\cap \beta_{M,Q}^{-1} (U)=\emptyset$.

On the other hand, in this special case, $r_Q^{-1}r_P$, defined on $U\setminus\{q\}$, 
where $U$ is an open 
neighborhood of $q$ with $U\cap P=Q$, extends to a positive smooth function on 
$\beta_{[M:Q], [P: Q]}^*(\beta_{M,Q}^* (U))=\beta_{M,Q}^* (U)$. The latter statement 
follows as $r_P|U$ is a smooth distance function to $Q$ in $U$ and using 
Remark~\ref{rem.r_X.one}. These argument provide a proof of the lemma  in this exceptional case.

This special case is also included in the following proof, if we use the convention $\SS^{-1}=\emptyset$
(here $\SS^{-1}$ is the unit sphere in $\RR^0$).
\end{remark}

\begin{proof}[Proof of Lemma \ref{lemma.Y:P}]
%
We need to prove that the function $r_Q^{-1}r_P : M\setminus P 
\to (0, \infty)$, extends to a smooth function on
\begin{equation*}
    [M: Q, P] \ede \bigl[[M:Q]: [P: Q]\bigr]
\end{equation*}    
and that this extension of $r_Q^{-1} r_P$ is a defining function
for the hyperface  
\begin{equation*}
     P^\prime \ede \SS\bigl(N^{[M: Q]}_+[P:Q]\bigr)  
      \seq \beta_{[M:Q], [P:Q]}^{-1}([P: Q])
\end{equation*}
of $[M: Q, P]$. Let 
\begin{equation*}
    \beta \ede \beta_{M, Q, P} \ede
    \beta_{M,Q}  \circ\beta_{[M:Q], [P: Q]} : 
    [M: Q, P] \to M
\end{equation*}
be the blow-down map and $z\in P^\prime$ 
(see Definition~\ref{def.iterated.bu}). Then $\beta (z) \in P$. 

Recall that by the definition of $r_P$, we have that 
$r_P \circ \beta_{M, P}$ is a boundary defining function 
$r_P \circ \beta_{M, P} \colon[M:P]\to [0,\infty)$
of $\SS(N^M_+P)
 = \beta_{M,P}^{-1}(P)\subset [M:P]$ as a hyperface of $[M: P]$.
The map $[M: Q, P] \to [M: P]$, defined by Lemma \ref{lemma.cancellation1}
is a diffeomorphism outside the preimage of $Q$. Thus, 
if $\beta(z) \notin Q$, the function $r_Q^{-1}r_P:M\setminus 
P = [M:Q]\setminus [P: Q] \to [0,\infty)$ extends locally -- i.e.\ in a neighborhood 
of $z$ -- to a defining function for $P^\prime\subset [M:Q,P]$, since 
$r_Q > 0$ at and near $z$ and since $r_P$ is a 
defining function of $\SS(N^M_P)$, the pull-back of $P$ 
in $[M: P]$, as we have just explained.
  
On the other hand, if $\beta(z) \in Q$, we use (again) the fact that our
problem is local. Thus by choosing a suitable chart around $\beta(z)$,
we can reduce the lemma to the special case
\begin{equation*}
\begin{cases}
  \ M  = \RR_{\ell}^{k} \times \RR_{\ell^\prime}^{k^\prime} \times 
  \RR_{\ell^{\prime\prime}}^{k^{\prime\prime}} \ni (x,x',x'') \\
  \ P  = \RR_{\ell}^{k} \times \RR_{\ell^\prime}^{k^\prime} \times \{0\} \ni (x,x',0)\\
  \ Q  = \RR_{\ell}^k \times \{0\} \times \{0\} \ni (x,0, 0)\\
  \ \beta(z) = (0, 0, 0) \,.
\end{cases}
\end{equation*}
Assume first
that $\ell^\prime = \ell^{\prime\prime} = 0$.
In this very special case we have
\begin{eqnarray*}
     [M:Q] &=& \RR_\ell^k\times \SS^{k'+k''-1} \times [0,\infty)\ni (x,\xi,r),\\{}
     [P:Q] &=& \RR_\ell^k\times \SS^{k'-1} \times [0,\infty),
\end{eqnarray*}
where $r=\sqrt{(x')^2+(x'')^2}$ and $\xi = (x', x'')/r$ away from $Q$. 
Furthermore, the inclusion $[P: Q] \subset [M: Q]$ is
given by the inclusion of $\SS^{k'-1} \subset \SS^{k'+k''-1}$ on the first $k^\prime$
components of $\SS^{k'+k''-1}$ and by the identity on the other factors (that is, on 
$\RR_\ell^k$ and on $[0, \infty)$).
For $\xi \in \SS^{k'+k''-1}$ let 
$\theta(\xi)$ be the length of the shortest geodesic from $\xi$ to $\SS^{k'-1}$, unless $k'=0$.
In the case $k'=0$ we set $\theta\equiv \pi/2$, thus $\sin\circ \theta\equiv 1$.
Then
\begin{equation*}
     \hat r \ede \sin \circ\, \theta\colon  \SS^{k'+k''-1}\to [0,1]
\end{equation*}
is a smoothed distance function to $\SS^{k'-1}$ in $\SS^{k'+k''-1}$, thus 
-- by pull-back -- $\hat r$ is also a smoothed distance function 
to $[P:Q]$ in $[M:Q]$. Smoothed distance functions $r_Q$ and 
$r_P$ (for $Q$ and $P$ in $M$) are then given by
\begin{equation*}
    r_Q (x,\eta,r) \seq \sqrt{\|x'\|^2+\|x''\|^2} \seq r \ \text{ and } \
    r_P \seq \|x''\| \seq r \sin(\theta(\xi))\,.
\end{equation*}
Since, obviously, $\hat r = r_Q^{-1}r_P$, we obtain the desired statement
if $\ell^\prime = \ell^{\prime\prime} = 0$. 
The general case follows by replacing $\SS^{k'-1}$ with $\SS_{\ell'}^{k'-1}$ and 
$\SS^{k'+k''-1}$ with $\SS^{k'+k''-1} \cap (\RR_{\ell'}^{k'} \times \RR_{\ell''}^{k''})$.
\end{proof}

We now turn to the second case, that when $P$ and $Q$ are disjoint.
For later use, we prove a more general statement.

\begin{lemma}\label{lemma.disjoint.sd}
Let $M$ be a compact manifold with corners, let $P \subset M$ be a closed 
\psbmanifold{} and $\maQ = (Q_1, Q_2, \ldots, Q_k)$ be a blow-up-suitable $k$-tuple 
of closed subsets of $M$ \emph{disjoint} from $P$. Let $\beta : [M: \maQ] \to 
M$ be the blow-down map and let $r_P$ be a smoothed distance to $P$
(in $M$). Let $\widehat{P} := \beta^{-1}(P) = \beta^*(P)$. Then
$r_P \circ \beta$ is a smoothed distance function to $\widehat{P}$
in $[M: \maQ]$.
\end{lemma}

\begin{proof}
Let us prove our result by induction on $k$. For $k =0$ (\ie for $\maQ = 
\emptyset$) there is nothing to prove according to our conventions for the 
blow-up with respect to an empty family. Let us write $Q = Q_1$, for the 
simplicity of the notation. Let us prove our result for $k = 1$. The blow-down 
$[M: Q, P] \to [M: P]$ induces a diffeomorphism $[M: Q, P] \smallsetminus 
\beta^{-1}(Q) \to [M \smallsetminus Q: P]$ (see Lemma \ref{lemma.cancellation2}).
The function $r_P \circ \beta_{M, Q}$ is a smoothed distance function to
$\beta^{-1}(P)$ on $M \smallsetminus Q \subset [M: Q]´$ (since it coincides with $r_P$
there). Moreover, $r_P \circ \beta_{M, Q}$ is smooth everywhere on the
iterated blow-up $[M: Q, P]$ and vanishes only on the preimage of $P$
(in particular, it is $>0$ on the preimage of $Q$, which is disjoint
from the preimage of $P$). Hence $r_P \circ \beta_{M, Q}$ 
is a smoothed distance function to $\beta_{M, Q}^*(P) = \beta_{M, Q}^{-1}(P)$ 
in $[M: Q]$. The induction step is always to perform an additional blow-up,
so it reduces to the case $k = 1$ just proved.
\end{proof}

\subsection{Smoothed distance functions to a family}

We need to extend the definition of the smoothed distance function 
to families. (The reader may want to review Definition~\ref{def.iterated.bu} 
at this point.)

\begin{definition} \label{def.rho_maS}
Let $\maP \coloneqq  (P_i)_{i=1}^k$, $P_i \subset M$ be a blow-up-suitable 
$k$-tuple (that is, an ordered family of closed subsets of $M$ such that the 
iterated blow-up $[M: P_1, P_2, \ldots, P_k]$ is defined, Definition \ref{def.bu.suitable}).
Let $\beta : [M: P_1] \to M$ be the blow-down map and $\maP^\prime \coloneqq  
\bigl( \beta^*(P_2), \beta^*(P_3), \ldots, \beta^*(P_k) \bigr)$, which,
we recall, is a family of closed \psubmanifolds{} of $[M: P_1]$. We then define 
\emph{a smoothed distance function} $\rho_{\maP}$ to $\maP$ in $M$ to be a function 
$\rho_{\maP} : M \to [0, \infty)$ given by induction on $k$ by the formula
\begin{equation*}
    \rho_{\maP}(x) \ede
    \begin{cases}
      \ \ r_{P_1}(x) & \mbox{ if } k = 1 \\
      \ r_{P_1}(x) \rho_{\maP^\prime}(y) & \mbox{ if } k > 1 \ \mbox{ and } 
      y \in \beta^{-1}(\{x\}) \,,
    \end{cases}
\end{equation*}
where $r_{P_1}:M\to [0,\infty)$ is a smoothed distance function to $P_1$ 
in $M$ (Definition \ref{def.r_X}) 
and $\rho_{\maP'}$ is a smoothed distance function to $\maP'$
in $[M: P_1]$.
(We note that the last expression is well-defined since, for $x\notin P_1$,  $y$ is unique 
whereas, for $x\in P_1$, we have $r_{P_1}(x)=0$.)
\end{definition}

Again, we use the notation introduced in Definition~\ref{def.iterated.bu} (which is, in turn,
the same as the one in \citeschrSpace{Definition}{def.iterated.bu}). We have then, 
similarly to \citeschrSpace{Remark}{rem.cond.ex}.

\begin{remark}\label{rem.cond.ex.bis}
We have that $\rho_\maP$ is continuous on
$M$ and that $\rho_\maP^{-1}(\{0\})= \unionP = \bigcup_{j=1}^k P_j$.\
Let $\gamma_0 = \id$, $\gamma_1 \coloneqq  \beta^{*}_1$
and $\gamma_j \coloneqq  \beta^{*}_j \circ
\gamma_{j-1} = \beta^{*}_j \circ ... \circ \beta^{*}_1$, where
\begin{equation*}
  \beta_k \ede \beta_{[M:P_1,\ldots, P_{k-1}], [P_k:P_1,\ldots, P_{k-1}] } : [M :
    P_1, \ldots, P_k] \to [M : P_1, \ldots, P_{k-1}] \,.
\end{equation*}
If,  $x \notin \unionP$, then
\begin{equation*}
   \rho_{\maP}(x) \seq r_{\gamma_0^*(P_1)} (x) r_{\gamma_1^*(P_2)} (x) 
   \ldots r_{\gamma_{k-1}^*(P_k)} (x)\,.
\end{equation*}
In particular, it follows that the pull-back of $\rho_{\maP}$ to  
$[M :\maP]$ is smooth, or equivalently, suppressing all pull-backs from the notation, 
$\rho_{\maP}\in \maC^\infty\bigl([M :\maP]\bigr)$. However note, that in contrast to 
the non-iterated case,  $\rho_{\maP}\in \maC^\infty\bigl([M :\maP]\bigr)$ is no longer 
a boundary defining function for some boundary hypersurface of $[M :\maP]$.
\end{remark}

\begin{remark}\label{rem.rho_maF} As in Remark \ref{rem.r_X.two},
a smoothed distance to $\maP$ function $\rho_\maP$ (see Definition \ref{def.rho_maS})
is not unique, but it is unique \emph{up to equivalence} on $[M: \maP]$.
Indeed, this follows from Remark \ref{rem.cond.ex.bis} since all the factors 
$r_{\gamma_{k-1}^*(P_k)}$ are uniquely determined up to equivalence on $[M: \maP]$
and the equivalence is compatible with products and preserved when we increase~$\maP$.
\end{remark}

The following result extends to smoothed distance functions the
compatibility with products of Proposition \ref{prop.lemma.product}.

\begin{proposition}\label{prop.distance.product} 
Let $M$ and $M_1$ be two manifolds with corners and $\maP = (P_i)_{i=1}^k$ be a
blow-up-suitable $k$-tuple of subsets of $M$. Then
$$\maP \times M_1\coloneqq\left( P_j\times M_1\right)_{j=1}^k$$
is a blow-up-suitable $k$-tuple of subsets of $M \times M_1$. Let $r_\maP$ be 
a smoothed distance to $\maP$ in $M$ and $p : M \times M_1 \to M$ the projection.
Then $r_\maP \circ p$ is a smoothed distance function to
$\maP \times M_1$ in $M \times M_1$.
\end{proposition}

\begin{proof}
If $H$ is a hyperface of $M$ and $r_H$ is a defining function for
$H$ in $M$, then $r_H \circ p$ is a defining function for the hyperface
$H \times M_1$ in $M \times M_1$. The result follows from definitions 
by induction on $k$ using repeatedly this observation. Indeed, there is nothing 
to prove if $k = 0$. The case $k = 1$ follows from Proposition \ref{prop.lemma.product} (used
also for $k = 1$) and the observation about hyperfaces. 
The induction step is completely similar to the case $k=1$ and it reduces to
that one, as in the proof of Proposition \ref{prop.lemma.product}, using that proposition,
the relation $\beta^*(P \times M_1) = \beta^*(P) \times M_1$, and
the definition of smoothed distance functions.
\end{proof}

From now on, we shall assume that our $k$-tuple is a clean semilattice
of closed \psbmanifolds{} of $M$,
endowed with an admissible order. Also, from now on we shall denote our $k$-tuple
with $\maS$ instead of $\maP$ in order to stress that it is stable for
intersections (\ie that  $\maS$ is a semilattice). (Recall that we can choose any 
admissible order on $\maS$.) The following proposition will play a crucial role in 
our application to $N$-body type problems.

\begin{proposition} \label{prop.smoothness}
Let $\maS$ be a clean semilattice of closed \psubmanifolds{}
of a connected manifold with corners $M$ and $\emptyset \neq Y \in \maS$.
Let $r_Y$ be a smoothed distance function to $Y$ in $M$ (Definition \ref{def.r_X}) 
and $\rho_\maS$ be a smoothed distance function to $\maS$ in $M$ 
(Definition \ref{def.rho_maS}). Then the function 
$\rho_{\maS}/r_Y : M \smallsetminus \unionma{S} \to (0, \infty)$ extends 
to a smooth function on $[M: \maS]$.
\end{proposition}

Note that in the statement of the proposition, it is important that $r_Y$ 
is a smoothed distance in $M$ and not in some blow-up of $M$. The proposition would 
be a trivial consequence of the recursive definition of $\rho_\maS$ if we replaced 
$r_Y$ by a smooth distance function in some suitable blow-up of $M$.

\begin{proof} 
There is no loss of generality to assume that $\emptyset \in \maS$,
so we will do so for brevity. (The case $\emptyset \notin \maS$ is
completely similar.) We shall prove the statement by induction on the number of elements
of $\maS$.

If $\maS$ has only one non-trivial element, that is, if 
$\maS = (\emptyset,Y)$, then we have  $\rho_{\maS} \simeq r_Y$,
by the definition of $\rho_{\maS}$, as explained in Remarks \ref{rem.r_X.two} and 
\ref{rem.rho_maF}.  

Now, for the induction step, let us arrange 
$\maS = (P_0  = \emptyset, P_1, P_2, \ldots, P_k)$ in an admissible order.
Let us blow up along
$P_1 \in \maS$ (which is a minimal element of $\maS \smallsetminus \{\emptyset\}$,
by the definition of an admissible order).
We use the notation introduced in Definition \ref{def.rho_maS}, in particular~~$\maS'$ \sout{$\maS''$} consists of 
the lifts to $[M: P_1]$ of the \psubmanifold s $P_2,\cdots, P_k$. If $Y = P_1$,
then the result follows from the formula for $\rho_{\maS}$ in Remark~\ref{rem.cond.ex.bis}.
Let us assume therefore that $Y \neq P_1$. Since $P_1$ is a minimal (non-empty) element
of~$\maS$, we have two possibilities: either $P_1 \subset Y$ or $P_1 \cap Y = \emptyset$.

Let us assume that we are in the case $P_1 \subset Y$. 
We choose then as a smoothed distance for $[Y: P_1]$ in $[M: P_1]$
the function $r_{[Y: P_1]}^\prime \coloneqq  r_Y/r_{P_1}$, which is possible
in view of Lemma~\ref{lemma.Y:P}. By the induction hypothesis, we know that
$(r_Y/r_{P_1})^{-1} \rho_{\maS^\prime}$ extends to a smooth function
on $\bigl[[M:P_1]:\maS^\prime\bigr]$.

Since $[[M:P_1]:\maS^\prime] = [M: \maS]$, by the definition of the iterated
blow-up, by suppressing pull-backs in notation we obtain that 
\begin{equation}
   r_Y^{-1} \rho_{\maS} \seq r_Y^{-1} r_{P_1} \rho_{\maS^\prime} 
   \seq (r_Y/r_{P_1})^{-1} \rho_{\maS^\prime} \in \CI([M: \maS])\,.
\end{equation}
This equation should be understood in the sense of functions $M\setminus\bigcup\maS\to \RR_{+}$ 
that extend to smooth functions on $[M: \maS]$. Thus, we have obtained 
the desired relation, so the proof is complete in the case $P_1 \subset Y$.

Let us assume now that $P_1 \cap Y = \emptyset$. The proof is then an
immediate application of Lemma~\ref{lemma.disjoint.sd} for $k = 1$.
\end{proof}

The following remark recalls \cite[Lemma~3.17]{ACN}.

\begin{remark}\label{rem.useful}
Recall that in this and in the following sections,
$\maS$ is a clean semilattice of closed \psbmanifolds{} of~$M$.
Let $\rho(x) \coloneqq  \dist_{\overline g}(x, \maS)
\coloneqq \dist_{\overline{g}}(x,\unionS)$ be the 
distance to $\maS$ in some \emph{true} metric $\overline g$ on $M$. 
(That is, a metric on $TM$ that extends smoothly to the boundary of~$M$,
unlike a ``compatible metric,'' a concept that will be introduced
shortly and which is really only a metric on $M \smallsetminus \pa M$ 
with some additional properties.) One can show that the functions $\rho_{\maS}$ 
and $\rho$ are continuously equivalent.
A special case of this statement which is in fact the only case needed in 
the present article was proved in  \cite[Lemma~3.17]{ACN} based on preliminary 
work in \cite{BMNZ}. In \cite{ACN} it is additionally assumed that all  
\psubmanifolds{} in~$\maS$ are \emph{interior submanifolds}. By definition, 
an interior submanifold is a  \psubmanifold{} of  positive
boundary depth, \ie no component is contained in $\partial M$. However, a proof 
may be extended to the full statement claimed above; the result extends to 
\psubmanifolds{} that are not interior, either by adapting the proofs or by 
embedding the manifold~$M$ into a Riemannian manifold~$N$ without corners and 
boundary, such that all boundary hypersurfaces of~$M$ are totally geodesic.
The function $\rho$, however is not smooth on $M \smallsetminus \unionma{S}$, 
in general. This justifies the notations $\rho$ and $\rho_\maS$, since they are
very similar in scope. See Lemma~\ref{lemma.Lie.le} for an extension of this remark.
\end{remark}

\subsection{Blow-ups and Lie manifolds}
We now look at a class of differential operators that will ``desingularize''
the $N$-body Hamiltonian (with Coulomb or inverse square potentials). 
Let thus $M$ be a manifold with corners and $\maV \subset \CI(M; TM)$ be a subspace
of smooth vector fields on $M$ that is stable for the Lie bracket and for multiplication 
with functions in $\CI(M)$. We then let $\Diff_\maV^m(M)$ denote the space of 
differential operators of order $\le m$ on $M$ generated by $\CI(M)$ and by $\maV$
and 
\begin{equation}\label{eq.def.diffV}
  \Diff_\maV(M) \ede \bigcup_{m\in \NN_0} \Diff_\maV^m(M) \seq 
  \Bigl\{\, \sum a V_1 V_2 \ldots V_k \, 
  \vert \ a \in \CI(M), V_j \in \maV \, \Bigr\}\,.
\end{equation}
Furthermore, we will assume that $(M, \maV$) is a ``Lie manifold,'' a
concept from \cite{aln1} that we now recall.

\begin{definition} \label{def.Lie.Man}
Let $M$ be a smooth, compact manifold with corners and $\maV \subset
\CI(M; TM)$ be a subspace of smooth vector fields on~$M$. In addition to
the relations $\CI(M)\maV = \maV$ and $[\maV , \maV] \subset \maV$,
let us also assume that:
\begin{enumerate}[(i)]
    \item $\CIc(M \smallsetminus \pa M; TM) \subset \maV$;
    \item all vector fields in $\maV$ are tangent to the boundary; and
    \item\label{def.Lie.Man.iii} $\maV$ has a $\CI(M)$-local basis near each point.
\end{enumerate}
  
Then we shall say that $(M; \maV)$ is a \emph{Lie manifold}; the space $\maV$ 
will be called its associated \emph{structural Lie algebra of vector fields}.
\end{definition}

\begin{remark}\
The condition~\eqref{def.Lie.Man.iii} in the above definition
means that any $p\in M$ has a closed neighborhood $U$  such that the restriction of 
$\maV$ to $U$ is a free $\CI(U)$-module. An even better way of expressing 
condition~\eqref{def.Lie.Man.iii} is to assume the existence of a vector bundle 
$A\to M$ together a vector bundle morphism $\varrho: A \to TM$ over $\id_M$ with 
$\varrho(\Gamma(A)) = \maV$. The morphism $\varrho$ is called an \emph{anchor}
map, see \cite{aln1, CamilleBook} and the references therein. 
Then Condition~(i) is equivalent to saying that
$\varrho$ restricts to a vector bundle isomorphism on $\interior M$ $ \ede M 
\smallsetminus \pa M$. Similarly, Condition~(ii) is equivalent to saying that,
for any $p\in \partial M$ of boundary depth $1$, we have 
$\varrho(A_p)\subset T_p(\partial M)\subsetneq T_pM$. This condition 
then implies analogous tangency conditions for points of boundary depth 
$\geq 2$. 
\end{remark}

The following simple and basic example(s) will play a crucial role in our
applications to the $N$-body problem.

\begin{example}\label{ex.basic.LM.1}
We continue to assume that $X$ is a finite-dimensional real vector space and 
and that $\oX$ is its spherical compactification. We shall take $\oX$ for our
ambient manifold (thus $\oX$ plays the role of the manifold typically called 
$M$ so far). Then $X$ identifies with the space of constant vector fields on itself.
Corollary~\ref{cor.ext.zero} tells us that these vector fields extend to vector 
fields on $\oX$ and that these extended vector fields vanish on $\SS_X$. 
We shall consider then the structural Lie algebra of vector fields 
$\maV \coloneqq  \CI(\oX) X$. Then the pair $(\oX; \maV)$ is a Lie manifold 
(Definition \ref{def.Lie.Man}).
Indeed, $\maV$ is a free $\CI(\oX)$--module since any basis of $X$, viewed as a vector space, 
also provides a basis of $\maV$ as a $\CI(\oX)$-module. The resulting algebra of differential 
operators $\Diff_\maV(\oX)$ is the algebra of differential operators on~$X$ with coefficients
in $\CI(\oX)$.
\end{example}

To a Lie manifold $(M; \maV)$ there is canonically
associated a class of metrics on the interior of $M$, called \emph{$\maV$-compatible 
metrics} on $\interior M = M \smallsetminus \pa M$, as follows:

\begin{definition}\label{def.comp.metric}
Let $(M, \maV)$ be a Lie manifold, $n = \dim(X)$.
A \emph{$\maV$-compatible metric on~$M$} is defined as a Riemannian metric on 
 $\interior M$ with the following property: each 
$x \in M$ has a closed neighborhood $U_x$ such that the restriction of $\maV$ 
to $U_x$ has a $\CI(U_x)$-basis $(v_1, v_2, \ldots, v_n)$ that is orthonormal in the sense
  $$\forall y\in \interior M\cap U_x:\;(v_1, v_2, \ldots, v_n) \mbox{ is an orthonormal basis of}(T_yM,g_y).$$
(Of course, this is a true restriction 
on $g$ and $\maV$ only at the boundary of $M$).
\end{definition}

A $\maV$-compatible metric on $M$ will never be a true metric on $M$ since
the induced metric on $T\interior{M}$ will not extend continuously to $M$.
For instance, for Example \ref{ex.basic.LM.1} (and the related Example \ref{ex.basic.LM.2}
relevant for the $N$-body problem), 
any Euclidean metric on $X$ will be a $\maV$-compatible
metric on $M$.

\begin{remark}\label{rem.alt.descrip.compat.met}
An alternative way to define a \emph{$\maV$-compatible} is as follows. 
We briefly recall this, as it will be used in Appendix~\ref{app.a}. See  
\cite{aln1} for a proof of the equivalence of the definitions.
Let $(M,\maV)$ be a Lie manifold with anchor map $\varrho : A \to TM$ as 
defined above.  A compatible metric may be defined as a metric on $A$, which 
is by definition a symmetric positive definite section $\gamma\in \Gamma(A^*\otimes A^*)$. 
The anchor map $\varrho$ induces a map $\rho_*=\rho\otimes \rho\colon A\otimes A\to TM\otimes TM$, 
and on $\interior{M}$ the inverse of the anchor map, namely $\left(\varrho|_{\interior{M}}\right)^{-1}$, 
may be dualized to a pull-back map  $\left(\left(\varrho|_{\interior{M}}\right)^{-1}\right)^*\colon A^*|_{\interior{M}}\to T^*\interior{M}$. Then $g\coloneqq \left(\left(\varrho|_{\interior{M}}\right)^{-1}\right)^*\otimes \left(\left(\varrho|_{\interior{M}}\right)^{-1}\right)^*(\gamma)$ is a 
Riemannian metric on $\interior{M}$. A Riemannian metric on  $\interior{M}$ is  
\emph{$\maV$-compatible} if, and only if it arises this way.

If $B \in V^* \otimes V^*$ is a symmetric 
bilinear form on a finite-dimensional vector space $V$, then 
we denote $\flat_B : V\to V^*$, $X\mapsto B(X,\argu)$ the associated
canonical homomorphism. If $B$ is non-degenerate, $\flat_B$ is invertible and one defines
$\sharp_B\coloneqq \flat_B^{-1}$. In this case we can define $\doublesharp{B}\coloneqq 
B(\sharp_B(\argu),\sharp_B(\argu))\in V \otimes V$  whose matrix with respect to some 
basis of $V$ is the inverse of the matrix of $B$ with respect to its dual basis. 
Now, if $\gamma$ is a metric on the vector bundle $A$, we map apply this construction for any 
$q\in M$ to $\gamma|_q\in A_q^*\otimes A_q^*$, and we obtain a smooth section $\doublesharp{\gamma}\in 
\Gamma(A\otimes A)$. Then $G\coloneqq \bigl(\rho\otimes \rho\bigr)( \doublesharp{\gamma})\in \Gamma(TM\otimes TM)$ is 
a well-defined smooth tensor. One can easily show that $G$ is the unique continuous extension of 
$\doublesharp{g}$ for the $g$ defined above. We will use this description to proof in Appendix~\ref{app.a} 
for a true metric $\ol g$ and a compatible metric $g$ we have, see Lemma~\ref{lemma.Lie.le}, a 
constant $C$ with $\overline g \le C g$ in the sense of symmetric forms, \ie $Cg-\ol g$ is 
positive semi-definite.
\end{remark}
 Any two $\maV$-compatible metrics $g$ and $\tilde g$ are Lipschitz 
equivalent, \ie there exists a constant $C>0$ with $C^{-1} 
 \tilde g\leq g \leq C \tilde g$, again in the sense of symmetric forms.
Thus the volume forms are Lipschitz equivalent as well, and hence the 
definition of the space $L^p(M,g) = L^p(M; \maV)$
only depends on $\maV$, and not on the metric. This allows us to 
define the associated \emph{Sobolev spaces} on $X$ (or $M$) as 
\begin{equation}\label{eq.def.Wkp}
        W^{k,p}(M; \maV) \seq \bigl\{u \,\bigm|\; V_1 \ldots V_j
        u \in L^p(M; \maV),\, \forall V_1, \ldots, V_j \in \maV, \,
        j \in\{0,1,\ldots, k\}\bigr\}\, 
\end{equation}
where $k\in \NN_0$ and $1\leq p\leq \infty$. Obviously, they do not depend on the choice of 
$\maV$-compatible metric. These Sobolev spaces coincide with the standard Sobolev spaces on 
the complete Riemannian manifold associated with any compatible metric.

The main reason we are interested in Lie manifolds is that there is also a notion 
of an associated pseudodifferential calculus \cite{aln2}, a notion of
ellipticity in $\Diff_\maV(M)$, and, most importantly for this paper, an elliptic 
regularity result \cite{sobolev}. Explicitly, an operator in $\Diff_\maV^m(M)$ is 
\emph{elliptic} in $\Diff_\maV^m(M)$ if its principal 
symbol is uniformly elliptic in one (equivalently, any) $\maV$-compatible metric.

The following theorem is the special case $s=0$ and $m\in \ZZ_+$ of \cite[Theorem 8.7]{sobolev}.

\begin{theorem}[Ammann-Ionescu-Nistor]\label{theorem.ain}  
Let $(M; \maV)$ be a Lie manifold. Let $k, j \in \ZZ$, $m \in \ZZ_+$, 
$1 < p < \infty$, and $P \in \Diff_\maV^m(M)$ be elliptic. Let $u \in W^{k,p}(M; \maV)$
be such that $P u \in W^{j,p}(M; \maV)$. Then $u \in W^{j+m,p}(M; \maV)$.
\end{theorem}

Let us mention also that for in Example~\ref{ex.basic.LM.1}
(and the related Example \ref{ex.basic.LM.2}
relevant for the $N$-body problem), an operator $P \in \Diff_\maV(M)$ is 
elliptic in that algebra if, and only if, it is uniformly elliptic (in the usual, 
Euclidean metric of $X$). In particular, the above regularity theorem 
(Theorem~\ref{theorem.regularity}) can also be obtained from the regularity result 
in~\cite{AGN1}. 

Let us now provide an alternative, more general method to obtain a Lie
manifold structure on the Georgescu--Vasy space $\XGV$. (See the Introduction
or Section~\ref{sec4} \eqref{eq.def.XGV} for the definition of $\XGV$.)

\begin{remark}\label{rem.boundary.Lie}
Let $(M; \maV)$ be a Lie manifold and $Y \subset M$ be a closed \psubmanifold.
We assume $Y \subset \pa M$, 
\ie $Y$ is everywhere of positive boundary depth.  We also assume that all 
$V\in \maV$ satisfy $V|_{Y}\in \Gamma(TY)$. The latter condition implies 
(due to a straightforward modification of \cite[Proposition~3.2]{ACN}) that 
any $V\in \maV$ has a \emph{lift} $\tilde V$, that is, a vector 
field on $[M:Y]$  such that
\begin{equation*}
\xymatrix{ [M:Y] \ar[r]^{\tilde V} \ar[d]_{\beta_{M,Y} }&
  T[M:Y] \ar[d]^{d\beta_{M,Y}} \\
  M \ar[r]^V & TM. }
\end{equation*}
commutes. As a consequence, the inclusion $\maV\hookrightarrow \Gamma(TM)$ ``lifts'' 
to a map $\maV\hookrightarrow \Gamma(T[M: Y])$, both maps are injective Lie algebra 
homomorphisms and $\maC^\infty(M)$-linear. We obtain a Lie manifold structure 
on $[M: Y]$ by taking as structural Lie algebra of vector fields 
\begin{equation}
   \maWb \ede \CI([M: Y]\,\maV\,.
\end{equation} 
Let us record here how the various quantities associated with the Lie manifold
$(M, \maV)$ change when going to the blow-up Lie manifold $([M: Y], \maWb)$:
\begin{itemize}
\item The interior smooth manifold is the same: $\interior{M} = M \smallsetminus \pa M 
= [M: Y] \smallsetminus \pa [M: Y]$.
\item If a metric on $M_0$ is $\maV$-compatible, then it is also $\maWb$-compatible.
\item $W^{k,p}(M; \maV) = W^{k,p}([M: Y]; \maWb)$.
\item $\Diff_{\maWb}([M: Y]) = \CI([M: Y])\Diff_\maV(M)$.
\end{itemize}
The last equality is based on the fact that the lift of $\maV$ acts on  
$\CI([M:Y])$ by derivations.
By iterating this construction, we can obtain similarly a Lie manifold structure on 
$[M: \maY]$, where $\maY$ is a clean semilattice 
of \psubmanifolds\ of $\pa M$. Indeed, let us assume that $\maV$ is tangent 
to each manifold $Y \subset \maY$, then $\maV$ will lift to all intermediate
blow-ups leading to $[M : \maY]$. At each step, by density, the lifting of $\maV$ will
be tangent to the manifold with respect to which we blow-up.
\end{remark}

One typically describes a Lie manifold via the local behavior of
the vector fields near the boundary points. This is not what we did in our
case (in the remark right above and the one following next). See \cite{CNQ, Rochon2, Rochon1}
for similar Lie manifold structures that are described locally.

\begin{example} 
Recall from the Introduction that $\maF$ is a finite semi-lattice of linear 
subspaces of $X$ such that $\{0\}\in \maF$ and $X \not \in \maF$.
The Lie manifold structure on $\XGV$ will be obtained starting from $\oX$ by
taking $\maY \coloneqq  \SS_\maF \coloneqq  \{\SS_Y \mid \, Y \in \maF\}$. 
Since the action of
$X$ on the boundary $\SS_X$ of $\oX$ is trivial, see Remark \ref{ex.basic.LM.1} and 
Appendix~\ref{app.spherical.compact}, the vector fields in
$\maV \coloneqq  \{fV\mid f\in \CI(\oX), V\in X\}$ are tangent to all the 
submanifolds of $\SS_\maF$.
\end{example}

The situation described in the last remark changes, however, dramatically
if $Y$ has boundary depth $0$, \ie if no connected component of $Y$ is contained
of the boundary. In that case, there is, in principle, more work to
be done, but that has already been completed to a large extent in \cite{ACN}. 
Let us summarize some required results from that article.

\begin{remark}\label{rem.interior.Lie}
We use the notation introduced in Remark \ref{rem.boundary.Lie}. Unlike there,
however, we shall consider, this time, the case when $Y$ is connected
and has boundary depth~$0$ (that is, it is not contained in the boundary 
$\pa M$ of $M$). Let $r_Y$ be a smoothed distance function to~$Y$,
Definition \ref{def.r_X}. We then obtain a Lie manifold structure on $[M: Y]$ 
as in \cite{ACN} by taking as structural Lie algebra of vector fields 
\begin{equation}\label{eq.Win}
   \maWint \ede r_Y \CI([M: Y])  \maV\,.
\end{equation}
See \cite[Lemma~3.8 and Theorem~3.10]{ACN} where it was proved that all vector 
fields in $r_Y \maV$ lift canonically to smooth vector fields on $[M: Y]$.
(This justifies the additional factor~$r_Y$, since without it, that would not cover
our case.) Let us record here how the various quantities associated with the Lie 
manifold $(M, \maV)$ change when going to the blow-up Lie manifold $([M: Y], \maWint)$:
\begin{itemize}
\item The interior smooth manifold of the blow up is, this time: 
\begin{equation*}
 M_1 \ede [M: Y] \smallsetminus \pa [M: Y]
 \seq  M \smallsetminus (\pa M \cup Y)\,.
\end{equation*}
\item If $g$ is a $\maV$-compatible metric on $M$, then 
$\tilde g\coloneqq r_Y^{-2}\beta_{M,Y}^*g$ is a
 $\maWint$-compatible metric on $[M: Y]$.
\item The volume forms for $g$ and $\tilde g$ are related by the formula
      $\dvol^{\tilde g}= r_{Y}^{-n}\beta_{[M:Y]}^*\dvol^g$, $n = \dim (M)$. 
      As a consequence we have 
\begin{equation*}
  u\in L^p([M:Y]; \maWint) \; \Longleftrightarrow \; r_{Y}^{-n/p}u\in L^p(M; \maV)
  \end{equation*}
\item The Sobolev spaces defined by $\maWint$ are ``weighted Sobolev spaces'' 
in the old metric (compare to Equation \eqref{eq.def.Wkp} and notice the factor 
$r_Y^j$):
\begin{eqnarray*} 
        W^{k,p}\bigl([M:Y], \maWint\bigr) &=& \Bigl \{ u \,\big|\; r_Y^{j} V_1 \ldots V_j
        u \in L^p\bigl([M:Y] ; \maWint \bigr),\, \forall V_1, \ldots, V_j \in \maV, \\ 
        && \phantom{\Bigl\{u\mid}  j \in\{0,1,\ldots, k\}\Bigr \}\, 
\end{eqnarray*}

\item 
Let $m \in \ZZ_+$, then
$\Diff_{\maWint}^m([M: Y])$ is the linear span of differential monomials of
the form $r_Y^j a V_1 V_2 \ldots V_j$, with $a \in \CI([M: Y])$, $V_1, V_2, 
\ldots, V_j \in \maV$, and $0 \le j \le m$. Moreover, if $D \in \Diff_{\maV}^m(M)$,
then $r_Y^m D \in \Diff_{\maWint}^m([M: Y])$ and, if $D$ is elliptic in 
$\Diff_{\maV}^m(M)$, then $r_Y^m D$ is elliptic in $\Diff_{\maWint}^m([M: Y])$.
\end{itemize}
The last two statements are also obtained using that $V (r_Y) \in \CI([M: Y])$ for all 
$V \in \maV$. (We note in passing that $V (r_Y) \notin \CI(M)$.) 
\end{remark}

As in Remark \ref{rem.boundary.Lie}, we can iterate the construction of 
Remark \ref{rem.interior.Lie}, to obtain similarly a Lie manifold structure on $[M: \maY]$,
where $\maY$ is a clean semilattice (\ie a cleanly intersecting semilattice)
of \psubmanifolds\ of $M$. All the resulting
objects on $[M: \maY]$ will be independent on the admissible order chosen on $\maY$.

\section{The Lie manifolds associated to the generalized $N$-body problem}\label{sec4}
We now investigate how the constructions of the previous section particularize to the
case of $N$-body problems. Let us recall our standing conventions that~$X$ is a finite 
dimensional, real vector space and that~$\maF$ is a finite semilattice of \emph{linear} 
subspaces of~$X$ satisfying $\{0\} \in \maF$ and $X \notin \maF$. (We stress, however,
that the assumptions $\{0\} \in \maF$ and $X \notin \maF$ are not essential, since
the general case of a finite semilattice $\maF$ easily reduces to this case.)

\subsection{Semilattices and blow-ups for the $N$-body problem}\label{subsec.semilat.b-up}
We begin by introducing the semilattices that we will use for the study of the 
$N$-body problem. These semilattices will then be used to introduce the associated 
blow-ups and Lie manifolds.

\subsubsection{The semilattices and the blow-ups}
\label{sec.two.bups}
We denote by $\{0\}$ the vector space  consisting of just $0 \in X$. 
We agree that $\SS_{\{0\}} = \emptyset$, and hence $\overline{\{0\}} = \{0\}$.
As in \cite{ACN, AMN1}, we shall consider the semilattices $\overline{\maF}
\coloneqq \{\oY \mid Y \in \maF\}$ and $\SS_{\maF} \coloneqq \{\SS_Y \mid Y \in \maF\}$,
recalled earlier in Equation \eqref{eq.def.semilattices}.
Note that $\emptyset \in \SS_\maF$ because $0 \in \maF$.  

\begin{proposition} \label{prop.N_clean}
The sets $\overline{\maF}$, $\SS_{\maF}$, $\SS_\maF \cup \overline{\maF}$
are \emph{clean semilattices} of \psubmanifolds\ of $\oX$.
\end{proposition}

\begin{proof}
In view of Remark \ref{rem.clean.intersection}, this result is obtained by 
a direct application of Lemma~\ref{lem.clean}.
\end{proof}

We fix an admissible order on $\SS_\maF$, but the constructions below do not 
depend on this choice. For instance, the \emph{Georgescu--Vasy space} 
\begin{equation}\label{eq.def.XGV}
   \XGV \ede [\oX : \SS_\maF]
\end{equation}
does not depend on the choice of an admissible order on $\SS_\maF$. (This space is the
same as the one introduced in \cite{AMN1} and in Equation \eqref{eq.blownup.spaces}.) 
We now introduce the Lie manifold structure on $\XGV$.

\begin{example}\label{ex.basic.LM.2}
Let $X$ act on $\oX$ by translations, as in Example ~\ref{ex.basic.LM.1}.
Let also $\maS$ be  a finite, clean semilattice of closed submanifolds of 
$\SS_X = \oX \smallsetminus X$. The Lie group action of~$X$ on~$\oX$ 
acts trivially on $\SS_X$, thus it preserves 
$\maS$, and therefore we obtain an action of $X$ on the blown-up space $[\oX:\maS]$
(where the blow-up is defined using any admissible order on~$\maS$). 
As explained in Lemma~\ref{lemma.vect.GV.extend}, 
this implies that any constant vector field on~$X$ extends to a smooth vector field 
on $M \coloneqq [\oX:\maS]$ that is \emph{tangent} to the boundary. Let 
$\maWb \coloneqq  \CI(M) X$.

As in Example ~\ref{ex.basic.LM.1}, then 
$(M; \maWb)$ is a Lie manifold (Definition~\ref{def.Lie.Man}). Indeed, again, 
$\maWb$ is a free $\CI(M)$--module, a basis being given by extensions of the canonical 
basis, viewed as constant vector fields. In this case, the resulting algebra 
of differential operators $\Diff_{\maWb}(M)$ is the algebra of differential operators 
on~$X$ with coefficients in $\CI(M)$. So the blow-up of $\oX$ by $\maS$ leads to 
a larger algebra of differential operators than the one associated with the
Lie algebroid $(\oX, \maV)$, $\maV \coloneqq \CI(\oX)X$, considered in 
Example~\ref{ex.basic.LM.1}. The class of differential operators to which our
regularity results apply, however, is even larger than $\Diff_{\maWb}(M)$
(it allows also for interior blow-ups).

For us, the relevant choice is $\maS = \SS_\maF$, which yields $[\oX: \SS_\maF] =: 
\XGV .$ (Note that $\XGV$ was defined using only boundary blow-ups.) We thus have 
obtained a \emph{Lie manifold structure on the Georgescu--Vasy space 
$\XGV$.} This class of examples appears implicitly (but prominently) in Georgescu's 
work \cite{Georgescu2018, GeIf06}.
\end{example}

We now proceed to further blow-up $\XGV$ with respect to $\ol\maF$, \ie the lifts of
elements of~$\maF$.

\begin{lemma}\label{lem.def.XF}
Let us arrange both $\SS_{\maF}$ and $\overline{\maF}$ in admissible orders
(see Definition~\ref{def.admissible}). We then use the resulting order
to obtain an order relation on $\SS_{\maF} \cup \overline{\maF}$ in which
all the elements of $\SS_{\maF}$ precede the ones of $\overline{\maF}$.
Then the resulting order on the union $\SS_{\maF} \cup \overline{\maF}$
is an admissible order.
\end{lemma}

\begin{proof}
Let $\overline Y \in \overline{\maF}$,
where $Y \in \maF$. Let also $\maP \subset \SS_{\maF}$, $\maP \neq \SS_{\maF}$, 
and $\beta : [\oX : \maP] \to \oX$ be the blow-down
map. Then $\beta^*(\overline{Y})$ will contain $Y$ and hence cannot
be a minimal element of $(\SS_{\maF} \cup \overline{\maF}) \smallsetminus \maP$
since it is not contained in any element of $\SS_{\maF} \smallsetminus \maP$.
\end{proof}

We shall need for the following result 
the following notation for the lifts of $\maF$ to $\XGV$.
Let $\betaGV : \XGV \coloneqq [\oX : \SS_{\maF}] \to \oX$ 
be the blow-down map. If $Y \in \maF$, we let
\begin{equation}\label{eq.def.lift.hatF}
\begin{gathered}
    \wha Y \ede \betaGV^*\overline{Y} \subset \XGV \quad \mbox{ and}\\
    \wha \maF \ede  \{ \wha Y \mid\, \ Y \in \maF\} \seq \betaGV^*\overline{\maF}\,.
\end{gathered}
\end{equation}

(In general, later on, we shall denote by $\wha Y$ the lift of $\oY$ to any
intermediate blow-up between $\XGV := [X : \SS_\maF]$ and 
$[\oX : \SS_{\maF} \cup \overline{\maF}]$.)

The independence of the blow-up on the admissible order \cite{AMN1} (a
result reminded in Theorem \ref{theorem.mainAMN1}),
gives the following

\begin{corollary}\label{prop.def.XF}
We have canonical diffeomorphisms
\begin{equation*}
   X_{\maF} \ede [\oX : \SS_{\maF} \cup \overline{\maF}] \, \simeq \, 
   \bigl[[\oX: \SS_{\maF}]: \wha \maF \bigr] 
    \seq [\XGV: \wha{\maF}] \,,
\end{equation*} 
where $\wha{\maF}$ is the lift of $\overline{\maF}$ to $[\oX : \SS_{\maF}]$.
\end{corollary}

\begin{remark} Let us assume that we have an order $\maF = (Y_0 = \{0\},
Y_1, \ldots, Y_k)$ on the linear subspaces of $\maF$ such that
$Y_i \subsetneq Y_j \ \Rightarrow i < j$ (that is, we have a ``size-order''
on $\maF$, in the sense of Kottke \cite{Kottke-Lin}). Then this is an admissible
order on $\maF$ and the induced orders
on $\SS_{\maF}$,  $\overline{\maF}$, and $\wha \maF$ are again admissible.
\end{remark}

So $X_{\maF}$ is obtained from $\XGV$ using only \emph{interior} blow-ups. By
contrast, in~\cite{ACN}, the process of blowing up along the collision
planes~$\overline{\maF}$ started from~$\oX$ to arrive at $[\oX: \overline{\maF}]$. 
This had the disadvantage that the smooth functions on $\overline{X/Y}$
do not lift to smooth functions on $[\oX: \overline{\maF}]$. 
Our paper thus fixes this issue from \cite{ACN} and provides uniform estimates at infinity.

\subsubsection{Boundary blow-ups and the data on $\XGV$}
Let us now see how the constructions of Lie manifolds and of their
geometric objects of the previous section particularize
to our setting, taking into account the two types of blow-ups: boundary
blow-up (along a \psubmanifold\ contained in the boundary) and
interior blow-up (along a \psubmanifold\ \textbf{not} contained in the boundary).
We begin with the blow-up of $\oX$ with respect to $\SS_\maF$, which,
we recall, is obtained using a sequence of boundary blow-ups.
We will let $\maWeucl$ denote the structural Lie algebra of vector fields on $\oX$.

\begin{remark} \label{rem.ex.GV} 
We will now define a Lie manifold structure on $\oX$, as well as on any of the 
intermediate blow-ups leading all the way up to $\XGV \coloneqq  [\oX: \SS_{\maF}]$.  
At first, the Lie manifold structure on $\oX$ is defined by the action of $X$ 
on itself and on $\ol X$ as described in Example \ref{ex.basic.LM.1}, namely, 
the structural Lie algebra of vector fields $\maWeucl$ on $\oX$ is
$\maWeucl \coloneqq \CI(\oX)X$. We have already described in \cite{AMN1} and
in then in the Introduction of this paper how to get from~$\ol X$ 
to~$\XGV$ by iterated blow-ups.
On any of these intermediate blow-ups~$\wihat X$ 
Theorem~\ref{theorem.mainAMN1} defines an action of~$X$ (as a Lie group) on~$\wihat X$ extending smoothly the translation action.
Deriving this Lie group action yields a Lie algebra action of~$X$ on $\wihat X$, given by a linear map 
$E:X\to \Gamma(\wihat X)$, which in fact provides smooth extensions of all constant 
vector fields on~$X$ to the compactification $\wihat X$. In general $E(X)$ will not 
vanish at the boundary, but it will be tangent to the boundary, see below for further 
discussion.
 
On all these blow-ups leading to the Georgescu--Vasy space $\XGV$, 
the Euclidean metric will be compatible with the Lie manifold 
structure on the (intermediate) blow-ups,
since we are performing only boundary blow-ups.

An alternative way to obtain the Lie manifold structure on all the iterated blow-ups~$\wihat X$, including $\XGV$ is to use Remark~\ref{rem.boundary.Lie} to define the 
Lie manifold structure inductively, starting from the Lie manifold structure on~$\ol X$.
It is easy to check by induction over the iterated blow-ups that this 
construction yields the same Lie manifold structures as above.

We stress that, as a consequence of  Remark~\ref{rem.boundary.Lie}, the resulting 
Sobolev spaces on all the blow-ups between~$\oX$ and $\XGV$ will be the same. Namely,
they are the \emph{usual Sobolev spaces} on~$X$.
Recall that we identify $X$ with 
constant vector fields on itself. The algebra of differential operators, however, will
change with each blow-up (along some $\SS_Y$, $Y \in \maF$), eventually leading on 
$[\oX: \SS_{\maF}]$ to the structural Lie algebra of vector fields 
$\maWb \coloneqq  \CI([\oX: \SS_{\maF}])X$. Therefore,
\begin{equation*}
    \Diff_{\maWb} (\XGV) \seq \CI(\XGV) \CC[X]\,,
\end{equation*}
where $\CC[X]$ denotes the algebra of polynomials on $X$ and where a monomial of degree~$d$
is interpreted as a differential operator of degree $d$ with constant coefficients.
\end{remark}

\subsubsection{Interior blow-ups}
As we have seen already, the situation becomes more complicated once
we start performing blow-ups along (lifts of) elements of $\overline{\maF}$.

\begin{remark}\label{rem.nb.Lie}
We use the notation introduced in Example~\ref{ex.basic.LM.2} and Lemma~\ref{lem.def.XF}. 
In particular, let $\wha\maF$ be the lifting of $\overline{\maF}$ to $\XGV$.
Let $\rho_{\wha\maF}$ be a smoothed distance function to $\wha{\maF}$ in $\XGV 
\coloneqq  [\oX: \SS_{\maF}]$.
By iterating Remark \ref{rem.interior.Lie}, we obtain a Lie manifold
structure on 
\begin{equation*}
    X_{\maF} \ede [\oX: \SS_\maF\cup \ol\maF] \seq \bigl[[\oX: \SS_\maF] : \wha{\maF}\bigr] 
\seq [\XGV: \wha{\maF}] 
\end{equation*}
(see Corollary~\ref{prop.def.XF}) by taking
\begin{equation*}
   \maW_{\maF} \ede \rho_{\wha{\maF}}\, \CI(X_\maF)  \maWint 
   \seq \rho_{\wha{\maF}}\, \CI(X_\maF)  \maWeucl \seq \rho_{\wha{\maF}}\, \CI(X_\maF) X\,.
\end{equation*}
Let us record here what are the various quantities associated with the Lie 
manifold $(X_\maF, \maW_{\maF})$, again by iterating Remark \ref{rem.interior.Lie}
(starting from $M \coloneqq  \XGV$ and successively blowing up with respect to
the lifts of $Y \in \maF$):
\begin{itemize}
\item The interior smooth manifold of the blow-up is this time: 
\begin{equation*}
   X_\maF \smallsetminus \pa (X_\maF) \seq X \smallsetminus \unionF \,.
\end{equation*}

\item If $g$ is a $\maWeucl$-compatible metric on $\XGV$, then $\tilde g\coloneqq 
\rho_{\wha{\maF}}^{-2}\beta^*g$ is a $\maW_{\maF}$-compatible metric on $X_{\maF} 
\coloneqq  [\XGV : \wha{\maF}]$. Here $\beta : X_{\maF} \to \XGV$ is the associated 
blow-down map.


\item The volume forms for $g$ and $\tilde g$ are related by the formula
      $\dvol^{\tilde g}= \rho_{\wha\maF}^{-n}\beta^*\dvol^g$,  $n=\dim (M) =\dim (X)$.
      As a consequence, we have
\begin{equation*}
  u \in L^p(X_\maF; \maW_\maF) \; \Longleftrightarrow \; 
  \rho_{\wha\maF}^{-n/p}u\in L^p(\XGV; \maWint) \; 
  \Longleftrightarrow \; \rho_{\wha\maF}^{-n/p}u\in L^p(X; \dvol_g) \,.
\end{equation*}

\item The Sobolev spaces defined by $\maW_{\maF}$ are ``weighted Sobolev spaces'' 
in the old metric~$g$ (compare to Equation \eqref{eq.def.Wkp} and notice the factor 
$\rho_{\wha{\maF}}^j$):
\begin{multline*} 
        W^{k,p}\bigl(X_\maF, \maW_{\maF}\bigr) \seq 
        \Bigl\{u \,\bigm|\; \rho_{\wha{\maF}}^{j} v_1 \ldots v_j
         u \in L^p\bigl(X_\maF; \maW_{\maF}\bigr)\,,\
        \forall v_1, \ldots, v_j \in X, \,
        0\le j \le k\,\Bigr\} \\
        = \Bigl\{u \,\bigm|\; \rho_{\wha{\maF}}^{j} V_1 \ldots V_j
         u \in L^p\bigl(X_\maF; \maW_{\maF}\bigr)\,,\
        \forall V_1, \ldots, V_j \in \maWint, \,
        0\le j \le k\,\Bigr\}\, .
\end{multline*}

\item $\Diff_{\maW_{\maF}}^m(X_\maF)$ is the linear span of 
differential monomials of the form $a \rho_{\wha{\maF}}^j V_1 V_2 \ldots V_j$, with 
$a \in \CI(X_\maF)$, $V_1, V_2, \ldots, V_j \in \maWeucl$, $0 \le j \le m$.
\end{itemize}

In order to obtain the last statement, we use, in particular, that $\rho_{\wha{\maF}}$ 
is the product of the smoothed distance functions to the blow-ups of $\wha Y$, 
where $\wha Y\in \wha\maF$, see Remark~\ref{rem.cond.ex.bis}. 
It is important to keep in mind that $\rho_{\wha{\maF}}$ is a smooth distance functions 
in $\XGV$ and that it is a product of smoothed distance functions in blow-ups of $\XGV$; 
the statements would be different for 
smoothed distance functions in $\oX$ and the corresponding blow-ups thereof. 
\end{remark}

\subsection{A lifting lemma and smoothed distance functions}
We continue to let $X$ be a finite-dimensional Euclidean vector space. Let $Y$ 
be a linear subspace of $X$. In this subsection, we prove a technical lemma (a 
lifting lemma) and apply it to the study of smoothed distance functions, leading 
to some results that will be used in the next subsection. Recall that the 
projection map $\pi_{X/Y}:X\to X/Y$ extends canonically to a smooth map 
$\psi_{Y} : [\ol{X}:\SS_Y]\to \overline{X/Y}$, see 
\citeschrSpace{Proposition}{prop.X/Y} and the proof of 
\cite[Theorem 4.1]{Kottke-Lin}. See also \cite{MelroseBook}.

\begin{lemma}\label{lemma.newPsiYZ}
There exists a diffeomorphism 
\begin{equation}\label{eq.make.a.lemma}
   \Psi_{Y} \seq (\psi_Y, \chi) :[\ol{X}:\SS_Y] \to
\overline{X/Y}\times \overline{Y}
\end{equation}
such that $\chi(y) = y$ if $y \in Y$ and $\beta_{\oX, \SS_Y}\circ \Psi_{Y}^{-1}$ 
restricts on $\oXY \times \SS_Y$ to the projection $\oXY \times \SS_Y \to \SS_Y$.
\end{lemma}
  
\begin{proof}
This follows from \citeschrSpace{Lemma}{blow.sphere} for $k = 1$ and $k' = 0$, 
see  Proposition~\ref{prop.details.newPsiYZ} in Appendix~\ref{app.spherical.compact} for details.
\end{proof}

\begin{remark}In order to compare the above well-known lemma
    to the literature, we mention that it says that the map $\psi_Y$ is the projection maps of a trivial fibration with fiber $\ol{Y}$.
In this context a smooth map $f:A\to B$ between manifolds with corners is 
the projection of a trivial fibration with fiber $C$, if and only if, $f$ is the 
first component of a diffeomorphism $A\to B\times C$. This is a special case of 
a fibration in the sense of \cite[Subsection~2.4]{MelroseBook}, a concept whose 
definition will not be recalled here, as it will not be required further.
\end{remark}

%
%

Continuing our idea to extend the map $\Psi_Y$ to a case with more than one blow-up,
let us introduce some notation. Let $Y \in \maF$ and 
choose an admissible order $(P_0 = \emptyset, P_1, \ldots, P_k)$ on $\SS_\maF$ 
(which is the same as choosing an admissible order
on $\maF$). According to the last part of~Proposition~\ref{prop.ex.ao}
we can choose this admissible order such that $\SS_Y=P_t$ for some $t\leq k$ and $\SS_{P_{j}}\subset \SS_Y$ for all $j< t$. Admissibility implies  $\SS_{P_{j}}\not\subset \SS_Y$ for $j>t$. Let
$\Psi_{Y}$ be the diffeomorphism defined in Lemma
\ref{lemma.newPsiYZ}. Let $\maP \coloneqq (P_0, P_1, \ldots, P_{t-1})$,
$\widehat{X} \coloneqq [\oX : (\maP, \SS_Y)]$.

\begin{corollary} \label{cor.last.step}
Let $\wha Y$ be the lift of $\overline{Y}$ to $\widehat{X} \coloneqq 
[\oX : (\maP, \SS_Y)]$. Using the notation we have just introduced, we have 
that the diffeomorphism $\Psi_{Y}:[\ol{X}:\SS_Y] \to
\overline{X/Y}\times \overline{Y}$ of Lemma~\ref{lemma.newPsiYZ}
lifts to a diffeomorphism 
\begin{equation}
    \widehat{\Psi}_{Y} : \widehat{X} \ede [\oX : (\maP, \SS_Y)] \rightarrow
    \ol {X/Y} \times [\ol Y: \maP] \,.
\end{equation}
This diffeomorphism maps $\widehat{Y}$ to $\{0_{X/Y}\} \times [\ol Y: \maP]$ 
diffeomorphically. In particular, a smoothed distance function to $\{0_{X/Y}\}$
in $\ol {X/Y}$ pulls back via
$\widehat{\Psi}_{Y} $ to 
a smoothed distance function to $\widehat{Y}$ in $\widehat{X}$.
\end{corollary}

\begin{proof}
Let $\beta := \beta_{\oX, \SS_Y}$ and $\beta^{-1}(\maP) := (\beta^{-1}(P_0),
\beta^{-1}(P_1), \ldots, \beta^{-1}(P_{t-1}))$. (Notice that $\beta^{-1}(P_0) 
= \beta^{-1}(\emptyset) = \emptyset$.)
Thanks to the properties of the family $\mathcal{F}$ and the fact that $\SS_Y$ is a maximum of $\{P_0, P_1, \ldots, P_t\}$, the two orders $(P_0, P_1, \ldots P_t= \SS_Y)$ 
and $(P_0,\SS_Y,P_1, \ldots P_{t-1})$
are both ``intersection orders'' in the sense of~\cite{Kottke-Lin} for the semi-lattice 
$(P_i)_{i=1}^t$.
Corollary~3.5 of~\cite{Kottke-Lin} gives the first diffeomorphism of the following equation
\begin{eqnarray*}
   [\oX: (\maP, \SS_{Y}) ]  & \simeq & \bigl[[\oX: \SS_Y], \beta^{-1}(\maP) \bigr] \\
   & \simeq &[ \oXY \times \oY: \Psi_{Y}(\beta^{-1}(\maP)) ] \nonumber  \\
   & \simeq & [ \oXY \times \oY: \oXY \times \maP ] \nonumber \\
   & \simeq & \oXY \times [\oY: \maP ]   \nonumber
   \,.
\end{eqnarray*}
The second diffeomorphism of this equation is obtained by 
using the diffeomorphism $\Psi_{Y}$ of Lemma \ref{lemma.newPsiYZ}.
Using $P_j \subset \SS_Y$, Lemma~\ref{lemma.newPsiYZ} implies, 
$\Psi_{Y}(\beta^{-1}(P_j)) = \oXY \times P_j$, and hence $\Psi_{Y}(\beta^{-1}(\maP)) = \oXY \times \maP$. This provides the third diffeomorphism.
The last diffeomorphism follows then from Proposition~\ref{prop.lemma.product}. 

We now turn to the second statement of our result, namely, that
$\widehat{\Psi}_{Y}$ maps $\widehat{Y}$ to $\{0_{X/Y}\} \times [Y: \maP]$ 
diffeomorphically. By definition, the lift $\wha Y$ of $\overline{Y}$ to $\wha X$ is 
the closure of $\oY \smallsetminus \maP$ in $\wha Y$. In view of the 
diffeomorphism of the last equation and of Lemma \ref{lemma.newPsiYZ}, 
this lift is the closure of $\{0\} \times Y$ 
in $\oXY \times [\oY: \maP ]$, and hence diffeomorphic to $[\oY: \maP ]$.
(The diffeomorphism $\wha Y \simeq [\oY : \maP]$ also follows by
iterating Proposition \ref{prop.beta.m1}.)

The final statement about distances is then an immediate consequence of this 
diffeomorphism. (It is also a consequence of the more general result proved in
Proposition~\ref{prop.distance.product}.)
\end{proof}

We will use  Corollary \ref{cor.last.step} to define a canonical map $b : \XGV \coloneqq [\oX : \SS_\maF] 
\to \oXY$. This map coincides with the composite map considered also in \cite[Proposition~5.10]{AMN1}. Combining with Lemma 
\ref{lemma.disjoint.sd}, 
we obtain the following.

\begin{corollary}\label{cor.verylast.step}
Using the notation introduced in Corollary \ref{cor.last.step}, let $r_0$ be
a smoothed distance function to $0 \in \oXY$. Let $\betaGV : [\oX: \SS_\maF]
\to \oX$ be the blow-down map. Let $b$ be the composite map 
\begin{equation}
    \XGV \coloneqq [\oX: \SS_\maF] \rightarrow [\oX : (\maP, \SS_Y)] \rightarrow \oXY 
\end{equation} 
obtained from Corollary \ref{cor.last.step}. Then $r_0 \circ b$ is a smoothed 
distance function to $\betaGV^*(\oY) \subset [\oX: \SS_\maF]$ in $[\oX: \SS_\maF]$.
\end{corollary}

\begin{proof}
This follows from the properties of the map $b$
and the aforementioned Lemma~\ref{lemma.disjoint.sd} and Corollary~\ref{cor.last.step}
as follows. Let $\maP$ and $\wha Y$ be as in Corollary \ref{cor.last.step}. The remaining spheres at infinity yield
$\maP_{\mathrm{rem}} \coloneqq \beta_{\oX, (\maP, \SS_Y)} (\SS_\maF \smallsetminus \{\maP, \SS_Y\})$ which 
consists of \psbmanifolds{} disjoint from $\wha Y$.
We first have that $r_0 \circ b$ is a distance function to $\wha Y$
in $\wha X \coloneqq [\oX : (\maP, \SS_Y)]$. Then this lifts to a
distance function to $\beta^* (\oY)$
in $\XGV \coloneqq \bigl[[\oX : (\maP, \SS_Y)]: \maP_{\mathrm{rem}}\bigr]$
by Lemma~\ref{lemma.disjoint.sd}.
\end{proof}

\subsection{More on smooth distance functions on blow-ups}
As before let $\betaGV :\XGV\coloneqq [\oX: \SS_\maF]
  \to \oX$ be the blow-down map from the Georgescu--Vasy compactification
  to the ball compactification. Again for $Y \in \maF$, we write $\betaGV^*(Y)$ for the lift of $\overline{Y}\subset\overline X$ to $[\oX: \SS_{\maF}]$, i.e.\ it 
is the closure of $Y$ in $[\oX: \SS_{\maF}]$.  In this subsection we use the notation $\wha{Y}\coloneqq \betaGV^*(Y)$. (Note that now $\wha{Y}$ only has this specific meaning, whereas it denoted a more general class of lifts previously.)

The goal of the current subsection is to compare a smoothed distance to $\wha{Y}$ in $\XGV$ to the Euclidean distance to $Y$. More precisely, we show in the following lemma that functions similar to $\arctan(d_Y):X\to \mathbb{R}$, where $d_Y$ is the Euclidean distance function to~$X$,
extend to a smoothed distance function to $\wha{Y}$ in $\XGV$. This is mostly a question about the asymptotic of $d_Y$ at infinity. 

The reader should be aware that a smoothed distance
to $\oY$ in $\oX$ will not pull-back to a smoothed distance function to $\wha{Y}$ in $\XGV$, 
thus it is important to keep in mind with respect to which of the two compactifications,
$\ol{X}$ or $\wha X$, the smoothed distance functions is defined. Again, we let
$\wha \maF \coloneqq \betaGV^*(\overline{\maF})
=\{\wha Y\mid Y\in \maF\}$ and recall from Corollary~\ref{prop.def.XF} 
that $[\XGV: \wha{\maF}] = [\oX: \SS_{\maF}\cup \ol\maF] =: X_{\maF}$.

\begin{lemma}\label{lemma.dY}
Let $\phi_0 : [0, \infty) \to [0, 1]$ be a smooth, non-decreasing function 
satisfying $\phi_0(t) = t$ for $t \le 1/2$ and $\phi_0(t)=1$ for $t\geq 1$. 
Let $Y \in \maF$ and $d_Y(x) \coloneqq  \dist(x, Y)$ be the distance function 
from $x \in X$ to $Y$ with respect to the Euclidean distance on $X$,
as before. Then $\phi_0 \circ d_Y$ extends to a smoothed distance function
$r_{\wha{Y}}$ to $\wha{Y}$ in $\XGV \coloneqq [\oX: \SS_\maF]$ such that 
$r_{\wha{Y}}/d_Y$ extends to a smooth function on $\XGV$.
\end{lemma}

It follows trivially from the lemma that $r_{\wha{Y}}/d_Y$ extends to a smooth 
function on $X_{\maF} \coloneqq [\XGV: \wha{\maF}]$.

\begin{proof}
Let $r_0(z) \coloneqq \phi_0(|z|) = \phi_0(\dist(z, 0))$, where $z \in X/Y$.
Then $r_0$ extends to a smooth distance function to 0 in $\oXY$. Let $b$
be the map of Corollary \ref{cor.verylast.step}. Then that corollary
gives that $r_{\wha Y} =  r_0 \circ b$ is a smoothed distance function to $\wha Y$ 
in $\XGV$ that coincides with $\phi_0 \circ d_Y$ on $X$.

Next, the function $q(z) \coloneqq r_0(z)/|z| = \phi_0(|z|)/|z|$ extends
to a smooth function on $\oXY$. Hence $q \circ b$ is smooth on $\XGV$. 
This function is the desired smooth extension of  $r_{\wha{Y}}/d_Y$.
\end{proof}

\begin{remark}
Note that $d_Y|_X$ is a smoothed distance function to $Y$ in $X$, but \emph{not in} 
$\oX$ as it does not extend to a continuous (or smooth) function 
$\oX\to [0,\infty)$.
\end{remark}

The purpose of the framework developed in Subsection~\ref{ssec.s.dist} was to prove 
the following result.

\begin{proposition} \label{prop-smoothing-pot}
Let $Y \in \maF$ and $d_Y(x) \coloneqq  \dist(x, Y) = \inf \bigl\{\|x-y\|\mid y\in Y\bigr\}$ 
be the Euclidean distance function from $x \in X$ to $Y$ and $\wha{\maF}$
be the lift of $\overline{\maF}$ to $\XGV$, as before. Let $\rho_{\wha{\maF}} : 
\XGV \to [0, \infty)$ be a \emph{smoothed} distance function to $\wha{\maF}$ in $\XGV$. 
Then  $\rho_{\wha{\maF}}\, d_Y^{-1}:X\setminus Y\to [0,\infty)$ extends smoothly to a function in 
$\CI\bigl([\XGV:\wha \maF]\bigr) \seq \CI(X_\maF)$, denoted again by  $\rho_{\wha{\maF}}\, d_Y^{-1}$.
\end{proposition}

\begin{proof}  
We use the notation introduced in Lemma~\ref{lemma.dY}, in particular, 
$r_{\wha Y} \coloneqq  r_0 \circ b$ extends $\phi_0 \circ d_Y$. 
We thus have that $r_{\wha Y}/d_Y \in \CI(\XGV)\subset \CI(X_\maF)$, by 
Lemma~\ref{lemma.dY}. Moreover, $r_{\wha Y}$ is a smoothed distance to $\wha Y$ in 
$\XGV=[\oX : \SS_{\maF}]$, by the same lemma. 
Hence $\rho_{\wha{\maF}}/r_{\wha Y}$ extends to a smooth function on 
$\bigl[[\oX: \SS_{\maF}] : \wha{\maF}\bigr] 
= [\oX : \SS_\maF\cup \ol\maF]=X_\maF$, by Proposition \ref{prop.smoothness} for 
$M = [\oX: \SS_{\maF}]$ and $\maS=\wha\maF$. Hence
\begin{equation*}
   \rho_{\wha{\maF}}\, d_Y^{-1} \seq \frac{\rho_{\wha{\maF}}}{r_{\wha Y}} 
   \cdot \frac{r_{\wha Y}}{d_Y} \in \CI(X_\maF)\,.
\end{equation*}
The proof is complete.
\end{proof}

We fix in what follows $\rho_{\wha{\maF}}$ to be a smoothed distance to 
$\wha{\maF}$ in $[\oX: \SS_{\maF}]$. Our constructions and reasoning will 
not depend on the particular choice of $\rho_{\wha{\maF}}$.


\section{Regularity results for eigenfunctions}\label{sec5}
We now formulate and prove our main regularity results for certain differential operators 
with singular coefficients on $X$. These results apply, in particular,
to Schr\"odinger operators with 
``inverse square potentials'', a class of potentials which will be defined below and that includes the classical Schrödinger operator with Coulomb potential, which are used in physics and chemistry. The more general class of Schr\"odinger operators with inverse square potentials became  of renewed interest
\cite{DerezinskiFaupin, DerezinskiRichard, HLNU1, HLNU2, LiNi09}. 
To summarize our approach, we use the method in \cite{ACN}, but starting with
$\XGV$ instead of $\oX$. This improvement leads to regularity statements which are even 
uniform close to infinity.
  The Lie manifold structure on $\XGV$ is obtained
from the action of $X$, as explained in Example \ref{ex.basic.LM.2}. (See also 
Remarks~\ref{rem.boundary.Lie} and~\ref{rem.ex.GV}.)

Recall that, throughout this paper, $\maF$ is a finite semilattice of linear subspaces of $X$
with $0  \in \maF$, $X \notin \maF$, and $\SS_\maF \cup \ol\maF$ 
is the associated clean semilattice, as in Equation \eqref{eq.def.semilattices}. There is no loss
of generality to assume that $X = \RR^n$, when convenient. As before, we let 
\begin{itemize}
\item $\XGV \coloneqq  [X: \SS_\maF]$, the Georgescu--Vasy space; 
\item $X_{\maF} \coloneqq  [\oX :  \SS_\maF\cup \ol\maF] = \bigl[[X : \SS_\maF] : \wha{\maF}\bigr]$; 
the blow-up along $\SS_\maF\cup \ol\maF$; 
\item $\rho_{\wha{\maF}} : \XGV \to [0, \infty)$, a smoothed distance 
to $\wha{\maF}$ in $\XGV$, see Definition~\ref{def.rho_maS};
\item $\maW_{\maF} \coloneqq  \rho_{\wha{\maF}} \, \CI(X_{\maF})X$,
the structural Lie algebra of vector fields 
on~$X_\maF$, discussed in more detail in Remark \ref{rem.nb.Lie}.
\end{itemize}

\begin{theorem}\label{theorem.regularity} 
We use the usual notation, recalled for instance in the last paragraph.
Let $m \in \NN$ and $D_0, D_1, \ldots, D_m$
be differential operators on $X$ with coefficients in $\CI(X_\maF)$, $D_k$
of order $\le k$, and
$D \coloneqq  D_m + \rho_{\wha{\maF}}^{-1}D_{m-1} + \ldots + \rho_{\wha{\maF}}^{-m}D_0$. 
Then 
\begin{enumerate}[{\rm (i)}]
\item  $\rho_{\wha{\maF}}^m D \in \Diff_{\maW_{\maF}}^m(X_{\maF})$. 

\item If $D_m$ is uniformly elliptic on $X$, then $\rho_{\wha{\maF}}^m D$
is elliptic in $\Diff_{\maW_{\maF}}^m(X_{\maF})$. 

\item For each boundary hyperface $H$ of $X_\maF$, 
let $x_H$ be a defining function and $\mu_H \in \RR$.
Let $\chi \coloneqq  \prod_H x_H^{\mu_H}$ and $1<p<\infty$. We assume that  
$\rho_{\wha{\maF}}^m D$ is elliptic in $\Diff_{\maW_{\maF}}^m(X_{\maF})$ and that
$u \in \chi L^p(X_{\maF})$ satisfies $D u= \lambda u$ on $X_{\maF} \smallsetminus 
\pa X_{\maF} \subset X$ for some $\lambda \in \RR$. Then
$u \in \chi W^{\ell,p}(X_{\maF}, \maW_{\maF})$ for all $\ell \in \NN$.
\end{enumerate}
\end{theorem}

Before proceeding to the proof, let us make a few comments on the setting.

\begin{remark}
All of $\XGV$, $X_{\maF}$, $\rho_{\wha{\maF}}$, and $\maW_{\maF}$ 
are defined using an inductive procedure
based on an ordering of $\SS_\maF\cup \ol\maF$. This ordering is not important, as long as it is
an admissible order (see Definition \ref{def.admissible}). In our case, however, it is
convenient to use an admissible order that puts first the elements of $\SS_\maF$
and then the elements of $\overline{\maF}$, as many objects are defined on the intermediary
blow-up $\XGV\coloneqq  [\oX: \SS_\maF]$. We note also that $L^p(\XGV) = L^p(X)$
since 
\begin{equation}
   \XGV \smallsetminus \pa \XGV \seq \oX \smallsetminus \partial \oX \seq 
   X \smallsetminus \unionF\,,
\end{equation}
with the same induced measure. However, the measure on the interior of
$X_{\maF}$, $X_{\maF} \smallsetminus \pa X_{\maF} = X \smallsetminus \unionF$ 
differs from the measure on $X$ by the factor $\rho_{\wha{\maF}}^{-n}$, where $n=\dim M$. 

Also, we mention that this theorem generalizes Theorem 4.2 in
\cite{ACN}, but note that, in the statement (ii) of that theorem, the radial
compactification (denoted $\SS$ in that paper) needs to be replaced with the
Georgescu--Vasy space $\XGV$.
\end{remark}

\begin{proof}\ 
\begin{enumerate}[{Ad} (i):]
\item\label{thm.reg.proof.i} This follows from Remarks \ref{rem.ex.GV} and \ref{rem.nb.Lie} and
  from the definitions of $X_{\maF}$ and $\maW_{\maF}$.
  
\item\label{thm.reg.proof.ii} This follows from \eqref{thm.reg.proof.i} just proved by combining 
with Remarks~\ref{rem.interior.Lie}, \ref{rem.ex.GV}, and~\ref{rem.nb.Lie}.

\item\label{thm.reg.proof.iii} Let $u \in \chi L^p(X_{\maF})$ be such that $D u = \lambda u$. 
Then $Q \coloneqq  \rho_{\wha{\maF}}^m D - \lambda \rho_{\wha{\maF}}^m$ is an elliptic operator 
in $\Diff_{\maW_{\maF}}(X_\maF)$, by \eqref{thm.reg.proof.ii} just proved. It satisfies $Q u = 0$. 
The result is hence a direct consequence of the regularity result in \cite{sobolev} (this result 
was recalled in Theorem~\ref{theorem.ain}).
\end{enumerate}
\end{proof}

Let us give now a more concrete application. We fix a Euclidean metric on $X$, $n=\dim (X)$.
First, let us note that it follows from Remarks \ref{rem.interior.Lie} and 
\ref{rem.ex.GV} that 
\begin{eqnarray}
  W^{k,p}(X_{\maF}, \maW_{\maF}) &\coloneqq & \{u: X \to \CC \,\mid\
         \rho_{\wha{\maF}}^{|\alpha|}
        \pa^\alpha u \in L^p(X, \mu_\maF),\, \
        |\alpha| \le k\,\}\nonumber\\
        &=& \{u: X \to \CC \,\mid\
         \rho_{\wha{\maF}}^{|\alpha|-(n/p)}
        \pa^\alpha u \in L^p(X, \mu_\eucl),\, \
        |\alpha| \le k\,\}\, ,
  \label{eq.Sobolev.XS}
\end{eqnarray}
where we used the standard Lebesgue measure $\mu_\eucl$ and the measure 
$\mu_\maF= \rho_{\wha{\maF}}^{-n}\mu_\eucl$  associated with 
$g_\maF=\rho_{\wha{\maF}}^{-2}\geucl$ on  $X$.

\begin{definition}\label{def.isp} 
Let $X$, $\maF$, $\SS_\maF\cup \ol\maF$, and $X_{\maF} \coloneqq  [\oX :  \SS_\maF\cup \ol\maF]$ be
as above (as in Theorem~\ref{theorem.intro.reg}, for instance). 
Let $d_Y(x) \coloneqq  \dist(x, Y)$ denote the distance from $x \in X$ to $Y \in \maF$.
An \emph{inverse square potential with singularities in $\maF$}
is a function $V : X \smallsetminus \unionF \to \CC$ of the form 
\begin{equation*} 
    V(x) \ede \sum_{Y \in \maF}  \left( \frac{a_Y(x)}{d_Y(x)^{2}} 
    + \frac{b_Y(x)}{d_Y(x)} \right) + c(x)\,,
\end{equation*}
where $a_Y, b_Y, c  \in \CI(X_\maF)$.
\end{definition}

Note that, since $\CI(\oX)\subset \CI(X_\maF)$,
our inverse square potentials are rather general and include the usual inverse 
square potentials. The following lemma justifies our definition of inverse 
square potentials.

\begin{lemma}\label{lemma.inv.sq}
The set of inverse square potentials is a complex vector space.
Let $V$ be an inverse square potential, then $\rho_{\wha{\maF}}^2 V \in \CI(X_\maF)$.
\end{lemma}

\begin{proof} From the definition above, it is clear that the set of inverse square potentials is a complex vector space.
Proposition \ref{prop-smoothing-pot} gives $\rho_{\wha{\maF}} d_Y \in \CI(X_\maF)$. The result then
follows since $\rho_{\wha{\maF}} \in \CI(X_\maF)$.
\end{proof}

\begin{example}[The Schr\"odinger operator in quantum physics]
The Schr\"odinger operator of an atom with heavy nucleus and
with $N$ electrons, studied in physics, is the operator
\begin{equation*}
   u(x)\mapsto (\maH u)(x):= \Delta u(x)  +V(x) u(x)\,,
\end{equation*}
regarded as an unbounded, 
densely defined operator in $L^2(\RR^n)$, $n=3N$. We write
$x=(x_1,x_2,\ldots, x_N)\in \RR^{3N}$ with $x_i\in \RR^3$.
The potential $V$ is given by
\begin{equation*}
   V(x)\ede \sum_{1 \le j \le N} \frac{b_j}{\|x_j\|} 
   + \sum_{1 \le i < j \le N} \frac{c_{ij}}{\|x_i-x_j\|}.
\end{equation*}
The potential $V$  is an inverse square potential with singularities 
on the collision planes (more precisely, on the semilattice
generated by the collision planes). We thus may apply Theorem~\ref{theorem.regularity} 
to eigenfunctions of the differential operator $\maH$.
\end{example}

Moreover, the inverse square potentials considered in 
\cite{DerezinskiFaupin, DerezinskiRichard, HLNU1, HLNU2, LiNi09} are also 
inverse square potentials in our sense.

\begin{theorem}\label{theorem.application} 
We use the notation introduced in Theorem~\ref{theorem.regularity}.
Let $D$ be a constant coefficient elliptic operator on $X = \RR^n$
and $V$ be an inverse square potential. 
Let $\rho_{\wha{\maF}}$ be a smoothed distance function
to $\wha{\maF}$, as above.
Assume $u \in L^2(\RR^{n})$ is an eigenfunction of $D + V$, then 
\begin{equation*}
     \rho_{\wha{\maF}}^{|\alpha|} \pa^\alpha u \in L^2(\RR^n)
\end{equation*}
for all multi-indices $\alpha$.
\end{theorem}

\begin{proof} 
Let $g_{\RR^{n}}$ be the canonical Euclidean metric on $\RR^n$.
It is also a compatible metric~$g_{GV}$ on $\XGV$, as explained in 
Remark~\ref{rem.ex.GV}. Then $g_{X_{\maF}} = \rho_{\wha{\maF}}^{-2} g_{\RR^n}$ is
a compatible metric~$g_{\maF}$ on~$X_{\maF}$, as explained in Remark \ref{rem.nb.Lie}.
Hence we have that 
\begin{equation*}
   L^2(\RR^{n}, g_{\RR^n}) \seq L^2(\XGV, g_{GV}) \seq 
   \rho_{\wha{\maF}}^{-n/2} L^2(X_{\maF}, g_{\maF}) \,.
\end{equation*}
The function $\rho_{\wha{\maF}}$ is a product of defining functions of 
hyperfaces of $X_\maF$, by the definition of a smoothed distance
function to a semilattice (Definition \ref{def.rho_maS}), and hence 
$\rho_{\wha{\maF}}^{-3N/2} = \chi$, for some $\chi$ as in
Theorem~\ref{theorem.regularity} \eqref{thm.reg.proof.iii} with $\mu_H\equiv -n/2$. 
The result then follows from
Theorem~\ref{theorem.regularity} \eqref{thm.reg.proof.iii}  and the description of Sobolev spaces
in Equation~\eqref{eq.Sobolev.XS}. 
\end{proof}

Of course, $D = - \Delta$ satisfies the assumptions of the above theorem,
so our regularity estimates are valid for Schr\"odiner operators with
inverse square or Coulomb type singularities. To obtain Theorem \ref{theorem.intro.reg} 
stated in the Introduction, recall from Remark~\ref{rem.useful} the following. 
Let $\rho(x) \coloneqq \dist_{\overline g}(x, \unionma{\wha\maF})$ be the 
distance to $\unionma{\wha{\maF}}$ in some \emph{true}
metric $\overline g$ on $\XGV$. Then the functions $\rho_{\wha{\maF}}$, $\rho$, 
and the function $\delta_{\maF} \coloneqq \min \{ \dist (x, \unionF), 1\}$ 
are Lipschitz equivalent, see Appendix~\ref{app.a}.

\begin{remark}\label{rem.nonlin}
The methods in this article also generalize to non-linear equations.
As an example we consider equations of the form 
\begin{equation}\label{eq.nonlin}
  \Delta u + V |u|^s u=\lambda u
\end{equation}
on $\RR^n$, where $\lambda\in \CC$, $0<s\leq 2/(n-2)$,
and $V$ is of Coulomb type (i.e. as in \eqref{eq.V.theo.intro.reg}). 
Let $\hat g\coloneqq \rho_{\wihat\maF}^{-2}\geucl$ be an adapted Riemannian metric on $(X_\maF,\maW_\maF)$. We define the \emph{Yamabe operator} $L^g\coloneqq \Delta^g- \frac{n-2}{4(n-1)}\mathrm{scal}^g$, and one may find in any reference about the Yamabe problem or conformal geometry that
$$L^{\hat g}(\hat u)= (\rho_{\wihat\maF})^{(n+2)/2}L^{\eucl}\bigl((\rho_{\wihat\maF})^{-(n-2)/2}\hat u\bigr).$$
After multiplication with $(\rho_{\wihat\maF})^{(n+2)/2}$ Equation~\eqref{eq.nonlin}  transforms into
$$L^{\hat g}\hat u +\rho_{\wihat\maF}^t V|\hat u|^s \hat u = \rho_{\wihat\maF}^2 \lambda \hat u$$
with $t=\big((n+2)-(s+1)(n-2)\bigr)/2=2-(s+1)(n-2)/2$ and $\hat u\coloneqq \rho_{\wihat\maF}^{(n-2)/2}u$. We have $t\geq 1$ if and only if $s\leq 2/(n-2)$, and in this case the Coulomb assumption implies the boundedness of $\rho_{\wihat\maF}^t V$. Then regularity theory yields for any $\eta\in \CI(X_\maF)$: if
\begin{equation}\label{reg.cond}
  \hat u\in \eta L^2(X_\maF,\maW_\maF)\cap \eta^{1/(s+1)}  L^{2/(s+1)}(X_\maF,\maW_\maF)
\end{equation}
then
\begin{equation}\label{reg.concl}
  \hat u\in \eta W^{2,2}(X_\maF,\maW_\maF).
\end{equation}
We may apply this for example to $\eta\coloneqq \rho_{\wihat\maF}^{-\left(1+\frac{ns}2\right)}$. Then
\begin{equation}\label{reg.cond.eucl}
  u\in L^2(X,\geucl)\cap  L^{2/(s+1)}(X,\geucl)
\end{equation}
implies \eqref{reg.cond}, and then~\eqref{reg.concl} implies
\begin{equation}\label{reg.concl.eucl}
  \rho_{\wihat\maF}^{|\alpha|+\frac{ns}2}\partial^\alpha u \in L^2(X,\geucl)\text{ for all multi-indices }\alpha\text{ with }|\alpha|\leq 2.
\end{equation}
Thus we have \eqref{reg.cond.eucl} $\Rightarrow$ \eqref{reg.concl.eucl}, and many similar conclusions are possible.
\end{remark}

\begin{remark}\label{rem.symmetry}
The point of view of Subsection~\ref{subsec.blow-up-spherical} yields an alternative 
way to construct the space $X_\maF$. Recall that as a set we have $X_\infty=X \cup\{\infty\}$ with the differential structure given by stereographic projection.
Instead of taking the closure of a
$\{0\} \neq Y \in \maF \smallsetminus \{ 0 \}$ in $\ol X$, we may take its closure 
in $X_\infty$ which is then $Y_\infty = Y\cup\{\infty\}$. 
For $Y=\{0\}$, we use $\{0\}_\infty = \{0,\infty\}$ instead of the closure. We define
$$\maF_\infty \coloneqq \bigl \{Y_\infty\,\bigm|\, Y\in \maF \bigr\}.$$
Since we have (always) assumed that $\maF$ is a finite semi-lattice of 
linear subspaces of $X$ with $\{0\}\in \maF$, we see that $\maF_\infty$ is a clean
semilattice of closed \psbmanifolds{} of $X_\infty$
(i.e. a cleanly  intersecting family of closed \psbmanifolds{} and a semi-lattice),
which we endow with an admissible ordering. Because of the diffeomorphisms
$[X_\infty:\{\infty\}] \simeq \ol X$ and $[Y_\infty:\{\infty\}] = 
\ol Y\subset [X_\infty:\{\infty\}]$ we see that $[X_\infty:\maF_\infty]=[\ol X,\ol\maF]$. This is the blow-up used in \cite{ACN}. However it differs from the blow-up constructed in this article, which is $X_\maF=[\XGV:\ol F]$.
\end{remark}

\begin{example}
Consider $\mathcal{F}=\{ \{0\},Y\}$ then $\mathcal{F}_\infty=\{ \{0\}\cup \{\infty\},Y_\infty \}$.
Using the result that $[M:P,Q] \simeq [[M: P] :[Q:P]]$ for $P \subset Q$, we obtain
\begin{eqnarray*}
  [X_\infty: \maF_\infty]&=&[[X_\infty : \{ \infty ,0\}, [Y_\infty : \{\infty,0\}] ]=[[X_\infty : \{ \infty \} ]: \{0\}, [Y_\infty : \{\infty\}]: \{0\} ]\\
    &=& [\oX: \{0\} , \oY] = [\oX : \overline{\mathcal{F}}] \neq X_\maF= [\oX : \mathbb{S}_\maF \cup \overline{\mathcal{F}}].
\end{eqnarray*}
\end{example}

\appendix

\section{The equivalence of $\rho$ and $\delta_{\maF}$}\label{app.a}

Let $X$ and $\maF$ be as in the previous sections, that is, $\maF$ is a finite 
semilattice of linear subspaces of the real, finite-dimensional vector space $X$. 
For each $Y\in \maF$ let $\wha Y$ be the closure of $Y$ in $\XGV$
(it coincides with the lift $\beta^*\oY$ of $\oY$ to $\XGV$) and 
$\wha{\mathcal{F}}\ede \{\wha Y\mid Y\in \maF\}$. Recall from
the Introduction that $\rho(x) \coloneqq  \dist_{\overline g}(x, \wha{\maF}) 
=  \dist_{\overline g}(x, \maF)$ be the distance to $\wha{\maF}$ in some \emph{true}
metric $\overline g$ on $\XGV$. Let $\dist(x, \unionF)$ be the distance from $x$ to
$\unionF$ in the usual, Euclidean distance and set $\delta_{\maF}(x) \coloneqq  
\min\{\dist(x, \unionF), 1\}$, again as in the Introduction.
The function $\rho$ used to obtain our regularity results in the previous section
is maybe not explicit enough, so
we prove now that the functions $\rho$ and $\delta_{\maF}$ are Lipschitz
equivalent. More precisely, we shall prove the following
result.

\begin{proposition}\label{prop.equivalence}
Let $\maF$ be a semilattice of linear subspaces of
the real, finite-dimensional vector space $X$. The following 
functions are Lipschitz equivalent (as functions on $X$):
\begin{enumerate}
   \item $\rho \coloneqq $ the distance function to $\unionx{\wha{\maF}}$ 
    in some true metric on $\XGV$;
   \item $\delta_{\maF}(x) \coloneqq  \min\{\dist(x, \unionF), 1\}$; and
   \item $\rho_\maF  \coloneqq $ a \emph{smoothed} distance function to 
   $\unionx{\wha{\maF}}$.
\end{enumerate}
\end{proposition}

We recall that it was proved in \cite{ACN} that $\rho_\maF$
and $\rho$ are continuously equivalent and thus Lipschitz equivalent. 
(This result was discussed and reminded in Remark~\ref{rem.useful}.)
The rest of this section will then be devoted to proving that
$\rho$ and $\delta_\maF$ are Lipschitz equivalent.

\begin{lemma}\label{lemma.Lie.le}
Let $(M, \maV)$ be a Lie manifold (and hence $M$ is compact, 
see Definition~\ref{def.Lie.Man}),
let $\overline g$ be a true metric on $M$, and let
$g$ be a $\maV$-compatible metric. Then there exists $C > 0$
such that $\overline g \le C^2 g$. In particular, for all $x, y \in M \smallsetminus \pa M$,
\begin{equation*}
   \dist_{\overline g} (x, y) \le C \dist_g(x, y)\,.
\end{equation*}
\end{lemma}

\begin{proof}
As explained in Remark~\ref{rem.alt.descrip.compat.met}, a compatible metric on $M$ is given 
by a metric $\gamma\in\Gamma(A^*\otimes A^*)$ on the vector bundle (Lie algebroid)
$A$, and then $G := \bigl(\varrho\otimes \varrho\bigr)(\doublesharp{\gamma})\in \Gamma(TM\otimes TM)$ 
is a well-defined smooth symmetric section. On 
$\interior{M}$, the section $G$ is non-degenerate, and thus we have $G|_{\interior{M}}=\doublesharp{{g_0}\kern-.4pt}$ 
for some $\maV$-compatible Riemannian metric $g_0$ on $\interior{M}$. By definition, any  $\maV$-compatible 
Riemannian metric arises in this way for some $\gamma$.
We thus have seen that $\doublesharp{g}$ extends smoothly to a symmetric tensor $G\in\Gamma(TM\otimes TM)$.
The $\ol g$-unit cotangent bundle $\SS_{\ol g}^* M\subset T^*M$ is compact.
By continuity of $G$ and compactness of $\SS_{\ol g}^* M$ the supremum 
$c_1\coloneqq \sup \{G(X,X)\mid X\in \SS_{\ol g}^* M\subset T^*M\}$ 
is bounded. We get $G\leq c_1 \doublesharp{\bar g}$ and thus also $\doublesharp{g}\leq c_1 \doublesharp{\bar g}$. 
This implies the statement by duality for $C:=\sqrt{c_1}$. 
\end{proof}

We obtain the following corollary.

\begin{corollary} \label{cor.first.ineq}
Let $g$ be the Euclidean distance on $\RR^n$ and $\overline g$ be
any true metric on $\XGV$. Then there exists $C > 0$ such
that, for any $x \in X$, we have
\begin{equation*}
  \dist_{\overline g}(x, \unionF) \le C \dist_{g}(x, \unionF)
\end{equation*}
\end{corollary}
\begin{proof}
Let $y \in \unionF$ be such that $\dist_{g}(x, \unionF) = 
\dist_{g}(x, y)$, which exists since $\unionF$ is a closed
subset of $X$. Also, let $C$ be as in Lemma~\ref{lemma.Lie.le}.Then
\begin{equation*}
  C\dist_{g}(x, \unionF) \seq C\dist_{g}(x, y) \ge \dist_{\overline g}(x, y)
  \ge \dist_{\overline g}(x, \unionF)\,.
\end{equation*}
\vskip-\baselineskip
\end{proof}

Let $L \coloneqq  \{x \in X \mid \dist_{g}(x, \unionF) \ge 1\}$.

\begin{corollary} \label{cor.first.ineq2}
There is $C > 0$ such that,
for all $x \in X$, $\rho(x) \le C \delta_{\maF}(x)$.
\end{corollary}

\begin{proof} 
The function $\rho$ is bounded by some $C_1 > 0$ (since 
$\XGV$ is compact). Hence, if $x \in L$,
$\rho(x) \le C_1 = C_1 \delta_{\maF}(x)$.
On the other hand, with $C$ as in Corollary \ref{cor.first.ineq},
if $x \notin L \coloneqq  \{x \in X \mid \dist_{g}(x, \unionF) \ge 1\}$, we have
$\rho(x) = \dist_{\overline g}(x,\, \unionF) \le C \dist_{g}(x,\, \unionF) = 
C \delta_{\maF}(x)$.
So the desired $C$ is the largest of $C_1$
and the $C$ of Corollary \ref{cor.first.ineq}.
\end{proof}

Recall that in \cite{AMN1}, one of the main results, namely Theorem~\refschrSpace{thm.main1} 
states that $\XGV$ is diffeomorphic to the closure of $\delta(X)$, where $\delta$ is the diagonal 
map from $X$ to $\prod_{Y \in \maF} \oXY$ (we assume that $0 \in \maF$). The diffeomorphism is 
the unique extension of the diagonal map $x\mapsto (x,\ldots,x)$.

Let $(Z)_r$ resp.\ $\ol{(Z)}_r$ be the open resp.\ closed ball of radius $r$ around $0$ in 
a Euclidean vector space $Z$.

\begin{lemma}\label{lemma.Z} Let $Z$ be a finite-dimensional
Euclidean space with open unit ball $(Z)_1$. There exists a true metric
$g_Z$ on $\overline{Z}$ such that on $(Z)_1$ it yields the same
distances as the Euclidean metric.
\end{lemma}

\begin{proof} 
Let us define $[0,\infty]$ as the closure of $[0,\infty)$ in $\ol\RR$. 
Then  $[0,\infty]$ inherits a smooth structure of $\ol\RR$, and is a compact 
manifold with boundary, and there is a diffeomorphism $\rho:[0,\infty]\to [0,2]$ 
with $\rho(t)=t$ for $t\leq 1$. We obtain a 
diffeomorphism $\theta:\ol Z\to \ol{(Z)}_2$ defined by:
$0\mapsto 0$, $Z\setminus\{0\}\ni z\mapsto 
\frac{\rho(\|z\|)}{\|z\|}z$, and for $\|z\|=1$: $\SS_Z\ni\RR_+ z\mapsto 2 z$. 
Then $g_Z \ede\theta^*\geucl$ is a suitable true metric.
\end{proof}

Let~$g$ be the Euclidean metric on $X = \RR^n$, which, we
recall, is a compatible metric on~$\oX$.

We are ready now to prove Proposition \ref{prop.equivalence}.

\begin{proof}[Proof of Proposition \ref{prop.equivalence}]
We have proved in Corollary \ref{cor.first.ineq2} one
of the two desired inequalities ($C \delta_{\maF} \ge \rho$)
needed to prove that $\delta_{\maF}$ and $\rho$ are equivalent. Let us
prove now the opposite inequality. To that end, we can
choose any true metric on $\XGV$ (they are all equivalent). We shall
choose then on each $\oXY$ the true metric $g_Y$ defined in Lemma \ref{lemma.Z}, and on
$\prod_{Y \in \maF} \oXY$ we shall choose the product
metric of these metrics. On $\XGV$ we shall choose
the induced Riemannian metric $\overline{g}$ provided
by the (diagonal) embedding $\XGV = \overline{\delta(X)}
\subset \prod_{Y \in \maF} \oXY$ which is again a true metric. 
For this particular choice of true metric we will show
  \begin{equation}\label{eq.goal}
    \rho(x) \ge \delta_{\maF}(x).
  \end{equation}
Note that the Riemannian distance from $x$ to $y$ in $\XGV$ is bounded from below 
by the Riemannian distance of the same points $x$ and $y$ in $\prod_{Y \in \maF} \oXY$. 
This implies
  $$\rho(x)^2\geq \sum_{Y\in \maF} \dist_{g_Y} (\pi_Y(x),0)^2\geq \dist_{g_Z} (\pi_Z(x),0)^2$$
  for any $Z\in \maF$. We also have
  $$\delta_{\maF}(x) = \min\bigl\{1,\min\nolimits_{Y\in \maF}\dist(x, Y)\bigr\}
  \leq \min\{1,\dist(x, Z)\}.$$
Thus \eqref{eq.goal} will follow once we have shown
   \begin{equation}\label{comp.eq}
    \dist_{g_Z} (\pi_Z(x),0)\geq  \min\{1,\dist(x, Z)\}
  \end{equation}
for any $Z\in \maF$. Note that $\dist(x, Z)$ coincides with the Euclidean norm of $\pi_Z(x)$. 
As the distance with respect to $g_Z$ coincides with the Euclidean distance inside the unit ball 
of $\ol{X/Z}$, see Lemma~\ref{lemma.Z} inequality \eqref{comp.eq} follows immediately. This 
completes the proof of \eqref{eq.goal} and thus of the proposition.  
\end{proof}

\section{Group actions on compactifications of vector spaces}\label{app.group.actions}
Here we include the details on how the constant vector fields on a vector
space extend to various compactifications.
In order for the paper to be self-contained, we also recall and extend here some facts about 
several compactifications of a vector space $X$. We will start by considering the spherical 
compactification. We use the definitions of $\RR^n_k$ and $\SS^n_k$ as in~\eqref{eq.notation.corner}.

\begin{lemma}\label{lemma.extension}
Let $X$ be a finite-dimensional real vector space. Let $A\in \GL(X)$ and $V\in X$. 
Then the affine map $p\mapsto L_{A,V}(p)\ede Ap+V$ extends to a diffeomorphism of~$\oX$.
\end{lemma}

\begin{proof}
Let us assume first that $V=\RR^n$. As a first, step we want to construct a smooth map 
$f_{A,V}:\SS^n_1\to \SS^n_1$ such that $\Theta_n^{-1}\circ f_{A,V}\circ \Theta_n$ is an 
extension of the affine map $L_{A,V}$, where~$\Theta_n$ is the map $\oX\to \SS^n_1$ given 
by \eqref{eq.def.theta}. As a second step we prove that it is a diffeomorphism.

At first we consider the linear map $F_{A,V}:\RR^{n+1}\to \RR^{n+1}$, $(t,p)\mapsto (t,Ap+tV)$. 
Obviously $F_{A,V}$ maps $\RR^{n+1}_1$ to itself and $F_{A,V}((1,p))= (1, L_{A,V}(p))$. It is 
easy to check that the smooth map 
\begin{equation*}
   f_{A,V}\bigl((t,p)\bigr) \ede \frac{F_{A,V}((t,p))}{\|F_{A,V}((t,p))\|}
\end{equation*}
has the extension property for $L_{A,V}$ described above. It is obvious to see that
$F_{A,V}\circ F_{A',V'} = F_{A A', V+AV'}$ and as a consequence 
$f_{A,V}\circ f_{A',V'} = f_{A A', V+AV'}$. It follows that $f_{A^{-1}, -A^{-1}V}$ is the inverse 
to $f_{A,V}$, and hence $f_{A,V}$ is a diffeomorphism. The case of a general vector space $X$ is obtained
by choosing a linear isomorphism $X \simeq \RR^n$. In view of the results proved, the 
resulting smooth structure on $\oX$ and the smoothness of the affine maps on $\oX$
do not depend on the choice of the isomorphism $X \simeq \RR^n$.
\end{proof}

Obviously the diffeomorphism $f_{A,V}:\ol X\to \ol X$ restricts to a diffeomorphism 
of the boundary, namely the sphere at infinity $f_{A,V}|_{\SS_X}\colon \SS_X\to \SS_X$.
From the proof of the last lemma, we also get:

\begin{corollary}
The extension constructed in Lemma \ref{lemma.extension}
above yields a group homomorphism $\Aff(X)\to \Diff(\oX)$, 
$L_{A,V}\mapsto f_{A,V}$. For a translation $L_{1,V}:X\to X$ by $V$, we have
$$f_{1,V}|_{\SS_X}=\id_{\SS_X}.$$
\end{corollary}
	
\begin{proof}
The proof is immediate from the formula defining the extensions of 
affine maps to the spherical compactification (see the proof of Lemma
\ref{lemma.extension}).
\end{proof}

The map in the corollary defines a Lie group action of $\Aff(X)$ on $\oX$, and the derivative 
of this Lie group action yields a Lie algebra action of $\aff(X)$, the Lie algebra of the 
affine group of $X$ on $\oX$. We may restrict to the Lie group of translations, resp. to the 
corresponding Lie algebra. This Lie algebra is $X$ with the trivial bracket. We thus have a 
Lie algebra action of $X$ on $\oX$. If a Lie algebra $\frak{g}$ acts on a manifold with corners 
$M$, then this is given by a Lie algebra homomorphism $\maL :\frak{g}\to \Gamma(TM)$.
In the above situation $X=\frak{g}$ and $\maL$ maps a vector of $X$ to the corresponding constant 
vector field in $\Gamma(X)$. As the Lie algebra $X$ acts on $\oX$, any constant vector field on~$X$
extends to a vector field on $\oX$. As the Lie group of translations acts by the identity on 
$\SS_X$, the Lie group $X$ acts trivially on~$\SS_X$ which proves that the smooth extension of 
a constant vector field to $\ol{X}$ vanishes on~$\SS_X$.
We thus have proven the following.

\begin{corollary}\label{cor.ext.zero}
  For any $V\in X$, the vector field
\begin{equation*}
   W(p) \ede 
   \begin{cases} 
   V & \text{ if }\ p\in X \\ 
    0 & \text{ if }\ p\in \SS_X\,,
  \end{cases}
\end{equation*}
is a smooth vector field on $\oX$.  
\end{corollary}

We now turn to more involved compactifications. Let $\maS$ be a finite, clean 
semilattice of closed submanifolds of $\SS_X$. We consider 
the compactification, $M\ede [\oX: \maS]$, an example being the Georgescu--Vasy 
compactification $\XGV$ discussed in~\ref{sec.two.bups}.

The action of the Lie group $X$ on $\oX$ constructed above is the identity on 
any $P \in \maS$, as $P \subset \SS_X$. Thus \citeschrSpace{Theorem}{thm.cor.main1} 
or Theorem~\ref{theorem.mainAMN1} tells us that the Lie group action of $X$ on $\oX$ 
lifts (uniquely) to an action of $X$ on $M$. The action is by diffeomorphisms, and,
as~$X$ is connected, the action of an element of~$X$ on~$M$ maps each boundary face 
of~$M$ to itself. The corresponding Lie algebra action yields a Lie algebra homomorphism 
$\maL\colon X = \Lie(X) \to \Gamma(TM)$. Its image consists 
of vector fields tangent to the boundary of $M$ as the corresponding Lie group action 
preserves all boundary faces. We have thus obtained the following result:

\begin{lemma}\label{lemma.vect.GV.extend}
Let $X$ and $\maF$ be as above and $M\ede [\oX:\maF]$. Then any constant vector field 
on $X$ extends to a smooth vector field on $M$ that is tangent to the boundary.
\end{lemma}

We will now discuss whether the smooth extension of a constant vector field to 
$M \coloneqq [\oX: \maS]$ still vanishes on the boundary.

\begin{example}
Let $X$ and $\maS = \{\SS_{Y} \}$, $Y \subset X$, $Y \neq \{0\}$. Then we have seen
that the constant vector fields $v \in X$ extend to a vector fields on $\ol{X}$ 
vanishing at the boundary $\SS_{X}$. One can show that the lift of $v$ to $
[\ol{X} : \SS_Y ]$ vanishes at all the boundary faces of 
$[\oX: \SS_Y]$ if, and only if $v \in Y$.
If $v\notin Y$, then its lift 
does not vanish on the boundary hyperface 
emerging from $\SS_Y$, but it still vanishes on the other boundary hyperface of $M$ 
(respectively, in the exceptional  case $\dim Y=\dim X -1$, on the other two boundary 
hyperfaces of $[\oX: \SS_Y]$).
\end{example}

\section{A splitting lemma}\label{app.spherical.compact}

We now include the promised details explaining how Lemma \ref{lemma.newPsiYZ}
follows from \citeschrSpace{Lemma}{blow.sphere}. In particular,
we consider the relations between (blow-ups of) the (spherical)
compactification of the vector space~$X$ and the compactification 
of linear subspaces, providing,
in particular, the full details of the proof of Lemma~\ref{lemma.newPsiYZ}.
So let~$Y$ be a linear subspace of~$X$.
We have shown in \citeschrSpace{Proposition}{prop.X/Y} that the quotient map $X\to X/Y$
extends uniquely to a smooth map  $\psi_Y:[\oX:\SS_Y]\to \ol{X/Y}$.
We will show that $\psi_Y$ is a trivial fibration, a trivialization will be provided by 
the choice of any complement; 
in other words: $[\oX:\SS_Y]$ is a product of manifolds with corners, and $\psi_Y$ 
the projection to one of its factors.

In order to work on the space $[ \oX : \mathbb{S}_X ]$, 
it is convenient to have a better understanding of its structure,
including of its smooth structure. The following result due to Melrose
(see \citeschrSpace{Lemma}{blow.sphere} and the references therein) provide the
needed initial results.

In the following lemma we write elements $\RR^{m+m'+1}_{k+k'}$ in the following way. At first, up to permutation of the coordinates $\RR^{m+m'+1}_{k+k'}$ coincides with $\RR^{m}_k\times \RR^{m'+1}_{k'}$, 
which we write as  $\RR^{m+m'+1}_{k+k'}\cong \RR^{m}_k\times \RR^{m'+1}_{k'}$. Then a vector in 
$\RR^{m+m'+1}_{k+k'}$ will be written as a pair $(\eta,\mu)$ with respect to the product 
$\RR^{m}_k\times \RR^{m'+1}_{k'}$.

\begin{lemma}
\label{lemma.blow.sphere}
Let $\SS_{k, k'}^{m, m'} := \SS^{m+m'} \cap \big ( \RR^{m}_k
  \times \RR^{m'+1}_{k'} \big ) \, \cong\, \SS^{m+m'}_{k+k'}$ . If we define
\begin{equation*}
    \Psi : \SS_{k, k'}^{m, m'} \smallsetminus \big ( \{0\} \times
    \SS^{m'}_{k'} \big ) \ \to \ \SS^{m-1}_k \times \SS^{m'+1}_{k'+1}
    \,, \quad \Psi(\eta,\mu) \seq \Big (\, \frac{\eta}{|\eta|},\, \big
    (|\eta|,\mu \big ) \, \Big ) \,,
\end{equation*}
then  $\Psi$ extends to a diffeomorphism
\begin{equation}\label{def.psi.sphere}
  \widetilde \Psi : [\SS_{k,k'}^{m,m'} : \{0\} \times
    \SS^{m'}_{k'}]\ \stackrel{\sim}{\longrightarrow}\ \SS^{m-1}_{k}
  \times \SS^{m'+1}_{k'+1}
\end{equation}
Up to a suitable permutation of coordinates, we permute the coordinates and if we identify
$\mathbb{S}^m_k$ with  $\{0\} \times \mathbb{S}^m_k$ as subset of $\mathbb{S}^{m+m'}_{k+k'}$,
we obtain a diffeomorphism :
\begin{equation*}
  \widetilde \Psi : [\SS_{k+k'}^{m+m'} : \{0\} \times
    \SS^{m'}_{k'}]\ \stackrel{\sim}{\longrightarrow}\ \SS^{m-1}_{k}
  \times \SS^{m'+1}_{k'+1}
\end{equation*}
\end{lemma}

A proof of this result with our notations can be found in \citeschrSpace{Lemma}{blow.sphere} 
In the following proposition, we shall identify $X/Y$ with $Y^\perp$, for
some fixed scalar product $\<\twoargu\>$ on~$X$.

\begin{proposition}\label{prop.details.newPsiYZ}
Let~$X$ be a real finite vector space and $Y \neq 0$ be a linear subspace of~$X$. 
Then there is a diffeomorphism
\begin{equation*}
  \Psi_Y \seq (\psi_Y, \chi)\colon
   [ \overline{X} : \SS_Y ] \to  \oXY \times \oY
\end{equation*}
such that the restriction of $\psi_Y$ to $\oX\setminus \SS_Y$ is the extension 
of the canonical projection $\pi_Y\colon X\to X/Y$ and $\chi(y)=y$ if $y \in Y$.
Moreover, the blow-down map $\beta_{\oX, \SS_Y}$  is determined by $\Xi_Y \coloneqq 
\beta_{\oX, \SS_Y} \circ \Psi_Y^{-1} : \oXY \times \oY \to \oX$ which is given by:
\begin{equation*}
 \Xi_Y(z,y) \seq
 \begin{cases} 
    \ y & \ \mbox{ if } y \in \SS_Y\,, \\
    \ z +  \sqrt{1 + \<z,z\>}\, y & 
    \ \mbox{ if } (z,y) \in X/Y \times Y \,.
  \end{cases}
\end{equation*}
\end{proposition}

\begin{proof}
The proof is a modification of the proof of \cite[Proposition~5.2]{AMN1}. 
The map $\psi_Y$ is the same as the one defined in the aforementioned Proposition
and used before in this paper. Let us recall briefly the construction of this 
map $\psi_Y$ and then introduce the map~$\chi$. Let $\dim Y=p$ and $\dim X=n$. Without 
loss of generality, we may assume that $Y=\RR^p\subset X=\RR^n$. Then 
$\Theta_p(\SS_Y) = \partial \SS^p_1 \cong \SS^{p-1}$ may be viewed  as subset 
of $\Theta_n(\SS_X) = \partial \SS^n_1 \cong \SS^{n-1}$.
Using this identification, we can obtain
\begin{equation*}
     [\overline{X} : \SS_Y] \simeq^{\Theta_n^\beta} [\SS^n_1 : \SS^{p-1}] \, ,
\end{equation*}
where $\Theta^\beta_n$ is the lift of $\Theta_n$ to the two blow-up spaces.
Using Lemma~\ref{lemma.blow.sphere} for $m=n-p+1$ and $m'=p$, we have the diffeomorphism
$[\SS^n_1 : \SS^{p-1}] \simeq \SS^{n-p}_1 \times \SS^p_1$. We go back to the compactified 
vector spaces using  $(\Theta_{n-p}^{-1},\Theta_p^{-1}) : \SS^{n-p}_1 \times \SS^p_1 \to 
\oXY \times \oY$. Composing these diffeomorphisms, we obtain:
\begin{equation*}
    [\oX : \SS_Y] \simeq^{\Theta^\beta_n} [\SS^n_1 : \SS^{p-1}] 
    \simeq^{\Psi} \SS^{n-p}_1 \times \SS^p_1
\simeq^{(\Theta_{n-p}^{-1},\Theta_p^{-1})} \oXY \times \oY \, .
\end{equation*}
At this point, the reader might want to check the definition of $\Theta_n$ and $\Theta_n^{-1}$
defined in \eqref{eq.def.theta} and \eqref{eq.def.theta.inv}, as well as $\Psi$ defined in \eqref{def.psi.sphere}.
If $v=(v_Y,v_\perp) \in Y \oplus Y^\perp$, we obtain
\begin{equation}
  \widetilde \Psi \circ \Theta_n (v) \,=\, \Psi \Big (\,
  \frac{1}{\<v\>} (1,v)\, \Big) \,=\, \Big (\, \frac{1}{
    \<v_\perp\>}\, (1,v_\perp),\ \frac{1}{\<v\>}\, (
  \<v_\perp\> , v_Y)\, \Big )\,.
\end{equation}
For the first component, we obtain
$\Theta_{n-p}^{-1}\left(\frac{1}{ \<v_\perp\>}(1, v_\perp)\right)=v_\perp=\pi_Y(v)$. 
Hence we have the required property
for $\Psi_Y$.
In the second component, if $v \in Y$, that is $v_\perp=0$, we obtain 
$ \<v_\perp\>=1$ and $ \<v\>= \<v_Y\>$ then
\begin{equation*}
   \Theta_p^{-1}\left(\frac{1}{\<v_Y\>}(1,v_Y)\right)=v_Y \,.
\end{equation*}
Hence we have $\chi(y)=y$ for $y \in Y$.
A long computation gives the required propriety for $\Xi_Y=\beta_{\oX,\SS_Y} \circ \Psi_Y^{-1}$. 
\end{proof}

Note that in contrast to the map $\psi_Y$, the constructions of the maps $\chi$ and $\Xi_Y$ 
depend on more than the space $X$ and its subspace $Y$, it also depends on the scalar product 
(and thus indirectly on the choice of the orthogonal complement). Thus if $A:(X,Y)\to (X',Y')$ 
is an automorphism of pairs of vector space, and if we denote maps induced by $A$ also by~$A$, 
then we have the naturality relation $\pi_{Y'}\circ A= A\circ \pi_Y$, but in general  $\chi$ 
and $\Xi_Y$ are not natural, \ie $\chi\circ A\neq A \circ \chi$ and  $\Xi_{Y'}\circ A
\neq A\circ \Xi_Y$.

\ifarxivversion
\section{More about clean intersections}\label{app.clean.intersections}
Here we include some technical details for the
fact that our semilattices coming from linear subspaces are clean.

\begin{lemma}\label{lem.clean.spec}
Let $Y,Z$ be two linear subspaces of $X$. Then each of the following pairs 
intersects cleanly:
\begin{enumerate}
\item $\overline{Y}$ and $\overline{Z}$,
\item $\SS_Y$ and $\SS_Z$,
\item $\SS_Y$ and $\overline{Z}$.
\end{enumerate}
\end{lemma}

\begin{proof}
The clean intersection property for $\overline{Y}$ and $\overline{Z}$ is obvious in 
$X$, so study this property in $p\in \SS_Y=\overline{Y}\cap \SS$. Let 
$\dim (X) = n$, $\dim (Y) = k$, $\dim (Z) = \ell$, $\dim (Y\cap Z) = m$. We may choose 
a basis $e_1,\ldots,e_n$ of $X$ such that $e_1,\ldots,e_m$ is a basis of $Y\cap Z$, such 
that $e_1,\ldots,e_k$ is a basis of $Y$, and such that $e_1, \ldots , e_m , e_{k+1}, \ldots ,
 e_{\ell+k-m}$ is a basis of $Z$. It remains to check the clean intersection of the above 
pairs in any given $p\in\SS_Y\cap\SS_Z$, and without loss of generality we 
can assume that~$p$ is the ray through~$e_1$.
  As $\Theta_n$ is a diffeomorphism from $\oX$ to $\SS^n_1\subset \RR^{n+1}$, we will argue in this submanifold. To simplify the indices, we assume that the basis of $\RR^{n+1}$ is $e_0,\ldots,e_n$. We calculate
$T_{\Phi_n(p)}\overline{Y}= \spann\{e_0,e_2,e_3,\ldots e_k\}$,  $T_{\Phi_n(p)}(\overline{Y}\cap \overline{Z}) =  \spann\{e_0,e_2,e_3,\ldots e_m\}$,  $T_{\Phi_n(p)}\overline{Z} =  T_{\Phi_n(p)}(\overline{Y}\cap \overline{Z}) \oplus  \spann\{e_{k+1},\ldots,e_{\ell+k-m}\}$, $T_{\Phi_n(p)}\SS_Y=  \spann\{e_2,e_3,\ldots e_k\}$, $T_{\Phi_n(p)}(\SS_y\cap \SS_Z)=  \spann\{e_2,e_3,\ldots e_m\}$, and finally we have  $T_{\Phi_n(p)}\SS_Z= T_{\Phi_n(p)}(\SS_y\cap \SS_Z)\oplus  \spann\{e_{k+1},\ldots,e_{\ell+k-m}\}$. 
From these calculations the claimed statement follows.
\end{proof}

\begin{corollary}\label{cor.ess.lem.clean}
 Let $V,W$ be two linear subspaces of $X$. Then all subsets of $\{\overline V,\overline W, \SS_V,\SS_W\}$ intersect cleanly.
\end{corollary}
  
\begin{proof}For subsets with at most one element the statement is trivial. For subsets with two elements, the corollary follows from Lemma~\ref{lem.clean.spec}
  by setting $(Y,Z)=(V,W)$ or $(Y,Z)=(W,V)$.

  In order two prove it for subset with 3 elements, let us consider first the case $\{\overline V,\overline W, \SS_V\}$.
Applying Lemma~\ref{lem.clean.spec} (3) for $Z=V$ and $Y=V\cap W$, we see that 
$\overline V$ and $\SS_{V\cap W}=\overline W\cap \SS_V$ intersect cleanly.  We already know that $\overline W$ and $\SS_V$ intersect cleanly, thus the triple intersects cleanly.  In the case $\{\overline W, \SS_V,\SS_W\}$ we apply Lemma~\ref{lem.clean.spec} (2) for $Z=W$ and $Y=V\cap W$, and we get that $\SS_W$ and $\SS_{V\cap W}=\overline W\cap \SS_V$ intersect cleanly, the rest of the proof is as above. The other cases with three elements are obtained by switching the role of $V$ and $W$.

The set $\{\overline V,\overline W, \SS_V,\SS_W\}$ intersects cleanly by applying the Lemma~\ref{lem.clean.spec} (3) for $Z=V$ and $Y=V\cap W$, which yields the clean intersection of $\overline V$ and $\SS_{V\cap W}= \overline{W}\cap \SS_V\cap \SS_W$. Using the result on the triples above, the statement follows.
\end{proof}

We end by citing a proposition which shows that cleanly intersecting pairs of submanifolds may be linearized by a chart.

\begin{proposition}[{\cite[Prop.~C.3.1]{hormander3}}]\label{hor3.prop}
  Let $P_1$ and $P_2$ be submanifolds of an $n$-dimensional manifold $M$ (without boundary and corners). Then $P_1$ and $P_2$ intersect cleanly if and only if for every $x\in P_1\cap P_2$, there is a chart $\phi:U\to V$ of $M$, $x\in U$ that maps each $P_i\cap U$ to an open subset of a linear subspace $L_i$ of $\RR^n$. \end{proposition}
\fi

\ifarxivversion
\pagebreak[4]
\section{Notation}

\subsection*{Fixed notation}
\begin{itemize}
\item $\NN = \{1, 2, \ldots \}$, $\ZZ_+ = \NN \cup \{0\}$. 

\item $X$ is a fixed finite-dimensional, Euclidean real vector space.

\item $\maF$ is a finite-dimensional semilattice of linear subspaces of $X$
such that $X \notin \maF$ but $\{0\} \in \maF$.

\item $\dist$ is the Euclidean distance on $X$ or on some other
Euclidean vector space. 

\item $d_Y(x) = \dist(x, Y)$ for any linear subspace $Y$ of $X$. 

\item $\dist_{\ol{g}}(x, P)$ is the distance from $x \in M$ to some subset
$P \subset M$ in a \emph{true metric}~$\ol{g}$ on $M$, that is, one that is
defined up to and including the boundary of the manifold with corners $M$.

\item 
   $\delta_{\maF}(x) \ede \min_{Y \in \maF} \{ d_Y(x), 1\} \seq 
   \min\{\dist(x, \unionF), 1\}$

\item 
    $\overline{\maF} \ede \{\, \oY \, | \ 
    Y \in \maF \, \}$,  

\item  $\SS_{\maF} \ede \{\, \mathbb{S}_Y \, | \ Y \in \maF \, \}$. 

\item 
    $\XGV \ede [\oX: \SS_\maF]$ is the Georgescu--Vasy space.

\item $\betaGV : \XGV \to \oX$ is the blow-down map.
    
\item $X_{\maF} \ede  [\oX : \SS_\maF \cup \overline{\maF}].$

\item $\rho(x) \coloneqq  \dist_{\overline{g}}(x, \unionF)$, where $\overline g$ is
a (true) metric on $\XGV$.

\item 
    $\RR_k^n \ede [0, \infty)^k \times \RR^{n-k}$
    
\item $\SS_k^{n-1} \ede \SS^{n-1} \cap \RR_k^n$.
\item $\Theta_n : \oX \to \mathbb{S}^n_1$ is the bijection used to defined the smooth structure on the spherical compactification of $X$.

Let $P \subset M$ be a \psubmanifold{} of a manifold with
corners~$M$.

\item $N^MP \coloneqq  TM\vert_P/TP$ (the normal
  bundle of $P$ in $M$). 
\item $N^M_+P$ is the image $T^+M\vert_{P}$ in
$N^M P$ (the inward pointing normal fiber bundle of $P$ in
  $M$). 
  
\item $\SS(N^M_+P)=$ the set of rays starting at $0$ in $N^M_+P$ (the 
inward pointing spherical normal bundle of $P$ in $M$). 
\item If $\maP = (P_1, P_2, \ldots, P_k)$ is a $k$-tuple of
closed subset of $M$ with $P_1$ a \psbmanifold{} of $M$, then
$\maP' \coloneqq \beta^*(\maP \smallsetminus \{P_1\} ) 
    \coloneqq \bigl(\beta^*(P_i)\bigr)_{i=2}^k$ is the
    \emph{pull-back} of $\maP$.
    
\item $\psi_Y : [\oX: \SS_Y] \to \oXY$ is the unique smooth
  extension of the projection $X \to X/Y$.
\item $\doublesharp{B}$: If $B$ is a non-degenerate symmetric bilinear form on a finite-dimensional vector space $V$, then $\doublesharp{B}$ is the associated non-degenerate symmetric bilinear form on the dual $V^*$.
\end{itemize}

\subsection*{Generic notation}

\begin{itemize}
\item $M$ is a manifold with corners.

\item $[M: P]$ is the blow-up of $M$ with respect to a closed
\psbmanifold{} $P \subset M$.

\item $Y$ is a general finite-dimensional real vector space, usually
a subspace of $X$.

\item $\SS_Y$ is the sphere at infinity of $Y$.

\item $\oY = Y \cup \SS_Y$ is the spherical compactification of $Y$.

\item $|A|$ is the set of elements of a set $A$.

\item $x_H$ is a defining function of the hyperface $H$ of some
manifold with corners.

\item $\beta$, usually decorated with various indices, denotes some
blow-down map, such as $\beta_{M, P} : [M: P] \to M$.
\end{itemize}
\fi

\bernd{Number of appendices: \arabic{section}}

\addtostream{baaux}{\protect\newcommand{\protect\arabicnumberappendices}{\arabic{section}}}
\closeoutputstream{baaux}


\bibliographystyle{plain}
\bibliography{schrReg}


\end{document}